\documentclass[twoside,a4paper,11pt]{article}
\usepackage{amsfonts, amsbsy, amsmath, amsthm, amssymb, latexsym,rotating,pdflscape}
\usepackage{mathrsfs}
\usepackage{enumerate,centernot}
\usepackage[top=30mm,right=30mm,bottom=30mm,left=30mm]{geometry}

\usepackage{bm}
\usepackage{fancyhdr}

\numberwithin{equation}{section}

\renewcommand{\bf}{\textbf}

\newcommand{\leqs}{\leqslant}

\newtheorem{theorem}{Theorem}

\newtheorem{propn}[theorem]{Proposition}

\newtheorem{thm}{Theorem}[section]

\newtheorem{prop}[thm]{Proposition}
\newtheorem{lem}[thm]{Lemma}
\newtheorem{corol}[thm]{Corollary}

\newtheorem{cor}[theorem]{Corollary}
\newtheorem{con}[theorem]{Conjecture}

\theoremstyle{definition}

\begin{document}

\title{Varieties of elements of given order in simple algebraic groups}

 \author{Claude Marion\\ \\ \small{\begin{tabular}{c}Dipartimento di Matematica, Universit\`{a} degli Studi di Padova,  Padova, Italy\\ marion@math.unipd.it\end{tabular}}}
\date{}
\maketitle

\noindent\textbf{Abstract.}  
Given a positive integer $u$ and a simple algebraic group $G$ defined over an algebraically closed field $K$ of characteristic $p$, we  derive properties about  the subvariety $G_{[u]}$ of $G$ consisting of elements of $G$ of order dividing $u$. In particular, we determine the dimension of $G_{[u]}$, completing results of Lawther \cite{Lawther} in the special case where $G$ is of adjoint type.   We also apply our results to the study of finite simple quotients of triangle groups, giving further insight on a conjecture we proposed in \cite{Marionconj} as well as proving that some finite quasisimple groups are not quotients of certain triangle groups. 

\tableofcontents

\section{Introduction}\label{s:intro}
Let $G$ be a  reductive algebraic group defined over an algebraically closed field $K$ of  characteristic $p$ (possibly equal to 0), $C$ be a conjugacy class of $G$ and $u$ be a positive integer.  In 2007 Guralnick \cite{Guralnick} proved the following result:

\begin{theorem}\cite[Theorem 1.1]{Guralnick}.
Given a reductive algebraic group $G$, a conjugacy class $C$ of $G$ and a positive integer $u$, the set $\{g \in G: g^u\in C\}$   is a finite union of conjugacy classes of $G$. 
\end{theorem}

In this paper, we concentrate our attention to the case where $G$ is connected and $C=\{1\}$ is the trivial conjugacy class of $G$.   For  a positive integer $u$, we let $$G_{[u]}=\{g \in G: g^u=1\}$$ be the subvariety of $G$ consisting of elements of $G$ of order dividing  $u$ and set $j_u(G)= \dim G_{[u]}$.  We also let  $d_u(G)$ be the minimal dimension of a centralizer in $G$ of an element of $G$ of order dividing $u$.  We are merely interested in determining $j_u(G)$ for every positive positive integer $u$ (when $G$ is a simple algebraic group).  \\

For completeness, we begin by proving Guralnick's result in the case where $G$ is connected and $C=\{1\}$.  In the statement below, given $g \in G$, we let $g^G$ denote the conjugacy class of $g$ in $G$.

\begin{propn}\label{p:fcu}
Let $G$ be a connected reductive  algebraic group  defined  over an algebraically closed field $K$ of characteristic $p$. Let $u$ be a positive integer. Then the number of conjugacy classes of $G$ of elements  of order dividing $u$ is finite. In particular $G_{[u]}$ is a finite union of conjugacy classes of $G$. Moreover $\dim G_{[u]}=\max_{g \in G_{[u]}} \dim g^G$ and ${\rm codim}\ G_{[u]} = d_u(G)$. 
\end{propn}

Given a simple algebraic group $G$  defined over an algebraically closed field $K$ of characteristic $p$, we denote  by $G_{s.c.}$ (respectively, $G_{a.}$)  the simple algebraic group over $K$ of simply connected (respectively, adjoint) type having the same Lie type and Lie rank as  $G$.  We prove the following result partially proved in \cite[Theorem 3.11]{Lawther}:

 \begin{propn}\label{p:ingsimple}
Let  $G$ be simple algebraic group defined over an algebraically closed field $K$ of characteristic $p$.   Let $u$ be a positive integer.  Then
  $$ d_u(G_{a.})\leq d_u(G) \leq d_u(G_{s.c.}).$$
    \end{propn} 

In \cite{Lawther}, given a positive integer $u$ and a simple algebraic group $G$ defined over an algebraically closed field $K$ of characteristic $p$, Lawther gives a lower bound for  $d_u(G)$. Moreover, he proves that this lower bound is equal to $d_u(G)$ in the case where $G$ is of adjoint type. We establish the following result completing Lawther's result. Recall that the Coxeter number of a simple algebraic group is defined to be the quotient of the number of roots  in the root system associated to $G$ by the number of such roots which are simple.  Concretely,  $h=\ell+1$,  $2\ell$, $2\ell$, $2\ell-2$, $6$, $12$, $12$, $18$ or $30$ according respectively as $G$ is of type $A_\ell$, $B_\ell$, $C_\ell$, $D_\ell$, $G_2$, $F_4$, $E_6$. $E_7$ or $E_8$.

\begin{theorem}\label{t:blawther}
Let $G$ be a simple algebraic group  defined over an
algebraically closed field $K$ of characteristic $p$.  Let $h$ be the
Coxeter number of $G$ and $u$ be a positive integer. Write $h=zu+e$ and  $u=qv$, where $z$ and
$e$ are nonnegative integers with $0 \leq e \leq u-1$ and $q$, $v$ are
coprime positive integers such that $q$ is a power of $p$ and $p$ does not divide $v$.
  The following assertions hold.
\begin{enumerate}[(i)]
\item If $G$ is of type $A_\ell$ then $$d_u(G_{s.c.})=d_u(G_{a.})$$
except  if $p\neq 2$, $u$ is even, $h=zu$ and $z$ is odd in
which case $$d_u(G_{s.c.})-d_u(G_{a.})=2.$$
\item If $G$ is of type $B_\ell$ then  $$d_u(G_{s.c.})=d_u(G_{a.})$$
except  if
\begin{enumerate}
\item $p\neq 2$,  $u\equiv 2 \mod 4$, $h=zu$, and $z\equiv u/2 \mod
4$ or $z\equiv 2 \mod 4$ in which case
$$d_u(G_{s.c.})-d_u(G_{a.})=2.$$
\item $p\neq 2$, $u \equiv 4 \mod 8$, $h=zu$ and $z$ is odd in which
case  $$d_u(G_{s.c.})-d_u(G_{a.})= 2.$$
\end{enumerate}
\item If $G$ is of type $C_\ell$ then $$d_u(G_{s.c.})=d_u(G_{a.})$$
except  if $p\neq 2$ and  $u$ is even in which case
$$d_u(G_{s.c.})-d_u(G_{a.})=2\left\lceil \frac z2\right\rceil.$$
\item  If $G$ is of type $D_\ell$ then  $$d_u(G_{s.c.})=d_u(G_{a.})$$
except  if
\begin{enumerate}
\item  $p\neq 2$, $u=2$, $h=zu$ and $z \equiv u/2 \mod 4$ in which
case  $$d_u(G_{s.c.})-d_u(G_{a.})=4.$$
\item   $p\neq 2$, $u \equiv 2 \mod 4$, $u>2$, $h=zu$ and $z \equiv
u/2 \mod 4$ in which case  $$d_u(G_{s.c.})-d_u(G_{a.})=2.$$
\item $p\neq 2$, $u\equiv 2 \mod 4$, $e=u-2\neq 0$ and
$z \equiv 1 \mod 4$  in which case  $$d_u(G_{s.c.})-d_u(G_{a.})= 2.$$
\item $p\neq 2$, $u \equiv 4 \mod 8$, $z$ is odd and  $h=zu$ 
 in which case $$d_u(G_{s.c.})-d_u(G_{a.})=2.$$
\end{enumerate}
\item If $G$ is of exceptional type then $d_u(G_{s.c.})=d_u(G_{a.})$ except if $G$ is of type $E_7$, $p \neq 2$ and $u\in \{2,6,10,14,18\}$ in which case
$$d_2(G_{s.c.})-d_2(G_{a.})=6 \quad \textrm{and} \quad d_u(G_{s.c.})-d_u(G_{a.})=2 \ \textrm{for} \ u \in \{6,10,14,18\}.$$
\end{enumerate}
\end{theorem}

We note that for  $G$  of exceptional type the only cases where the simply connected and the adjoint groups are not abstractly isomorphic is when $G$  is of type $E_6$ or $E_7$ and $p\neq 3$ or $p\neq 2$ respectively.\\

Recall that if a simple algebraic group $G$ over an algebraically closed field $K$ of characteristic $p$  is not of simply connected type  nor of  adjoint type then either $G$ is abstractly isomorphic to ${\rm SL}_n(K)/C$ where $C\leqs Z({\rm SL}_n(K))$, or $G$ is of type $D_{\ell}$, $p \neq 2$ and $G$ is abstractly isomorphic to ${\rm SO}_{2\ell}(K)$ or a half-spin group ${\rm HSpin}_{2\ell}(K)$ where $\ell$ is even in the latter case. For simple algebraic groups that are neither of simply connected nor adjoint type we obtain the following result. 

 \begin{theorem}\label{t:duotsca}
   Let $G$ be a simple algebraic group, neither of simply connected nor adjoint type, defined over an algebraically closed field $K$ of characteristic $p$. Let $u$ be a positive integer. Write $u=qv$ where $q$ and $v$ are coprime positive integers such that $q$ is a power of $p$  and $p$ does not divide $v$.   Then the following assertions hold. 
   \begin{enumerate}[(i)]
   \item If $d_u(G_{s.c.})=d_u(G_{a.})$ then $d_u(G)=d_u(G_{a.})$
   \item  Suppose $d_u(G_{s.c.})\neq d_u(G_{a.})$. 
   \begin{enumerate}
   \item If $G={\rm SL}_{\ell+1}(K)/C$ where $C\leq Z({\rm SL}_{\ell+1}(K))$, then $d_u(G)=d_u(G_{a.})$ if $C$ has an element of order $2$, otherwise $d_u(G)=d_u(G_{{s.c.}})$. 
   \item If $G={\rm SO}_{2\ell}(K)$ where $p\neq 2$, then $d_u(G)=d_u(G_{a.})$.
   \item If $G={\rm HSpin}_{2\ell}(K)$ where $p\neq 2$ and $\ell$ is even, then $d_u(G)=d_u(G_{s.c.})$. 
   \end{enumerate}
   \end{enumerate}
   \end{theorem}

In \cite{Lawther} Lawther showed that if $G$ is a simple algebraic group of rank $\ell$ with Coxeter number $h$ over an algebraically closed field and $u$ is a positive integer with $u\geq h$ then  $d_u(G)=\ell$. As a corollary of Theorems \ref{t:blawther} and  \ref{t:duotsca} we obtain the following result:

  \begin{cor}
Let $G$ be a simple algebraic group   of rank $\ell$ over an algebraically closed field $K$ of  characteristic $p$. Let $h$ be the Coxeter number of $G$ and let $u$ be a positive number.  Then
\begin{enumerate}[(i)] 
\item If $u>h$ then $d_u(G)=\ell$.
\item Suppose $u=h$ then $d_u(G)=\ell$ if  $G$ is (abstractly isomorphic to a group) of adjoint type, or $p=2$, or $h$ is odd (and then $G$ is of type $A$), or $G={\rm SL}_{\ell+1}(K)/C$ and $C$ has an element of order $2$ where $C\leq Z({\rm SL}_{\ell+1}(K))$,   or $G\cong {\rm SO}_{2\ell}(K)$ and $p\neq 2$, or $G$ is of type $B$ or $D$ and $h\equiv 0\mod 8$, or $G=E_6(K)$, otherwise  $d_u(G)=\ell+2$. 
\end{enumerate}
\end{cor}

Given a simple algebraic group $G$ over an algebraically closed field, we also derive a further property of $$d_u(G): \begin{array}{l}
\mathbb{N}\rightarrow \mathbb{N}\\
u \mapsto d_u(G).
\end{array}
$$  

\begin{propn}\label{p:dudecreasing}
Let $G$ be a simple algebraic group  over an algebraically closed field $K$ of  characteristic $p$. 
Then
$$d_u(G): \begin{array}{l}
\mathbb{N}\rightarrow \mathbb{N}\\
u \mapsto d_u(G).
\end{array}
$$  
 is a decreasing function of $u$.
\end{propn}

We now apply our results to the study of finite quasisimple images of triangle groups. Recall that a finite group is quasisimple if it is perfect and the quotient by its centre is simple. \\
Every finite quasisimple group being 2-generated, given a finite  quasisimple  group $G_0$ and a triple $(a,b,c)$ of positive integers, a natural question to consider is whether $G_0$ can be generated by two elements of  orders dividing respectively $a$ and $b$ and  having product of order dividing $c$.  A finite group generated by two such elements is called an $(a,b,c)$-group. Equivalently, an $(a,b,c)$-group is a finite quotient of the triangle group $T=T_{a,b,c}$ with presentation
$$ T=T_{a,b,c}=\langle x,y,z:x^a=y^b=z^c=xyz=1\rangle. $$
When investigating the finite (nonabelian) quasisimple quotients of $T$, we can  assume that $1/a+1/b+1/c<1$ as otherwise $T$ is either soluble or $T\cong T_{2,3,5}\cong{\rm Alt}_5$ (see \cite{Conder}). The group $T$ is then a hyperbolic triangle group.  Without loss of generality, we will further assume that $a \leq b \leq c$ and call $(a,b,c)$ a hyperbolic triple of integers. (Indeed, $T_{a,b,c}\cong T_{a',b',c'}$ for any permutation $(a',b',c')$ of $(a,b,c)$.)\\

Recall that a finite quasisimple group $G_0$ of Lie type  occurs as the derived subgroup of the  fixed point group  of a simple algebraic group $G$, when ${\rm char}(K)=p$ is  prime, under a Steinberg endomorphism  $F$, i.e. $G_0=(G^F)'$.    We use the standard notation $G_0=(G^F)'= G(q)$ where $q=p^r$ for some positive integer $r$. (We include the possibility that $G(q)$ is of twisted type.)\\
 Given a simple algebraic group $G$ and a triple $(a,b,c)$ of positive integers, we say that $(a,b,c)$ is rigid for $G$ if the sum $j_a(G)+j_b(G)+j_c(G)$ is equal to $2\dim G$. When the  latter sum is less (respectively, greater) than $2\dim G$, we say that $(a,b,c)$ is reducible (respectively, nonrigid)
 for $G$. \\
In \cite{Marionconj} we  conjectured the following finiteness result in the rigid case  (first formulated for hyperbolic triples $(a,b,c)$ of primes):
\begin{con}\label{c:marionconj}
Let $G$ be a simple algebraic group defined over an algebraically closed field $K$ of prime characteristic $p$ and let $(a,b,c)$ be a rigid triple of integers for $G$.  Then there are only finitely many  quasisimple groups $G(p^r)$ (of the form $(G^F)'$ where $F$ is a Steinberg endomorphism) that are $(a,b,c)$-generated.  
\end{con}

In \cite[Theorem 1.7]{LLM1}, Larsen, Lubotzky and the author proved that Conjecture \ref{c:marionconj} holds except possibly if $p$ divides $abcd$ where $d$ is the determinant of the Cartan matrix of $G$. 
In a future paper \cite{JLM} we will make further progress on Conjecture \ref{c:marionconj} by, in particular, completely settling it in the case where the finite group $(G^F)'$ is simple.  A crucial ingredient therein is the classification of the hyperbolic triples $(a,b,c)$ of integers as reducible, rigid or nonrigid for a given simple algebraic group $G$. We determine this classification below. 

   \begin{theorem}\label{t:classification}
  The following assertions hold.
  \begin{enumerate}[(i)]
  \item The reducible hyperbolic triples $(a,b,c)$ of integers for simple algebraic groups $G$ of simply connected or adjoint type are exactly those given in Table \ref{ta:red}. In particular there are no reducible hyperbolic triples of integers for simple algebraic groups of adjoint type.
  \item The rigid hyperbolic triples of integers $(a,b,c)$ for simple algebraic  groups $G$ of simply connected type are exactly those given in Table \ref{ta:rigidsc}.
    \item The rigid hyperbolic triples of integers $(a,b,c)$ for simple algebraic groups $G$ of adjoint type are exactly those given in Table \ref{ta:rigida}.
  \item The classification of the reducible and the rigid hyperbolic triples of integers for ${\rm SO}_n(K)$ is the same as for ${\rm PSO}_n(K)$.
  \item The classification of the reducible and the rigid hyperbolic triples of integers for ${\rm HSpin}_{2\ell}(K)$ is the same as for ${\rm Spin}_{2\ell}(K)$. 
  \item If $C\leqs Z({\rm SL}_n(K))$ contains an involution then the classification of  the reducible and the rigid hyperbolic triples of integers for ${\rm SL}_n(K)/C$ is the same as for ${\rm PSL}_n(K)$. Otherwise, the classification of the reducible and the rigid hyperbolic triples of integers for ${\rm SL}_n(K)/C$ is the same as for   ${\rm SL}_n(K)$.
  \end{enumerate}
  \end{theorem}

  \begin{table}[h!]
\center{
\begin{tabular}{|l|l|l|}
\hline
$G$ & $p$ & $(a,b,c)$\\
\hline
${\rm SL}_2(K)$ & $p \neq 2$  & $(2,b,c)$\\
\hline
${\rm Sp}_4(K)$ & $p \neq 2$  & $(2,3,c)$, $(2,4,c)$, $(3,3,4)$, $(3,4,4)$, $(4,4,4)$\\
${\rm Sp}_6(K)$ & $p \neq 2$  & $(2,3,c)$, $(2,4,c)$, $(2,5,5)$, $(2,5,6)$, $(2,6,6)$\\
${\rm Sp}_8(K)$ & $p \neq 2$  & $(2,3,7)$, $(2,3,8)$, $(2,4,5)$, $(2,4,6)$\\
${\rm Sp}_{10}(K)$ & $p \neq 2$  & $(2,3,7)$, $(2,3,8)$, $(2,3,9)$, $(2,3,10)$, $(2,4,5)$, $(2,4,6)$\\
${\rm Sp}_{12}(K)$ & $p \neq 2$  & $(2,3,7)$, $(2,3,8)$, $(2,4,5)$\\
${\rm Sp}_{14}(K)$ & $p \neq 2$  & $(2,3,7)$, $(2,3,8)$, $(2,4,5)$\\
${\rm Sp}_{16}(K)$ & $p \neq 2$  & $(2,3,7)$\\
${\rm Sp}_{18}(K)$ & $p \neq 2$  & $(2,3,7)$\\
${\rm Sp}_{22}(K)$ & $p \neq 2$  & $(2,3,7)$\\
\hline
\end{tabular}
}
\caption{Reducible triples for simple algebraic groups of simply connected or adjoint type}\label{ta:red}
\end{table}

\begin{table}[h!]
   \center{
\begin{tabular}{|l|l|l|}
\hline
$G$ & $p$ & $(a,b,c)$\\
\hline
${\rm SL}_2(K)$ & $p =2$  & $(a,b,c)$\\
& $p \neq 2$ & $(a,b,c)$ $a\geq 3$\\
${\rm SL}_3(K)$ & any  & $(2,b,c)$\\
${\rm SL}_4(K)$ & $p =2$  & $(2,3,c)$\\
& $p\neq 2$ & $(2,3,c)$, $(2,4,c)$, $(3,3,4)$, $(3,4,4)$,$(4,4,4)$\\
${\rm SL}_5(K)$ & any  & $(2,3,c)$\\
${\rm SL}_6(K)$ & $p\neq 2$  & $(2,3,c)$, $(2,4,5)$, $(2,4,6)$\\
${\rm SL}_{10}(K)$ & $p\neq 2$  & $(2,3,7)$\\
\hline
${\rm Sp}_4(K)$ & $p = 2$  & $(2,3,c)$, $(3,3,c)$\\
${\rm Sp}_4(K)$ & $p \neq 2$  & $(2,b,c)$ $b \geq 5$, $(3,3,c)$ $c \geq 5$, $(3,4,c)$ $c\geq 5$, $(4,4,c)$ $c \geq 5$\\ 
${\rm Sp}_6(K)$ & $p \neq 2$  & $(2,5,c)$ $c \geq 7$,  $(2,6,c)$ $c\geq 7$, $(3,3,4)$, $(3,4,4)$, $(4,4,4)$\\
${\rm Sp}_8(K)$ & $p \neq 2$  & $(2,3,c)$ $c\geq 9$, $(2,4,7)$, $(2,4,8)$, $(2,5,5)$, $(2,5,6)$, $(2,6,6)$\\
${\rm Sp}_{10}(K)$ & $p \neq 2$  & $(2,3,c)$ $c\geq 11$,  $(2,4,7)$, $(2,4,8)$\\
${\rm Sp}_{12}(K)$ & $p \neq 2$  & $(2,3,9)$, $(2,3,10)$, $(2,4,6)$\\
${\rm Sp}_{14}(K)$ & $p \neq 2$  & $(2,3,9)$, $(2,3,10)$, $(2,4,6)$\\
${\rm Sp}_{16}(K)$ & $p \neq 2$  & $(2,3,8)$, $(2,4,5)$\\
${\rm Sp}_{18}(K)$ & $p \neq 2$  & $(2,3,8)$, $(2,4,5)$\\
${\rm Sp}_{20}(K)$ & $p \neq 2$  & $(2,3,7)$\\
${\rm Sp}_{24}(K)$ & $p \neq 2$  & $(2,3,7)$\\
${\rm Sp}_{26}(K)$ & $p \neq 2$  & $(2,3,7)$\\
\hline
${\rm Spin}_{11}(K)$ & $p \neq 2$  & $(2,3,7)$\\
${\rm Spin}_{12}(K)$ & $p \neq 2$  & $(2,3,7)$\\
\hline
\end{tabular}
}
\caption{Rigid triples for simple algebraic groups  of simply connected type}\label{ta:rigidsc}
\end{table}

  \begin{table}[h!]
\center{
\begin{tabular}{|l|l|l|}
\hline
$G$ & $p$ & $(a,b,c)$\\
\hline
${\rm PSL}_2(K)$ & any & $(a,b,c)$\\
${\rm PSL}_3(K)$ & any  & $(2,b,c)$\\
${\rm PSL}_4(K)$ & any  & $(2,3,c)$\\
${\rm PSL}_5(K)$ & any  & $(2,3,c)$\\
\hline
${\rm PSp}_4(K)$ & any & $(2,3,c)$, $(3,3,c)$\\
\hline
$G_{2}(K)$ & any  & $(2,4,5)$, (2,5,5)\\
\hline
\end{tabular}}
\caption{Rigid triples for simple algebraic groups  of adjoint type}\label{ta:rigida}
\end{table}

 As Theorem \ref{t:classification}   gives the list of the rigid triples of integers for a given simple algebraic group, it puts Conjecture  \ref{c:marionconj} into a very concrete context. Concerning the reducible case, we  prove the following nonexistence result.  

\begin{propn}\label{p:marionred}
 Let $G$ be a simple algebraic group defined over an algebraically closed field $K$ of prime characteristic $p$ and let $(a,b,c)$ be a reducible triple of integers for $G$.  Then a quasisimple group $G(p^r)$ (of the form $(G^F)'$ where $F$ is a Steinberg endomorphism) is never an $(a,b,c)$-group.
\end{propn}
  
  Theorem \ref{t:classification} which gives inclusively the  list of the reducible triple of integers for a given simple algebraic group is one of the ingredients of the proof of Proposition \ref{p:marionred}. As an immediate corollary of Theorem \ref{t:classification} and Proposition \ref{p:marionred} we obtain the following result. 
  
  \begin{cor}\label{c:classification}
  If $(G,p,(a,b,c))$ is as in  Table \ref{ta:red}  then $(G^F)'=G(p^r)$ is never an $(a,b,c)$-group. 
   \end{cor}

Unless otherwise stated, we let $G$ be a connected reductive algebraic group defined  over an algebraically closed field $K$ of  characteristic $p$ with maximal torus  $T$. If $T'$ is a torus of $G$ of dimension $r$ we sometimes write $T'=T_r$. \\
Given a positive integer $u$,  we let  $G_{[u]}$ be the subvariety of $G$ consisting of elements of order dividing $u$, we  set 
$$j_u(G)= \dim G_{[u]}$$ and we let $d_u(G)$ be the minimal dimension of a centralizer in $G$ of an element of $G$ of order dividing $u$. \\
\newpage
For $G$ simple, we fix some more notation as follows. We let $\Phi$ be the root system of $G$ with respect to $T$ and set $\Pi=\{\alpha_1,\dots, \alpha_\ell\} \subset \Phi$ to be a set of simple roots  of $G$, where $\ell$ is the rank of $G$. 
We let $h=|\Phi|/{\rm rank}(G)=|\Phi|/\ell$ be the Coxeter number of $G$. Recall that $h=\ell+1$, $2\ell$, $2\ell$, $2\ell-2$, $6$, $12$, $12$, $18$ or $30$ according respectively as $G$ is of type $A_\ell$, $B_\ell$, $C_\ell$, $D_\ell$, $G_2$, $F_4$, $E_6$. $E_7$ or $E_8$. Also we denote  by $G_{s.c.}$ (respectively, $G_{a.}$)  the simple algebraic group over $K$ of simply connected (respectively, adjoint) type having the same Lie type and Lie rank as  $G$. \\
 Unless otherwise stated given a positive integer $u$, we also write:\\
  $u=qv$ where $q$ and $v$ are coprime positive integers such that $q$ is a power of $p$  and $p$ does not divide $v$,\\
  $h=zu+e$ where $z,e$ are nonnegative integers such that $0\leq e \leq q-1$,\\
  $h=\alpha v+\beta$ where $\alpha,\beta$ are nonnegative integers such that $0\leq \alpha \leq v-1$,\\
  $\alpha=\gamma q+\delta$ where $\gamma,\delta$ are nonnegative integers such that $0\leq \delta \leq q-1$.\\
   An easy check yields $\gamma=z$.  Note that if $e=0$ then $\beta=\delta=0$. Furthermore, if $\beta =0$ then $\delta=0$ if and only if $e=0$.\\
Finally for an nonnegative integer $r$, we let $\epsilon_r \in \{0,1\}$ be $0$ if $r$ is even, otherwise $\epsilon_r=1$, and set $\sigma_r\in\{0,1\}$ to be 1 if $r=0$, otherwise $\sigma_r=0$.\\

The outline of the paper is as follows.  In \S\ref{s:precent} we give some preliminary results on centralizers. In \S\ref{s:tfr} we give detailed proofs for Propositions \ref{p:fcu} and \ref{p:ingsimple}.  In \S\ref{s:ex}, given a positive integer $u$, we determine $d_u(G)$ for $G$ a simple simply connected algebraic group of exceptional type over an algebraically closed field $K$. This establishes Theorem \ref{t:blawther} for $G$ of exceptional type. In \S\ref{s:spin} we recall some properties of the spin groups and determine when a preimage under the canonical map ${\rm Spin}_n(K) \rightarrow {\rm SO}_n(K)$ of a semisimple element  of ${\rm SO}_n(K)$ of a given order  is also of that order. In \S\ref{s:upc} given a positive integer $u$, we determine precise upper bounds for $d_u(G)$ where $G$ is a simple algebraic group of classical type over an algebraically closed field. In \S\ref{s:pc} given a positive integer $u$, we determine $d_u(G)$ for $G$ of classical type, completing the proofs of Theorems \ref{t:blawther} and  \ref{t:duotsca}.  In  \S\ref{s:decreasing} we prove that for a given simple algebraic group $G$ over an algebraically closed field,  $d_u(G)$, seen as a function of $u$, is decreasing. This establishes Proposition \ref{p:dudecreasing}. In  \S\ref{s:classification} we classify the reducible and the rigid hyperbolic triples of integers for simple algebraic groups, establishing Theorem \ref{t:classification}.  In \S\ref{s:reducibility} we prove Proposition \ref{p:marionred}. Finally in  \S\ref{s:tables} we collect Tables \ref{t:asc}-\ref{ta:casestoconsider} which appear later in the paper.\\

  \noindent \textbf{Acknowledgements.} The author thanks the MARIE CURIE and PISCOPIA research fellowship  scheme and the University of Padova  for their support.  The research leading to these results has received funding from the European Comission, Seventh Framework Programme (FP7/2007-2013) under Grant Agreement 600376.

\section{Preliminary results on centralizers}\label{s:precent}
In this section, unless otherwise stated,  $G$ denotes a connected reductive group defined over an algebraically closed field $K$ of characteristic $p$.  Given a positive integer $u$, recall that $d_u(G)$ is defined to be the minimal dimension of a centralizer in $G$ of an element of $G$ of order dividing $u$. We first recall some generalities on the centralizer in $G$ of an element of  $G$. 

\begin{lem}\label{l:prelimdu}
Let $G$ be a connected reductive group defined over an algebraically closed field $K$ of characteristic $p$. Let $g$ be an element of $G$ with Jordan decomposition $g=xy$ where $x$ is unipotent and $y$ is semisimple. Let $H=C_G(y)^0$. The following assertions hold. 
 \begin{enumerate}[(i)]
  \item We have  $x\in C_G(y)^0$, $C_G(g)=C_{C_{G(y)}}(x)$ and $C_G(g)^0=C_{C_G(y)^0}(x)^0$.
 \item  The group $H=C_G(y)^0$ is reductive and so $H=H_1\cdots H_r T'$ for some nonnegative integer $r$ and  simple groups $H_i$ and central torus $T'$. Morever  each $H_i$ is closed and normal in $H$, $[H_i,H_j]=1$ for $i\neq j$, and $H_i \cap H_1 \dots H_{i-1}H_{i+1}\cdots H_r$ is finite for each $i$. 
 \item  Write $x=x_1\cdots x_r t$ where $x_i \in H_i$ and $t\in T'$.  Then 
 $$C_{C_G(y)^0}(x)=C_{H_1}(x_1)\dots C_{H_r}(x_r)T'.$$
 \item $\dim C_G(g)=\dim C_{C_G(y)^0}(x)={\rm rank}(T')+ \sum_{i=1}^r \dim C_{H_i}(x_i).$ 
 \end{enumerate}
\end{lem}

\begin{proof}
We first consider part (i). Since $G$ is connected reductive we have  $x\in C_G(y)^0$ - see for example \cite[Proposition 14.7]{MaTe}. It is well-known  that $C_G(g)=C_{C_{G(y)}}(x)$. As $C_G(g)\leqs C_G(y)$, we have $C_G(g)^0\leqs C_G(y)^0$.  It now follows that $C_G(g)^0 \leqs C_{C_G(y)^0}(x)$ and so $C_G(g)^0=C_{C_G(y)^0}(x)^0$.\\
We consider part (ii). Since $G$ is connected reductive and $y$ is a semisimple element of $G$, $C_G(y)^0$ is connected reductive - see for example \cite[Theorem 14.2]{MaTe}.  Since $H=C_G(y)^0$ is connected reductive, we have $H=[H,H]Z(H)^0$ where $Z(H)^0$ is a central torus and $[H,H]$ is semisimple - see for example \cite[Corollary 8.22]{MaTe}. Part (ii) follows - see for example \cite[Theorem 8.21]{MaTe}. Part (iii) is now an easy consequence of part (ii).  \\
We now consider part (iv). We have 
$$\dim C_G(g)=\dim C_G(g)^0\quad \textrm{and}\quad  \dim C_{C_G(y)^0}(x)= \dim C_{C_G(y)^0}(x)^0. $$ Part (iv) now follows from parts (i)-(iii).\\
\end{proof}

 We now give some properties of $d_u(G)$. In the statement below, we adopt the notation of  Lemma \ref{l:prelimdu}(ii) for the connected component of the centralizer of a semisimple element in a connected reductive algebraic group.

\begin{prop}\label{p:prelimdu}
Let $G$ be a connected reductive algebraic group over an algebraically closed field $K$ of characteristic $p$. Let $u$ be a positive integer. Write $u=qv$ where $u$ and $v$ are coprime positive integers with $q$  a power of $p$  and $p$ does not divide $v$. Then the following assertions hold.
\begin{enumerate}[(i)]
\item $$\begin{array}{llll}d_u(G) &= & \min & d_q(C_G(y)^0).    \\
& &  y \in G_{[v]} & \\
\end{array}$$
\item  Let $g=xy$ be the Jordan decomposition of an element $g\in G$ where $x$ and $y$ are respectively the unipotent and semisimple parts of $g$.  If $g$ has order dividing $u$ and  $d_u(G)=\dim C_G(g)$, then $$d_u(G)=d_q(C_G(y)^0).$$ 
\item  If $y$ is a semisimple element  of $G$ with $C_G(y)^0=H_1\dots H_rT'$ (as in Lemma \ref{l:prelimdu}(ii)) then $$d_q(C_G(y)^0)= {\rm rank}(T')+\sum_{i=1}^r d_q(H_i).
$$
\item $$\begin{array}{llll}d_u(G) &= &   \min & {\rm rank}(T')+d_q(H_1)+\dots +d_q(H_r).  \\
&&  C_G(y)^0=\prod_{i=1}^r H_iT'& \\
& &  y \in G_{[v]} &\\
\end{array}$$
\end{enumerate}
\end{prop}

\begin{proof}
We first consider part (i).  By definition $d_u(G) = \min_{g \in G_{[u]}} \ \dim C_G(g)$.  Let $g$ be any element of $G_{[u]}$. Write $g=xy$ where $x$ and $y$ are respectively the unipotent and semisimple parts of $g$. Clearly $x\in G_{[q]}$ and $y\in G_{[v]}$. Following Lemma \ref{l:prelimdu}(i), we in fact have $x \in (C_G(y)^0)_{[q]}$ and $C_G(g)^0=C_{C_G(y)^0}(x)^0$. It follows that
$$\begin{array}{lllll}d_u(G) &= & \min & \min & \dim C_{C_G(y)^0}(x).    \\
&&  y\in G_{[v]} &          x \in (C_G(y)^0)_{[q]} &                        
\end{array}$$
This establishes part (i). \\
We now consider part (ii). Note  that $x\in (C_G(y)^0)_{[q]}$ and $y\in G_{[v]}$.   Since by assumption $d_u(G)=\dim C_G(g)$ and by Lemma \ref{l:prelimdu}(iv) $\dim C_G(g)= \dim C_{C_G(y)^0}(x)$, we obtain
\begin{equation}\label{e:lalabobo} 
d_u(G)= \dim C_{C_G(y)^0}(x).
\end{equation}
By part (i) $d_u(G)=\min_{l \in G_{[v]}} d_q(C_{G}(l)^0)$ and so $d_q(C_G(y)^0)\geq d_u(G)$. 
However if $d_q(C_G(y)^0)> d_u(G)$ then $\dim C_{C_G(y)^0}(x) > d_u(G)$, contradicting (\ref{e:lalabobo}). Hence $d_u(G)=d_q(C_G(y)^0)$ as claimed. \\
We now consider part (iii). Recall that by Lemma \ref{l:prelimdu}(ii), $C_G(y)^0$ is connected reductive. 
For $x\in (C_G(y)^0)_{[q]}$, write $x=x_1\dots x_r t$ where $x_i\in H_i$ for $1\leq i \leq r$ and $t\in T'$.  
By Lemma \ref{l:prelimdu}(iv) we have 
$$\dim C_{C_G(y)^0}(x)={\rm rank}(T')+ \sum_{i=1}^r \dim C_{H_i}(x_i).$$
Hence
\begin{eqnarray*}
d_q(C_G(y)^0) & = & \min_{x \in (C_G(y)^0)_{[q]}} \dim C_{C_G(y)^0}(x)  \\
& = &  \min_{x \in (C_G(y)^0)_{[q]}} {\rm rank}(T')+ \sum_{i=1}^r \dim C_{H_i}(x_i)\\
& = &{\rm rank}(T')+ \sum_{i=1}^r  \min_{x \in (C_G(y)^0)_{[q]}} \dim C_{H_i}(x_i)\\
& = & {\rm rank}(T')+\sum_{i=1}^r d_q(H_i),
\end{eqnarray*}
where to obtain the final equality we used Lemma \ref{l:prelimdu}(ii)-(iii). 
Part (iii) follows. \\
Finally  part (iv) follows from parts (i) and (iii). 
\end{proof}

The result below is the main ingredient in the proof of Proposition \ref{p:ingsimple}. 

\begin{lem}\label{l:dufactg}
Let $G$ be any  group, $C$ a normal subgroup of $G$, and $\phi: G\rightarrow G/C$ be the canonical surjective map.  Let  $g,k$ be  any elements of $G$ and set 
$$G_{g,k}=\{x \in G:g^{-1}x^{-1}gx=k\}.$$ The following assertions hold.
\begin{enumerate}[(i)]
\item Either $G_{g,k}=\emptyset$ or $G_{g,k}$ is a coset of $C_G(g)$.
\item We have $C_{G/C}(gC)= \bigcup_{k \in C} \phi(G_{g,k})$.
\item Suppose  $G$ is a simple algebraic group defined over an algebraically closed field $K$ of characteristic $p$ and $\phi$ is an isogeny. Then 
$$ C_{G/C}(gC)^0= (C_G(g)/C)^0\cong C_G(g)^0/(C\cap C_G(g)^0).$$ In particular, 
 $\dim C_{G/C}(gC)=\dim C_G(g)$. Moreover if $q$ is a power of $p$ and $g$ is a semisimple element of $G$ then $$d_q(C_{G/C}(gC)^0)=d_q(C_G(g)^0).$$  
\end{enumerate} 
\end{lem}

\begin{proof}
We first consider  part (i).  Note that if $g$ and $gk$ are not conjugate then $G_{g,k}=\emptyset$. Suppose that $g$ and $gk$ are conjugate, say $l^{-1}gl=gk$ for some $l\in G$. \\
Let $x$ be any element of $G_{g,k}$.  Then 
$$x^{-1}gx  =  gk  =   l^{-1}gl.$$
Hence $lx^{-1}g(lx^{-1})^{-1}=g$ and $xl^{-1}\in C_G(g)$. It follows that $x\in C_G(g)l$ and so $G_{g,k}\subseteq C_G(g)l$.  \\
Suppose now that $x \in C_G(g)l$. Then $x=yl$ for some $y\in C_G(g)$. Hence
$$ x^{-1}gx= (yl)^{-1}g(yl)= l^{-1}\cdot y^{-1}gy\cdot l= l^{-1}gl=gk$$ and so 
$x \in G_{g,k}$. It follows that $C_G(g)l\subseteq G_{g,k}$. We obtain $G_{g,k}= C_G(g)l$, proving part (i). \\
We now consider part (ii).  Suppose first that $xC$ is any element of $\phi(G_{g,k})$ where $k \in C$. Without loss of generality, $x\in G_{g,k}$. Hence 
$x^{-1}gx = gk$. Since $k \in C$, we obtain  $(xC)^{-1}\cdot gC\cdot xC = gC$ and so $xC \in C_{G/C}(gC)$. Hence $\bigcup_{k \in C} \phi(G_{g,k})\subseteq  C_{G/C}(gC)$.\\
Suppose now that $xC \in C_{G/C}(gC)$. Then 
$ (xC)^{-1}\cdot gC \cdot xC = gC$  and so $ (gC)^{-1}\cdot(xC)^{-1} \cdot gC \cdot xC  = C$. It follows that  there exists $k\in C$ such that $g^{-1}x^{-1}gx=k$  and so $x\in G_{g,k}$. Hence $xC\in \phi(G_{g,k})$ with $k\in C$. This shows that   $C_{G/C}(gC)\subseteq \bigcup_{k \in C} \phi(G_{g,k})$. We obtain  $C_{G/C}(gC)= \bigcup_{k \in C} \phi(G_{g,k})$, proving part (ii).\\
Finally we consider part (iii).   As $G$ is simple and $\phi$ is an isogeny, ${\rm ker} \ \phi=C$ is a finite central subgroup  of $G$.
Also following part (ii), $\phi (C_G(g))=C_G(g)/C$ is a subgroup of  $C_{G/C}(gC)$ of finite index and so $ C_{G/C}(gC)^0 \leqs (C_G(g)/C)^0$.   On the other hand, as $ C_G(g)/C\leqs   C_{G/C}(gC)$, we have $ (C_G(g)/C)^0\leqs   C_{G/C}(gC)^0$. Hence  \begin{equation*}\label{e:o} C_{G/C}(gC)^0 = (C_G(g)/C)^0.\end{equation*}
Note that as $C_G(g)^0$ is connected so is the factor group $C_G(g)^0/(C\cap C_G(g)^0)$ - see for example \cite[Proposition 1.10]{MaTe}. By the second isomorphism theorem, $C_G(g)^0/(C\cap C_G(g)^0)$ can be seen as a connected subgroup of $C_G(g)/C$. As this subgroup is of finite index, we obtain 
\begin{equation}\label{e:oo} C_G(g)^0/(C\cap C_G(g)^0) \cong (C_G(g)/C)^0=C_{G/C}(gC)^0.\end{equation}
As $C$ is a finite group, it  now follows from  (\ref{e:oo}) that  $\dim C_{G/C}(gC)=\dim C_G(g)$.  To conclude suppose that $g$ is a semisimple element of $G$. By Lemma \ref{l:prelimdu} $C_G(g)^0$ is a connected reductive group. By \cite{Lawther} given a simple algebraic group $H$ defined over $K$, $d_q(H)$ is independent of the isogeny type of $H$ (see Lemma \ref{l:lawther} below). 
It follows that  $d_q(C_G(g)^0/(C\cap C_G(g)^0))=d_q(C_G(g)^0)$ and so by (\ref{e:oo}) $d_q(C_G(g)^0)=d_q(C_{G/C}(gC)^0),$ as required.
\end{proof}

 For matter of clarity we record   \cite[Lemma 2.5 and Corollary 2.6]{Lawther}:

\begin{lem}\label{l:lawther}
Let $G$ be  a simple algebraic group defined over an algebraic closed field of  characteristic $p$. Let $q$ be a power of $p$. Then $d_q(G)$ does not depend on the isogeny type of $G$. 
Furthermore, let $\alpha$ be a nonnegative integer and write $\alpha=\gamma q +\delta$ where $0\leq \delta \leq q-1$ and let $\varsigma\in \{0,1\}$. Then the following assertions hold:
\begin{enumerate}[(i)]
\item $d_q(A_{\alpha-\varsigma})=\gamma^2q+(2\gamma+1)(\alpha-\varsigma-\gamma q+1)-1$.
\item $d_q(B_{\lceil \frac{\alpha}2\rceil}-\varsigma\epsilon_{\alpha})=\frac{\gamma^2q}2+(2\gamma+1)(\lceil \frac{\alpha}2\rceil-\varsigma\epsilon_{\alpha}-\frac{\gamma q}2)+\lceil  \frac{\gamma}2\rceil\epsilon_q$.
\item $d_q(C_{\lceil \frac{\alpha}2\rceil}-\varsigma\epsilon_{\alpha})=\frac{\gamma^2q}2+(2\gamma+1)(\lceil \frac{\alpha}2\rceil-\varsigma\epsilon_{\alpha}-\frac{\gamma q}2)+\lceil  \frac{\gamma}2\rceil\epsilon_q$.
\item $d_q(D_{\lceil \frac{\alpha+1}2\rceil-\varsigma\epsilon_{\alpha+1})}=\frac{\gamma^2q}2+(2\gamma+1)(\lceil \frac{\alpha+1}2\rceil-\varsigma\epsilon_{\alpha+1}-\frac{\gamma q}2)+\lceil  \frac{\gamma}2\rceil\epsilon_q-\gamma - \epsilon_\gamma+2\varsigma\epsilon_{\gamma(\alpha+1)}\sigma_{\alpha-\gamma q}$.
\end{enumerate}
\end{lem}

\begin{corol}\label{c:lawther}
Let $m$ be a natural number. The following assertions hold:
\begin{enumerate}[(i)]
\item If $m+1=\gamma q+\delta$ with $0 \leq \gamma \leq q-1$, then $d_q(A_{m+1})-d_q(A_m)=2\gamma+1$.  
\item If $2m+1=\gamma q+\delta$ with  $0 \leq \gamma \leq q-1$, then $d_q(B_{m+1})-d_q(B_m)=2\gamma+1$.
\item  If $2m+1=\gamma q+\delta$ with  $0 \leq \gamma \leq q-1$, then $d_q(C_{m+1})-d_q(C_m)=2\gamma+1$.
\item   If $2m=\gamma q+\delta$ with  $0 \leq \gamma \leq q-1$, then $d_q(D_{m+1})-d_q(D_m)=2\gamma+1-2\epsilon_\gamma\sigma_\delta$.
\end{enumerate}
\end{corol}

 Finally we record  the following result:
 
 \begin{lem}\label{l:centmod2} \cite[6.17]{Humphreys}.
Let $G$ be a simple algebraic group of rank $\ell$ over an algebraically closed field $K$ and let $x\in G$. Then 
$$ \dim C_G(x) \equiv \ell \mod 2.$$
\end{lem}

 \section{Proofs of Propositions \ref{p:fcu} and \ref{p:ingsimple}}\label{s:tfr}
 
 In this section we prove Propositions \ref{p:fcu} and \ref{p:ingsimple}. We first give a detailed proof of Proposition \ref{p:fcu}.

\noindent\textit{Proof of Proposition \ref{p:fcu}.}
Write $u=qv$ where $q$ and $v$ are coprime positive integers such that $q$ is a power of $p$  and $p$ does not divide $v$. Let $g$ be an element of $G$ of order dividing $u$. Let $g=xy$ be the Jordan decomposition of $g$ where $x$ and $y$ denote respectively the unipotent and semisimple parts of $g$. In particular the order of $x$ divides $q$ and that of $y$ divides $v$.  As $G$ is connected, every semisimple element of $G$ lies in a maximal torus of $G$, see for example \cite[Corollary 6.11]{MaTe}. Since all maximal tori of $G$ are conjugate, see for example  \cite[Corollary 6.5]{MaTe}, it follows that there are only finitely many conjugacy classes of $G$ of semisimple elements of order dividing $v$.  Let $y_1,\dots, y_r$ be  representatives of the conjugacy classes of $G$ of semisimple elements of order dividing $v$.   For $1 \leq i \leq r$, $C_G(y_i)^0$ is a connected reductive algebraic group
(see Lemma \ref{l:prelimdu}(ii)) and so  by \cite{Lusztig} $C_G(y_i)^0$ has only finitely many unipotent classes. Let $x_{i,1}, \dots , x_{i,t_i}$ be representatives of  the unipotent classes of $C_G(y_i)^0$. 
Now there exists $1\leq i\leq r$ such that $y$ is conjugate in $G$ to $y_i$. Let $k$ be an element of $G$ such that $kyk^{-1}=y_i$. Then 
\begin{eqnarray*}
kgk^{-1}& = & kxyk^{-1}\\
& = & kxk^{-1} \cdot k yk^{-1}\\
&  = & kxk^{-1}  \cdot y_i.
\end{eqnarray*}
Note that $kxk^{-1}\in C_G(y_i)$. Indeed 
\begin{eqnarray*}
kxk^{-1} \cdot y_i \cdot (kxk^{-1})^{-1} &=& kx \cdot  k^{-1}y_ik \cdot x^{-1}k^{-1}\\
& = & kx\cdot y x^{-1}k^{-1}\\
& = & kyxx^{-1}k^{-1}\\
& = & kyk^{-1}\\
& =& y_i.
\end{eqnarray*}
Since $G$ is a connected reductive group and $kxk^{-1}\cdot y_i$ is the Jordan decomposition  of $kxk^{-1}\cdot y_i$, it follows that $kxk^{-1}$  is a unipotent element of $C_G(y_i)^0$ - see Lemma \ref{l:prelimdu}(i). Hence there exist $l\in C_G(y_i)^0$ and $1\leq j \leq t_i$ such that $l\cdot kxk^{-1}\cdot l^{-1} = x_{i,j}.$ Also as $l\in C_G(y_i)^0$, we have $ly_il^{-1}=y_i$. Hence
\begin{eqnarray*}
lkgk^{-1}l^{-1}& = & l\cdot kgk^{-1} \cdot l^{-1}.\\
& = & l\cdot kxk^{-1}\cdot y_i \cdot l^{-1} \\
&  = & l\cdot kxk^{-1} l^{-1}\cdot l \cdot y_il^{-1}\\
& = & x_{i,j}y_i
\end{eqnarray*}
and so $g$ is conjugate to $x_{i,j}y_i$. It follows that there are at most $\sum_{i=1}^r t_i$ conjugacy classes of elements of $G$ of order dividing $u$. It follows that $G_{[u]}$ consists of the finite union of conjugacy classes of $G$ of elements of order dividing $u$. Hence 
$$ \dim G_{[u]} = \max_{g \in G_{[u]}} \dim g^G.$$ Since for any element  $g \in G$, ${\rm codim}\ g^G=\dim C_G(g)$ - see for example \cite[Proposition 1.5]{Humphreys}, we obtain $$ {\rm codim} \ G_{[u]}= \dim G -  \max_{g \in G_{[u]}} \dim g^G = \min_{g \in G_{[u]}} \dim C_G(g) = d_u(G). \quad \square$$

We can now prove Proposition \ref{p:ingsimple}.

\noindent \textit{Proof of Proposition \ref{p:ingsimple}.}
 First, by \cite{Lawther}, we have $d_u(G_{a.})\leq d_u(G)$. It remains to show that $d_u(G) \leq d_u(G_{s.c.})$. Note that the result is trivial unless $G$ is neither of adjoint nor simply connected type.    
 Let  $\phi: G_{s.c.} \rightarrow G$ be an isogeny.   Let $g$ be any element of $G_{s.c.}$ of order dividing $u$ such that $d_u(G_{s.c.})=\dim C_{G_{s.c.}}(g)$. Then $\phi(g)$ is an element of $G$ of order dividing $u$ and by Lemma \ref{l:dufactg} $\dim C_G(\phi(g))=\dim C_{G_{s.c.}}(g)$. Hence $\dim C_G(\phi(g))=d_u(G_{s.c.})$.  Since  $d_u(G)\leq \dim C_G(\phi(g))$ we obtain $d_u(G)\leq d_u(G_{s.c.})$. $\square$

 \section{Proof of Theorem \ref{t:blawther} for $G$ of exceptional type}\label{s:exceptionalgroups}\label{s:ex}

   In this section, we determine $d_u(G)$ for $G$ a simple algebraic group of exceptional type defined over an algebraically  closed field $K$ of characteristic $p$. In particular we prove  Theorem \ref{t:blawther} for $G$ of exceptional type (see Propositions \ref{p:e6sc} and \ref{p:e7sc} below).  Recall that a simple algebraic group of exceptional type is either simply connected or adjoint. Following the result of Lawther \cite{Lawther} giving $d_u(G)$ for $G$ of adjoint type, we suppose that $G$ is of simply connected type. Without loss of generality,
we  assume that $G$ is of type $E_6$ or $E_7$ and $p\neq 3$ or $2$ respectively, as
these are the only cases where the simply connected and the adjoint groups are
not abstractly isomorphic. \\
Let $T$ be a maximal torus of $G$ with corresponding root system $\Phi$ and set $\Pi=\{\alpha_1,\dots, \alpha_\ell\} \subset \Phi$ to be a set of simple roots  of $G$, where $\ell$ is the rank of $G$. Note  that $\ell=6$ or $\ell=7$ according respectively as $G=E_6$ or $G=E_7$. \\
Given a positive integer $u$, we write $u=qv$ where $q$ and $v$ are positive coprime integers with $q$ a power of $p$  and $p$ does not divide $v$.\\
Let $y$ be a semisimple element of $G$ of order dividing $v$, i.e. $y\in G_{[v]}$.  Since every semisimple element of $G$ belongs to a maximal torus of $G$ and all maximal tori are conjugate, we can assume without loss of generality that $y \in T$. 
Now 
\begin{equation}\label{e:centalgg}
C_G(y)^0= \langle T, U_{\alpha}: \alpha \in \Psi \rangle\end{equation}where 
$\Psi=\{\alpha \in \Phi: \alpha(y)=1\}$ and $U_{\alpha}$ is the root subgroup of $G$ corresponding to $\alpha$.
Also we can write \begin{equation*}\label{e:sseg}
y= \prod_{i=1}^\ell h_{\alpha_i}(k^{c_i})\end{equation*} 
where  for $l \in  K^*=K\setminus\{0\}$ and $1\leq i \leq \ell$, $h_{\alpha_i}(l)$ is the image of $\begin{pmatrix} l & 0 \\ 0 & l^{-1}\end{pmatrix}\in {\rm SL}_2(K)$ under the canonical map ${\rm SL}_2(K) \rightarrow \langle U_{\alpha_i}, U_{-\alpha_{i}}\rangle$, $k$ is an element of $K^*$ of order $v$, and $0\leq c_i\leq v-1$ is an integer  for all $1\leq i \leq \ell$. \\
Given $\alpha \in \Phi$ we then have 
\begin{equation}\label{e:centalgp}
 \alpha(y)= k^{\sum_{i=1}^\ell c_i\langle \alpha_i,\alpha  \rangle}
 \end{equation} where $$\langle  \alpha_i, \alpha\rangle= \frac{2(\alpha_i,\alpha)}{(\alpha_i,\alpha_i)}$$  and $(,)$ denotes the  inner product of the Euclidean space spanned by $\Phi$. 

\begin{lem}\label{l:duered}
Let $G$ be a simple simply connected algebraic group of type $X=E_6$ or $E_7$  defined over an algebraically closed field $K$ of characteristic $p$.   Let $u$ be a positive integer. Then $d_u(G)=d_u(G_{a.})$ except possibly if $G=E_6$, $p\neq 3$ and $u\equiv 0\mod 3$, or $G=E_7$, $p \neq 2$ and $u\equiv 0 \mod 2$. 
\end{lem}

\begin{proof}
Set $C=Z(G)$ and note that  $$|C|=\left\{\begin{array}{ll} 
3 & \textrm{if} \ G=E_6 \ \textrm{and} \ p\neq 3\\
2 & \textrm{if} \ G=E_7 \ \textrm{and} \ p\neq 2\\
1 & \textrm{otherwise.}\\
 \end{array} \right.$$
Also note that if $(X,p) \in \{(E_6,3), (E_7,2)\}$ then $G$ is abstractly isomorphic to $G_{a.}$ and so  $d_u(G)=d_u(G_{a.})$ for every positive integer $u$. We therefore assume that $(X,p)\not \in   \{(E_6,3), (E_7,2)\}$. Then $C=Z(G)=\langle c \rangle$ is cyclic group of order $3$ or $2$ according respectively as $X=E_6$ or $E_7$.   Let $\phi: G\rightarrow G_{a.}$ be the surjective canonical map. Let $gC$ be an element of $G_{a.}$ of order dividing $u$ such that $\dim C_{G_{a.}}(gC)=d_u(G_{a.})$. \\
Suppose $u\neq 0 \mod |C|$.  Then $g \in G$ has order dividing $u$, and by Lemma \ref{l:dufactg}  $\dim C_G(g)=\dim C_{G_{a.}}(gC)$. Hence $\dim C_G(g)=d_u(G_{a.})$ and $d_u(G) \leq d_u(G_{a.})$. It now follows from Proposition \ref{p:ingsimple} that $d_u(G)=d_u(G_{a.})$.
\end{proof}

\begin{prop}\label{p:e6sc}
Suppose $G$ is a simple simply connected algebraic group of type $E_6$  defined over an algebraic closed field $K$ of characteristic $p$.  Let $u$ be a positive integer. Then $d_u(G)=d_u(G_{a.})$. 
\end{prop}

\begin{proof}
Note that by \cite{Lawther} $d_u(G)\geq d_u(G_{a.})$ and $d_u(G_{a})=6$ for $u\geq h=12$.\\
By Lemma \ref{l:duered} we can assume without loss of generality that $p \neq 3$ and $u \equiv 0 \mod 3$.  Write $u=qv$ where $q$ and $v$ are coprime positive integers such that $q$ is a power of $p$  and $p$ does not divide $v$. \\
 We first suppose that $u \leq h=12$. \\
Assume $u=3$.  Then $q=1$ and $u=v=3$. Let $$y_3=h_{\alpha_1}(k)h_{\alpha_2}(k^2)h_{\alpha_3}(k)h_{\alpha_4}(k^2)h_{\alpha_5}(k^2)h_{\alpha_6}(1)$$ where $k \in K^*$ is an element of order 3.  Then $y_3$ is a semisimple element of $G$ of order 3. Moreover by  (\ref{e:centalgg}) and (\ref{e:centalgp}) we get
$$C_G(y_3)^0=A_2A_2A_2.$$ Hence $d_q(C_G(y_3)^0)=d_1(A_2A_2A_2)=24$. Now by \cite{Lawther} $d_u(G_{a.})=24$. We now deduce from  Proposition \ref{p:prelimdu}(i)  that $d_u(G)=d_u(G_{a.})=24.$\\
Assume $u=6$.  Then $q=1$ and $u=v=6$, or $q=p=2$ and $v=3$. Suppose first that $q=1$. 
Let $$y_6=h_{\alpha_1}(k)h_{\alpha_2}(k^2)h_{\alpha_3}(k)h_{\alpha_4}(k^5)h_{\alpha_5}(k^2)h_{\alpha_6}(1)$$ where $k \in K^*$ is an element of order 6.  Then $y_6$ is a semisimple element of $G$ of order 6. Moreover by  (\ref{e:centalgg}) and (\ref{e:centalgp}) we get
$$C_G(y_6)^0=A_1A_1A_1T_3.$$ Hence $d_q(C_G(y_6)^0)=d_1(A_1A_1A_1T_3)=12$. \\
 Suppose now that $q=2$ and $v=3$. 
Consider the semisimple element $y_3$ of $G$ of order 3 defined above. We have $C_G(y_3)^0=A_2A_2A_2$ and $d_q(C_G(y_3)^0)=d_2(A_2A_2A_2)=12$. \\
Now by \cite{Lawther} $d_u(G_{a.})=12$. We now deduce from  Proposition \ref{p:prelimdu}(i)  that $d_u(G)=d_u(G_{a.})=12.$\\
Assume $u=9$.  Then $q=1$ and $u=v=9$. Let $$y_9=h_{\alpha_1}(k)h_{\alpha_2}(k^2)h_{\alpha_3}(k)h_{\alpha_4}(k^5)h_{\alpha_5}(k^8)h_{\alpha_6}(1)$$ where $k \in K^*$ is an element of order 9.  Then $y_9$ is a semisimple element of $G$ of order 9. Moreover by  (\ref{e:centalgg}) and (\ref{e:centalgp}) we get
$$C_G(y_9)^0=A_1T_5.$$ Hence $d_q(C_G(y_9)^0)=d_1(A_1T_5)=8$. Now by \cite{Lawther} $d_u(G_{a.})=8$. We now deduce from  Proposition \ref{p:prelimdu}(i)  that $d_u(G)=d_u(G_{a.})=8.$\\
Assume $u=12$.  Then $q=1$ and $u=v=12$, or $q=4$, $p=2$ and $v=3$. Suppose first that $q=1$. 
Let $$y_{12}=h_{\alpha_1}(k)h_{\alpha_2}(k^2)h_{\alpha_3}(k)h_{\alpha_4}(k^5)h_{\alpha_5}(k^8)h_{\alpha_6}(1)$$ where $k \in K^*$ is an element of order 12.  Then $y_{12}$ is a semisimple element of $G$ of order 12. Moreover by  (\ref{e:centalgg}) and (\ref{e:centalgp}) we get
$$C_G(y_{12})^0=T_6.$$ Hence $d_q(C_G(y_{12})^0)=d_1(T_6)=6$. \\
 Suppose now that $q=4$, $p=2$ and $v=3$. 
Consider the semisimple element $y_3$ of $G$ of order 3 defined above. We have $C_G(y_3)^0=A_2A_2A_2$ and $d_q(C_G(y_3)^0)=d_4(A_2A_2A_2)=6$.
By  Proposition \ref{p:prelimdu}(i) we obtain  $d_{12}(G)=d_{12}(G_{a.})=6$.
 \\
We now suppose that $u>h=12$.  If $v\geq h$ then let $$y=h_{\alpha_1}(k)h_{\alpha_2}(k^2)h_{\alpha_3}(k)h_{\alpha_4}(k^5)h_{\alpha_5}(k^8)h_{\alpha_6}(1)$$ where $k \in K^*$ is an element of order $v$.  Then $y$ is a semisimple element of $G$ of order $v$. Moreover by  (\ref{e:centalgg}) and (\ref{e:centalgp}) we get
$$C_G({y})^0=T_6.$$ Hence $d_q(C_G(y)^0)=d_q(T_6)=6$ and so by  Proposition \ref{p:prelimdu}(i) $d_u(G)=6$. \\
We finally suppose that $u>h=12$ and $v<h$. As we are assuming that $u \equiv 0 \mod 3$ and $p\neq 3$ we have $v\equiv 0 \mod 3$. Hence $v\in \{3,6,9\}$. \\
Assume $v=3$. As $u>12$ we have $q\geq 5$.  Now $y_3$ is a semisimple element of $G$ of order $3$ with $C_G(y_3)^0=A_2A_2A_2$. As $q\geq 3$, we have 
$d_q(C_G(y_3)^0)=d_q(A_2A_2A_2)=6$. \\
Assume $v=6$. As $u>12$ we have $q\geq 3$.  Now $y_6$ is a semisimple element of $G$ of order $6$ with $C_G(y_3)^0=A_1A_1A_1T_3$. As $q\geq 2$, we have 
$d_q(C_G(y_6)^0)=d_q(A_1A_1A_1T_3)=6$. \\
Assume $v=9$. As $u>12$ we have $q\geq 2$.  Now $y_9$ is a semisimple element of $G$ of order $9$ with $C_G(y_9)^0=A_1T_5$. As $q\geq 2$, we have 
$d_q(C_G(y_9)^0)=d_q(A_1T_5)=6$. \\
Hence in the case where $u>h$ and $v<h$, by  Proposition \ref{p:prelimdu}(i), we get $d_u(G)=6.$
\end{proof}

\begin{prop}\label{p:e7sc} 
Suppose $G$ is a simple simply connected algebraic group of type $E_7$  defined over an algebraic closed field $K$ of characteristic $p$.  Let $u$ be a positive integer. The following assertions hold.
\begin{enumerate}[(i)]
\item We have $d_u(G)=d_u(G_{a.})$ unless $p\neq 2$ and $u\in\{2,6,10,14,18\}$.
\item Suppose $p\neq 2$. Then  $d_2(G)=d_u(G_{a.})+6$ and $d_u(G)=d_u(G_{a.})+2$ for $u\in\{6,10,14,18\}$. More precisely,
$$d_2(G)=69,\ d_6(G)=23,\ d_{10}(G)=15, \ d_{14}(G)=11 \quad \textrm{and} \quad d_{18}(G)=9.$$ 
\end{enumerate}
\end{prop}

\begin{proof}
Recall that by \cite{Lawther} $d_u(G)\geq d_u(G_{a.})$, $d_u(G_{a})=7$ for $u\geq h=18$, and $$d_2(G_{a.})=63,\ d_6(G_{a.})=21, \ d_{10}(G_{a.})=13 \quad \textrm{and}\quad d_{14}(G_{a.})=9.$$
By Lemma \ref{l:duered} we can assume without loss of generality that $p \neq 2$ and $u \equiv 0 \mod 2$, as otherwise $d_u(G)=d_u(G_{a.})$.  Write $u=qv$ where $q$ and $v$ are coprime positive integers such that $q$ is a power of $p$  and $p$ does not divide $v$. Note that $q$ is odd and $v\equiv 0 \mod 2$. \\
 We first suppose that $u \leq h=18$. \\
Assume $u=2$.  Then $q=1$ and $u=v=2$. By \cite[Table 6]{Cohen}, $d_2(G)=69.$ Let $$y_2=h_{\alpha_1}(k)h_{\alpha_2}(1)h_{\alpha_3}(1)h_{\alpha_4}(1)h_{\alpha_5}(k)h_{\alpha_6}(k)h_{\alpha_7}(1)$$ where $k \in K^*$ is an element of order 2.  Then $y_2$ is a semisimple element of $G$ of order 2. Moreover by  (\ref{e:centalgg}) and (\ref{e:centalgp}) we get
$$C_G(y_2)^0=A_1D_6.$$ Hence $d_q(C_G(y_2)^0)=d_1(A_1D_6)=69=d_2(G)$. \\
Assume $u=4$.  Then $q=1$ and $u=v=4$.  By \cite{Lawther} and  \cite[Table 6]{Cohen}, $d_4(G)=d_4(G_a)=33$.  Let $$y_4=h_{\alpha_1}(k)h_{\alpha_2}(1)h_{\alpha_3}(1)h_{\alpha_4}(k^2)h_{\alpha_5}(k)h_{\alpha_6}(k)h_{\alpha_7}(k^2)$$ where $k \in K^*$ is an element of order 4.  Then $y_4$ is a semisimple element of $G$ of order 4. Moreover by  (\ref{e:centalgg}) and (\ref{e:centalgp}) we get
$$C_G(y_4)^0=A_3A_3A_1.$$ Hence $d_q(C_G(y_4)^0)=d_1(A_3A_3A_1)=33=d_4(G)$. \\
Assume $u=6$.  Then $q=1$ and $u=v=6$, or $q=p=3$ and $v=2$.  Suppose first that $q=p=3$ and $v=2$. By \cite[Table 6]{Cohen} a semisimple element  $y$ of $G$ of order $2$ is   
such that $C_G(y)^0= A_1D_6$ or $E_7$. It follows from  Proposition \ref{p:prelimdu}(i) that $d_6(G)=d_3(A_1D_6)=23.$
 Suppose now that $q=1$. By \cite[Table 6]{Cohen}, $d_6(G)=23.$ Let $$y_6=h_{\alpha_1}(k)h_{\alpha_2}(1)h_{\alpha_3}(1)h_{\alpha_4}(k^2)h_{\alpha_5}(k^5)h_{\alpha_6}(k^3)h_{\alpha_7}(k^2)$$ where $k \in K^*$ is an element of order 6.  Then $y_6$ is a semisimple element of $G$ of order 6. Moreover by  (\ref{e:centalgg}) and (\ref{e:centalgp}) we get
$$C_G(y_6)^0=A_3A_1A_1T_2.$$ Hence $d_q(C_G(y_6)^0)=d_1(A_3A_1A_1T_2)=23=d_6(G)$. \\
Assume $u=8$.  Then $q=1$ and $u=v=8$.   Let $$y_8=h_{\alpha_1}(1)h_{\alpha_2}(1)h_{\alpha_3}(k)h_{\alpha_4}(k)h_{\alpha_5}(k^2)h_{\alpha_6}(k^7)h_{\alpha_7}(k^3)$$ where $k \in K^*$ is an element of order 8.  Then $y_8$ is a semisimple element of $G$ of order 8. Moreover by  (\ref{e:centalgg}) and (\ref{e:centalgp}) we get
$$C_G(y_8)^0=A_2A_1A_1T_3.$$ Hence $d_q(C_G(y_8)^0)=17$. Since, by \cite{Lawther}, $d_8(G_a)=17$, we deduce from  Proposition \ref{p:prelimdu}(i) that $d_8(G)=17$.\\
Assume $u=10$.  Then $q=1$ and $u=v=10$, or $q=p=5$ and $v=2$.  Suppose first that  $q=p=5$ and $v=2$. By \cite[Table 6]{Cohen} a semisimple element  $y$ of $G$ of order $2$ is   
such that $C_G(y)^0= A_1D_6$ or $E_7$. It follows from  Proposition \ref{p:prelimdu}(i) that $d_{10}(G)=d_5(A_1D_6)=15.$
 Suppose now that $q=1$. By  (\ref{e:centalgg}) and (\ref{e:centalgp}), a semisimple $y$ element of $G$ of order 10 is such that $C_G(y)^0$ is generated by at least four root subgroups $U_\alpha$ where $\alpha$  is a  positive root of $G$. 
     It follows from \cite{Lawther} that $d_{10}(G)>d_{10}(G_{a.})=13$. Hence by Lemma \ref{l:centmod2} $d_{10}(G)\geq 15$.   Let $$y_{10}=h_{\alpha_1}(k)h_{\alpha_2}(1)h_{\alpha_3}(1)h_{\alpha_4}(1)h_{\alpha_5}(k^6)h_{\alpha_6}(k)h_{\alpha_7}(k^8)$$  where $k \in K^*$ is an element of order 10.  Then $y_{10}$ is a  semisimple element of $G$ of order 10. Moreover by  (\ref{e:centalgg}) and (\ref{e:centalgp}) we get
$$C_G(y_{10})^0=A_1A_2T_4.$$ 
Hence $$d_q(C_G(y_{10})^0)=d_1(A_1A_2T_4)=15.$$ It now follows from  Proposition \ref{p:prelimdu}(i) that $d_{10}(G)=15.$\\
Assume $u=12$.  Then $q=1$ and $u=v=12$, or $q=p=3$ and $v=4$.   Suppose first that $q=p=3$ and $v=4$.  Consider the semisimple element $y_4$ of $G$ defined above. Then $C_G(y_4)^0=A_3A_3A_1$ and  $d_q(C_G(y_4)^0)=d_3(A_3A_3A_1)=11$. Hence by  Proposition \ref{p:prelimdu}(i), we obtain $d_{12}(G)\leq 11$. 
Suppose now that $q=1$ and $u=v=12$. Let $$y_{12}=h_{\alpha_1}(k)h_{\alpha_2}(1)h_{\alpha_3}(1)h_{\alpha_4}(k)h_{\alpha_5}(k^3)h_{\alpha_6}(k^5)h_{\alpha_7}(k^8)$$ where $k \in K^*$ is an element of order 12.  Then $y_{12}$ is a semisimple element of $G$ of order 12. Moreover by  (\ref{e:centalgg}) and (\ref{e:centalgp}) we get
$$C_G(y_{12})^0=A_1A_1T_5.$$ Hence $d_q(C_G(y_{12})^0)=11$.\\
 Since, by \cite{Lawther}, $d_{12}(G_a)=11$, we deduce  that $d_{12}(G)=11$ for all $p$.\\
Assume $u=14$.  Then $q=1$ and $u=v=14$, or $q=p=7$ and $v=2$.  Suppose first that $q=p=7$ and $v=2$. By \cite[Table 6]{Cohen} a semisimple element  $y$ of $G$ of order $2$ is   
such that $C_G(y)^0= A_1D_6$ or $E_7$. It follows from  Proposition \ref{p:prelimdu}(i) that $d_{14}(G)=d_7(A_1D_6)=11.$
 Suppose now that $q=1$. By  (\ref{e:centalgg}) and (\ref{e:centalgp}), a semisimple $y$ element of $G$ of order 14 is such that $C_G(y)^0$ is generated by at least two root subgroups $U_\alpha$ where $\alpha$  is a  positive root of $G$.   We deduce that $C_G(s)^0\neq A_1T_6$. It follows from \cite{Lawther} that $d_{14}(G)>d_{14}(G_{a.})=9$. Hence by Lemma \ref{l:centmod2} $d_{14}(G)\geq 11$.   Let $$y_{14}=h_{\alpha_1}(k)h_{\alpha_2}(1)h_{\alpha_3}(1)h_{\alpha_4}(1)h_{\alpha_5}(k^2)h_{\alpha_6}(k^5)h_{\alpha_7}(k^9)$$  where $k \in K^*$ is an element of order 14.  Then $y_{14}$ is a  semisimple elements of $G$ of order 14. Moreover by  (\ref{e:centalgg}) and (\ref{e:centalgp}) we get
$$C_G(y_{14})^0=A_1A_1T_5.$$ 
Hence $$d_q(C_G(y_{14})^0)=d_1(A_1A_1T_5)=11.$$ It now follows from  Proposition \ref{p:prelimdu}(i) that $d_{14}(G)=11.$\\
Assume $u=16$.  Then $q=1$ and $u=v=16$.   Let $$y_{16}=h_{\alpha_1}(k)h_{\alpha_2}(1)h_{\alpha_3}(1)h_{\alpha_4}(k)h_{\alpha_5}(k^3)h_{\alpha_6}(k^6)h_{\alpha_7}(k^{10})$$ where $k \in K^*$ is an element of order 16.  Then $y_{16}$ is a semisimple element of $G$ of order 16. Moreover by  (\ref{e:centalgg}) and (\ref{e:centalgp}) we get
$$C_G(y_{16})^0=A_1T_6.$$ Hence $d_q(C_G(y_{16})^0)=9$. Since, by \cite{Lawther}, $d_{16}(G_a)=9$, we deduce from  Proposition \ref{p:prelimdu}(i) that $d_{16}(G)=9$.\\
Assume $u=18$.  Then $q=1$ and $u=v=18$, or $q=9$, $p=3$ and $v=2$.  Suppose first that $q=9$, $p=3$ and $v=2$. By \cite[Table 6]{Cohen} a semisimple element  $y$ of $G$ of order $2$ is   
such that $C_G(y)^0= A_1D_6$ or $E_7$. It follows from  Proposition \ref{p:prelimdu}(i) that $d_{18}(G)=d_9(A_1D_6)=9$.
 Suppose now that $q=1$. By  (\ref{e:centalgg}) and (\ref{e:centalgp}), a semisimple $y$ element of $G$ of order 18 is such that $C_G(y)^0$ is generated by at least one root subgroup $U_\alpha$ where $\alpha$  is a  positive root of $G$.   We deduce that $C_G(y)^0\neq T_7$. It follows from \cite{Lawther} that $d_{18}(G)>d_{18}(G_{a.})=7$. Hence by Lemma \ref{l:centmod2} $d_{18}(G)\geq 9$.   Let $$y_{18}=h_{\alpha_1}(k)h_{\alpha_2}(1)h_{\alpha_3}(1)h_{\alpha_4}(k)h_{\alpha_5}(k^3)h_{\alpha_6}(k^6)h_{\alpha_7}(k^{10})$$  where $k \in K^*$ is an element of order 18.  Then $y_{18}$ is a  semisimple element of $G$ of order 18. Moreover by  (\ref{e:centalgg}) and (\ref{e:centalgp}) we get
$$C_G(y_{18})^0=A_1T_6.$$ 
Hence $$d_q(C_G(y_{18})^0)=d_1(A_1T_6)=9.$$ It now follows from  Proposition \ref{p:prelimdu}(i) that $d_{18}(G)=9.$\\

We now suppose that $u>h=18$. Recall that $p\neq 2$ and $v\equiv 0 \mod 2$.  Assume $v> h$. Note that  $v\geq 20$. Let $$y=h_{\alpha_1}(k)h_{\alpha_2}(1)h_{\alpha_3}(1)h_{\alpha_4}(k^2)h_{\alpha_5}(k^5)h_{\alpha_6}(k^9)h_{\alpha_7}(k^{14})$$ where $k \in K^*$ is an element of order $v$.  Then $y$ is a semisimple element of $G$ of order $v\geq 20$. Moreover by  (\ref{e:centalgg}) and (\ref{e:centalgp}) we get
$$C_G({y})^0=T_7.$$ Hence $d_q(C_G(y)^0)=d_q(T_7)=7$ and so by  Proposition \ref{p:prelimdu}(i) $d_u(G)=7$. \\
We finally suppose that $u>h=18$ and $v\leq h$. Since  $v\equiv 0 \mod 2$ we have $v\in \{2,4,6,8,10,12,14,16,18\}$. \\
Assume $v=2$. As $u\geq 20$ we have $q\geq 10$.  Now $y_2$ is a semisimple element of $G$ of order $2$ with $C_G(y_2)^0=A_1D_6$. As $q\geq 10$, we have 
$d_q(C_G(y_2)^0)=d_q(A_1D_6)=7$. \\
Assume $v=4$. As $u\geq 20$ we have $q\geq 5$.  Now $y_4$ is a semisimple element of $G$ of order $4$ with $C_G(y_4)^0=A_3A_3A_1$. As $q\geq 4$, we have 
$d_q(C_G(y_4)^0)=d_q(A_3A_3A_1)=7$. \\
Assume $v=6$. As $u\geq 20$ and $p\neq 2$ we have $q\geq 5$.  Now $y_6$ is a semisimple element of $G$ of order $6$ with $C_G(y_6)^0=A_3A_1A_1T_2$. As $q\geq 4$, we have 
$d_q(C_G(y_6)^0)=d_q(A_3A_1A_1T_2)=7$. \\
Assume $v=8$. As $u\geq 20$ we have $q\geq 3$.  Now $y_8$ is a semisimple element of $G$ of order $8$ with $C_G(y_8)^0=A_2A_1A_1T_3$. As $q\geq 3$, we have 
$d_q(C_G(y_8)^0)=d_q(A_2A_1A_1T_3)=7$. \\
Assume $v=10$. As $u\geq 20$ and $p\neq 2$ we have $q\geq 3$.  Now $y_{10}$ is a semisimple element of $G$ of order $10$ with $C_G(y_{10})^0=A_2A_1T_4$. As $q\geq 3$, we have 
$d_q(C_G(y_{10})^0)=d_q(A_2A_1T_4)=7$. \\
Assume $v\in\{12,14\}$. As $u\geq 20$ and $p\neq 2$ we have $q\geq 3$.  Now $y_{12}$  and $y_{14}$ are semisimple elements of $G$ of respective orders $12$ and $14$ with $C_G(y_{12})^0=A_1A_1T_5$ and $C_G(y_{14})^0=A_1A_1T_5$  . As $q\geq 2$, we have 
$d_q(C_G(y_{12})^0)=d_q(C_G(y_{14})^0)=d_q(A_1A_1T_5)=7$. \\
Assume $v\in\{16,18\}$. As $u\geq 20$ and $p\neq 2$ we have $q\geq 3$.  Now $y_{16}$  and $y_{18}$ are semisimple elements of $G$ of respective orders $16$ and $18$ with $C_G(y_{16})^0=A_1T_6$ and $C_G(y_{18})^0=A_1T_6$. As $q\geq 2$, we have 
$d_q(C_G(y_{16})^0)=d_q(C_G(y_{18})^0)=d_q(A_1T_6)=7$. \\
Hence in the case where $u>h$ and $v\leq h$, by  Proposition \ref{p:prelimdu}(i), we get $d_u(G)=7.$
\end{proof}

 \section{Some properties of ${\rm Spin}_n(K)$}\label{s:spin}

Let $K$ now denote an algebraically closed field $K$ of 
characteristic $p\neq2$. We give a characterization of semisimple elements of ${\rm
Spin}_n(K)$ of  a given order  (see Lemma
 \ref{l:soliftspin} below). Before doing so we recall some properties of
the group ${\rm Spin}_n(K)$ and the canonical surjective map ${\rm
Spin}_n(K)\rightarrow {\rm SO}_n(K)$, where $n\geq 7$ is a positive
integer.  Write $n=2\ell$ or $n=2\ell+1$ for some integer $\ell\geq
1$, according respectively as $n$ is even or odd.\\
Let $V$ be the natural module for ${\rm SO}_n(K)$ and let
$$\mathcal{B}=\{e_1,f_1,\dots,e_\ell,f_\ell\} \quad  \textrm{or} \quad
\mathcal{B}=\{e_1,f_1,\dots, e_\ell,f_\ell,d\}$$ (according
respectively as $n$ is even or odd) be a standard basis of $V$. That
is, $(e_i,f_i)=1$, $(e_i,e_i)=(f_i,f_i)=0$ for $1\leq i\leq \ell$,
$(e_i,f_j)=(e_i,e_j)=(f_i,f_j)=0$ for  $i\neq j$ such that $1\leq i,
j\leq \ell$ and if $n$ is odd then $(d,d)=1$ and $(d,e_i)=(d,f_i)$ for $1\leq i \leq \ell$. Here $(,): V\times V
\rightarrow K$ denotes the non-degenerate symmetric bilinear form
associated to $V$.  \\
Let $T_0(V,K)$ be the $K$-algebra of polynomials in
$e_1,f_1,\dots,e_\ell,f_\ell$  if $n$ is even (respectively, in
$e_1,f_1,\dots,e_\ell,f_\ell,d$ if $n$ is odd). Also let $I(V,K)$ be
the ideal of $T_0(V,K)$ generated by $\{vv-(v,v):v \in V\}$ and let
$C_0(V,K)=T_0(V,K)/I(V,K)$.\\
Let ${}^*: C_0(V,K)\rightarrow C_0(V,K)$ be the $K$-linear map such
that for any even  positive integer $r$
and elements $x_1,\dots, x_r \in \mathcal{B}$,  we have the equality
$(x_1\dots x_r)^*=x_r\dots x_1$ (in $C_0(V,K))$. Note that for $x$ in
$V$, we have $x^*x=xx^*=xx=(x,x)$. \\

We can now define the spin group ${\rm Spin}_n(K)$ as an abstract
group, namely ${\rm Spin}_n(K)$ consists of the elements $t$ of
$C_0(V,K)$ such that:
$$ t^*t=1,\quad tVt^{-1}=V$$ and the map
$$ \begin{array}{l} V \rightarrow V\\
x\mapsto txt^{-1}
 \end{array}$$ has determinant 1.

 Given an element $t$ of ${\rm Spin}_n(K)$, consider the map $\phi_t:
V\rightarrow V$
defined by $\phi_t(x)=txt^{-1}$ for $x$ in $V$. Then for any $x$ in $V$,
$$(\phi_t(x),\phi_t(x))=\phi_t(x)\phi_t(x)=txt^{-1}txt^{-1}=tx^2t^{-1}=t(x,x)t^{-1}=(x,x)$$
and so $\phi_t$ is an element of ${\rm SO}_n(K)$.  In fact the map
$\phi$:
$$\begin{array}{lll}{\rm Spin}_n(K) & \rightarrow & {\rm SO}_n(K)\\ t
& \mapsto & \phi_t \end{array}$$
 is a surjective homomorphism with kernel equal to $\{1,-1\}$. \\

We can now characterize a preimage of a semisimple element of ${\rm SO}_n(K)$ of a given order under the canonical surjective map ${\rm Spin}_n(K)\rightarrow {\rm SO}_n(K)$, where $n\geq 7$ is a positive integer and $K$ is an
algebraically closed field of characteristic $p \neq 2$.

\begin{lem}\label{l:soliftspin}
Let $g$ be an element of ${\rm SO}_{n}(K)$  defined over an
algebraically closed field $K$ of  characteristic $p\neq2$. 
Suppose $g$ is a semisimple element of ${\rm SO}_n(K)$ of order $u$,  and let $\omega$ be a $u$-th root of unity of $K$. Let $\pm w$ be
any of the two preimages of $g$ under the canonical surjective map
$\phi: {\rm Spin}_n(K)\rightarrow {\rm SO}_{n}(K)$. The following
assertions hold.
\begin{enumerate}[(i)]
\item The order of $w$ in ${\rm Spin}_n(K)$ is divisible by $u$.
\item There exits $t \in \{\pm w\}$ of order $u$ if and only if  $u$
is odd, or $u$ is even and  the number of eigenvalues of $\phi(t)$ of
the form $\omega^i$ with $i$ odd is divisible by $4$.
\item If neither $w$ nor $-w$ has  order $u$, then $u$ is even and
$\pm w$ has order $2u$.
\end{enumerate}
\end{lem}

\begin{proof}
This is a classical result which follows from the properties of the
canonical surjective map $\phi: {\rm Spin}_n(K)\rightarrow {\rm
SO}_n(K)$ described above.  We give a sketch of the argument. \\
We first consider part (i). Let ${g}$ be any semisimple element of ${\rm
SO}_n(K)$ of order $u$.
Let $w$ be a preimage of $g$ under $\phi$ (so that $-w$ is the other
preimage of $g$ under $\phi$) and let $m$ be the order of $w$. We have
$g^m=\phi(w)^m=\phi(w^m)=\phi(1)=1$. Hence $u$, which is the order of
$g$, divides the order of a preimage of $g$ under $\phi$. This
establishes part (i).\\

 We now consider parts (ii) and (iii). Let $V$ be the natural module
for ${\rm SO}_n(K)$. There is a standard basis $\mathcal{B}$ of $V$,
where $\mathcal{B}=(e_1,f_1,\dots,e_\ell,f_\ell)$ or $(e_1,f_1,\dots,
e_\ell,f_\ell,d)$ according respectively as $n$ is even or odd, such
that the matrix of ${g}$ with respect to $\mathcal{B}$ is diagonal
with entries $d_i$ satisfying $d_i^u=1$ for $1 \leq i \leq n$. Note
that  if $1\leq j \leq \ell$ then $d_{2j-1}=d_{2j}^{-1}$ is a power of
$\omega$, and if $n$ is odd then $d_n=1$. Moreover without loss of
generality $d_1=\omega$ and $d_2=\omega^{-1}$.  For $1\leq j \leq
\ell$, let ${g_j}$ be the element of ${\rm SO}_n(K)$ such that,  with
respect to $\mathcal{B}$, ${g_j}$ is diagonal and
$({g_j})_{2j-1,2j-1}=d_j$, $({g_j})_{2j,2j}=d_{j+1}$,
$({g_j})_{k,k}=1$ for $k\not\in\{2j-1,2j\}$. Note that
${g}={g_1}\dots{g_\ell}$. In the decomposition
${g}={g_1}\dots{g_\ell}$ delete the elements ${g_j}$ with $1\leq i
\leq \ell$ such that ${g_j}=1$, and renumbering the elements ${g_j}$
where $1\leq j\leq \ell$, if necessary, write
${g}={g_1}\dots{g_\iota}$ where $1\leq \iota \leq \ell$ and
${g_1},\dots, {g_\iota}$ are not the identity element. \\
Let $1\leq j \leq \iota$ and let $1\leq k\leq \iota$ be such that
$({g_j})_{2k-1,2k-1}$ and $({g_j})_{2k,2k}$ are both not 1.  Write
$\theta_j=({g_j})_{2k-1,2k-1}$. Note that $\theta_j\neq 1$ is a power
of $\omega$ and $\theta_j^{-1}=({g_j})_{2k,2k}$. Moreover
${g_j}(e_k)=\theta_j e_k$,  ${g_j}(f_k)=\theta_j^{-1} f_k$ and ${g_j}$
fixes every other element of $\mathcal{B}$. Write
$\theta_j=\omega^{l_j}$ for some positive integer $l_j$. \\
For a nonsingular element $x$ in $V$, let $R_x$ be the reflection:
$$\begin{array}{lll}V & \rightarrow&  V \\ y &\mapsto
&y-\frac{2(y,x)}{(x,x)}x.\end{array}$$  An easy check yields
$${g_j}=R_{(1-\theta_j)e_k+(\theta_j^{-1}-1)f_k}R_{2e_k+2f_k}.$$
Let $$x_j=((1-\theta_j)e_k+(\theta_j^{-1}-1)f_k)(2e_k+2f_k).$$ Then
$x_j\in C_0(V,K)$ and $x_jx_j^*=16(1-\theta_j)^2\theta_j^{-1}$. Let
$\tau_j \in K$ be such that $\tau_j^2=\theta_j^{-1}=\omega^{-l_j}$,
furthermore if $u$ and $l_j$ are both odd take
$\tau_j=\omega^{(u-l_j)/2}$, and if $l_j$ is even take
$\tau_j=\omega^{-l_j/2}$. Let $$t_j=\frac{x_j}{4(1-\theta_j)\tau_j}.$$
Then $t_j$ is an element of ${\rm Spin}_n(K)$ and the image of $t_j$
under the  canonical surjective map $\phi: {\rm Spin}_n(K)\rightarrow
{\rm SO}_n(K)$ is ${g_j}$. Furthermore, for any positive integer $m$,
$$t_j^m=\frac{1}{2\tau_j^m}(e_kf_k+\theta_j^{-m}(2-e_kf_k))$$ and so
\begin{equation}\label{e:orderpreso}
t_j^u=\frac1{\tau_j^u}=\left\{\begin{array}{ll}
1 & \textrm{if} \ l_j \ \textrm{is even, or} \ u \ \textrm{and} \
l_j\ \textrm{are both odd} \\
-1 & \textrm{if} \ u \ \textrm{is even and} \ l_j \ \textrm{is
odd}.\end{array} \right.\end{equation}
Finally, let $t=t_1\dots t_\iota$. Then $t \in {\rm Spin}_n(K)$ and
$\phi(t)={g}={g_1}\dots{g_\iota}$. Now for $1\leq i, j \leq \iota$,
$t_i$ and $t_j$ commute and so
$$t^u=(t_1\dots t_\iota)^u=t_1^u\dots t_\iota^u= \prod_{j=1}^\iota
\frac{1}{\tau_j^u}.$$
It follows from (\ref{e:orderpreso}) that $t^u=1$ if and only if  $u$
is odd or $u$ is even and  the number of eigenvalues of $\phi(t)=g$ of
the form $\omega^i$ with $i$ odd is divisible by $4$. Otherwise
$t^u=-1$ and $t^{2u}=1$.
Note that the other preimage of $g$ under $\phi$ is $-t$ and
$(-t)^u=t^u$ if $u$ is even, otherwise $(-t)^u=-t^u$.  Parts (ii) and
(iii) now follow from part (i).
 \end{proof}

  \section{Upper bounds for $d_u(G)$ for $G$ of classical type}\label{s:upc}
Let $G$ be a simple algebraic group of classical type defined over an algebraically closed field $K$ of characteristic $p$. Let $u$ be a positive integer. In this section, we give  an upper bound for $d_u(G)$. 
Since by Proposition \ref{p:ingsimple}  $d_u(G)\leq d_u(G_{s.c.})$ we assume without loss of generality that $G$ is of simply connected type. \\

Unless otherwise stated, we use the notation introduced at the end of \S\ref{s:intro}. Recall, we let $h$ be the Coxeter number of $G$ and write:\\
  $u=qv$ where $q$ and $v$ are coprime positive integers such that $q$ is a power of $p$  and $p$ does not divide $v$,\\
  $h=zu+e$ where $z,e$ are nonnegative integers such that $0\leq e \leq q-1$,\\
  $h=\alpha v+\beta$ where $\alpha,\beta$ are nonnegative integers such that $0\leq \alpha \leq v-1$,\\
  $\alpha=\gamma q+\delta$ where $\gamma,\delta$ are nonnegative integers such that $0\leq \delta \leq q-1$.\\
   An easy check yields $\gamma=z$.  Also if $e=0$ then $\beta=\delta=0$. Furthermore, if $\beta =0$ then $\delta=0$ if and only if $e=0$.\\
Finally for an nonnegative integer $r$, we let $\epsilon_r \in \{0,1\}$ be $0$ if $r$ is even, otherwise $\epsilon_r=1$, and set $\sigma_r\in\{0,1\}$ to be 1 if $r=0$, otherwise $\sigma_r=0$.\\

 Moreover, we also let $\omega$ be a $v$-th root of 1 in $K$ and consider the following diagonal blocks $M_{1}$, $M_{2}$, $M_3$, $M_4$, $M_5$, ${M_6}$, $M_7$, $M_8$ and $M_9$ where:
$$M_{1}={\rm diag}(1, \omega, \omega^2,\dots, \omega^{v-1}) \quad \textrm{is of size} \ v,$$
$$M_{2}={\rm diag}(\omega, \omega^{-1}, \omega^2, \omega^{-2}, \dots, \omega^{\lfloor\beta/2\rfloor}, \omega^{-\lfloor  \beta/2\rfloor}) \quad \textrm{is of size} \  \beta-\epsilon_{\beta},$$
 $$M_{3}={\rm diag}(\omega, \omega^{-1}, \omega^2, \omega^{-2}, \dots, \omega^{\lceil\beta/2\rceil}, \omega^{-\lceil  \beta/2\rceil}) \quad \textrm{is of size} \  \beta+\epsilon_{\beta},$$
$$M_{4}={\rm diag}(\omega, \omega^2,\dots, \omega^{v-1}) \quad \textrm{is of size} \ v-1.$$
and if $\beta \geq 2$ is even then $$ M_{5}={\rm diag}(\omega, \omega^{-1}, \dots, \omega^{\beta/2-1},\omega^{-(\beta/2-1)}) \quad \textrm{is of size} \ \beta-2$$
$$ M_{6}={\rm diag}(\omega^2,\omega^{-2}, \dots, \omega^{\frac{\beta}2},\omega^{-\frac{\beta}2}) \quad \textrm{is of size} \ \beta -2,$$
and if $v \geq 4$ is even then
$$ M_{7}={\rm diag}(\omega^2,\omega^{-2}, \dots, \omega^{\frac{v}2-1},\omega^{-(\frac{v}2-1)}) \quad \textrm{is of size} \ v -4,$$
$$ M_{8}={\rm diag}(\omega,\omega^{-1},\omega^2,\omega^{-2}, \dots,  \omega^{\frac{v}2-3},\omega^{-(\frac{v}2-3)},\omega^{\frac{v}2-1},\omega^{-(\frac{v}2-1)}) \quad \textrm{is of size} \ v -4,$$
$$ M_9 ={\rm diag}(\omega,\omega^{-1},\omega^2, \omega^{-2}, \dots, \omega^{\frac v2-2}, \omega^{-(\frac v2-2)}) \quad \textrm{is of size} \ v-4.$$
We also write $(1)$ for $I_1={\rm diag}(1)$ and $(-1)$ for $-I_1={\rm diag}(-1)$. More generally given a positive integer $i$, we write $(\omega^i)$ for ${\rm diag}(\omega^i)$  Finally, given some nonnegative integers $r_1,\dots, r_{11}$ and some nonnegative integers $s_1,\dots, s_{v-1}$we denote by $$ M_{1}^{r_1}\oplus M_2^{r_2}\oplus M_{3}^{r_3}\oplus M_{4}^{r_4}\oplus M_{5}^{r_5}\oplus M_{6}^{r_6}\oplus M_{7}^{r_7}\oplus M_{8}^{r_8}\oplus M_{9}^{r_9}\oplus (1)^{r_{10}}\oplus (-1)^{r_{11}}\oplus_{j=1}^{v-1} (\omega^j)^{s_j}$$ the diagonal matrix consisting of $r_i$ blocks $M_{i}$ for $1\leq i\leq 9$,  $I_{r_{10}}$, $-I_{r_{11}}$ and $\omega^jI_{s_j}$ for $1\leq j \leq v-1$.

\subsection{$G$ is  of type $A_\ell$}

Let $G=(A_\ell)_{s.c.}$  with $\ell \geq 1$ defined over an algebraically closed field $K$ of characteristic $p$. Here $|\Phi|=\ell(\ell+1)$ and $h=\ell+1$.

\begin{lem}\label{l:asc}
Let  $G=(A_{\ell})_{s.c.}$ be defined over an algebraically closed field of  characteristic $p$. Let $u=qv$ be a positive integer where $q$ and $v$ are coprime positive integers such that $q$ is a power of $p$  and $p$ does not divide $v$. Write $h=zu+e=\alpha v +\beta$ and $\alpha=zq+\delta$ where $z$, $e$, $\alpha$, $\beta$, $\delta$ are nonnegative integers such that $e<u$, $\beta<v$, and $\delta<q$. 
Let $y$ be the element of  $G_{s.c.}$ of order $v$ defined in Table \ref{t:asc} (see \S\ref{s:tables}). Then 
$C_{G}(y)^0$ and $d_q(C_{G}(y)^0)$ are given in Table \ref{t:asc} and $d_q(C_{G}(y)^0)$ is an upper bound for $d_u(G)$.
\end{lem} 

\begin{proof}
We first determine $C_G(y)^0$. 
Suppose that $\epsilon_v=1$ or $(\epsilon_v,\epsilon_\alpha)=(0,0)$. As any nontrivial eigenvalue of $y$ can be paired with its inverse, $y$ is an element of $G$ of order $v$. Also $y$ has $\alpha+\epsilon_\beta$ eigenvalues equal to $1$,  $\alpha+1$ eigenvalues equal to $\omega^i$ where $i$ is any integer with $1 \leq |i| \leq \lfloor \beta/2\rfloor$, and every other eigenvalue of $y$ occurs with multiplicity $\alpha$. Therefore $C_G(y)^0=A_{\alpha}^\beta A_{\alpha-1}^{v-\beta}T_{v-1}$. \\
Suppose that $(\epsilon_{v},\epsilon_{\alpha}, \epsilon_{\beta})=(0,1,1)$. As any  nontrivial eigenvalue of $y$ can be paired with its inverse,  $y$ is an element of $G$ of order $v$. Also $y$ has $\alpha+1$ eigenvalues equal to $-1$, $\alpha+1$ eigenvalues equal to $\omega^i$ where $i$ is any integer with $1 \leq |i| \leq  (\beta-1)/2$, and every other eigenvalue of $y$ occurs with multiplicity $\alpha$. Therefore $C_G(y)^0=A_{\alpha}^\beta A_{\alpha-1}^{v-\beta}T_{v-1}$. \\
Suppose that $(\epsilon_{v},\epsilon_{\alpha}, \epsilon_{\beta})=(0,1,0)$ and $\beta \geq 2$. As any eigenvalue of $y$ can be paired with its inverse, $y$ is an element of $G$ of order $v$. Also $y$ has $\alpha+1$ eigenvalues equal to $1$, $\alpha+1$ eigenvalues equal to $-1$, $\alpha+1$ eigenvalues equal to $\omega^i$ where $i$ is any  integer with $1\leq |i| \leq  (\beta-2)/2$, and every other eigenvalue of $y$ occurs with multiplicity $\alpha$. Therefore $C_G(y)^0=A_{\alpha}^\beta A_{\alpha-1}^{v-\beta}T_{v-1}$. \\
Suppose finally that $(\epsilon_v,\epsilon_\alpha)=(0,1)$ and $\beta=0$. As any eigenvalue of $y$ can be paired with its inverse,  $y$ is an element of $G$ of order $v$. Also $y$ has $\alpha-1$ eigenvalues equal to $1$, $\alpha+1$ eigenvalues equal to $-1$,  and every other eigenvalue of $y$ occurs with multiplicity $\alpha$. Therefore $C_G(y)^0=A_{\alpha} A_{\alpha-1}^{v-2}A_{\alpha-2}T_{v-1}$. \\

Note that by Proposition \ref{p:prelimdu}(i), $d_u(G)\leq d_q(C_G(y)^0)$ and so $d_q(C_G(y)^0)$ is an upper bound for $d_u(G)$. It remains to calculate $d_q(C_G(y)^0)$.\\

By Lemma \ref{l:lawther}, \begin{equation}\label{e:dqa}d_q(A_{\alpha}) = z^2q +(2z+1)(\alpha-zq+1)-1,\end{equation}
 \begin{equation}\label{e:dqam1}d_q(A_{\alpha-1})=z^2q+(2z+1)(\alpha-zq)-1,\end{equation}
 and \begin{equation}\label{e:dqam2}d_q(A_{\alpha-2})=\left \{\begin{array}{ll}  z^2q+(2z+1)(\alpha-1-zq)-1 & \textrm{if}\ \delta>0\\
 (z-1)^2q+(2z-1)(\alpha-1 -(z-1)q)-1 & \textrm{if} \ \delta=0,
  \end{array}\right.\end{equation}
where for determining $d_q(A_{\alpha-2})$ we do the Euclidean division of $\alpha-1$ by $q$:
$$\alpha-1=\left\{\begin{array}{ll} zq+(\delta-1) & \textrm{if} \ \delta>0\\ (z-1)q+q-1 & \textrm{if} \ \delta=0.\end{array}\right.$$

Assume first that  $(\epsilon_v,\epsilon_\alpha)\neq (0,1)$ or $\beta>0$. We then have $$C_G(y)^0=A_{\alpha}^\beta A_{\alpha-1}^{v-\beta}T_{v-1}.$$
By Proposition \ref{p:prelimdu}(iii), (\ref{e:dqa}) and (\ref{e:dqam1})
\begin{eqnarray*}
d_q(C_G(y)^0) & = & (v-1)+d_q(A_{\alpha}^\beta A_{\alpha-1}^{v-\beta})\\
& = & (v-1)+\beta d_q(A_\alpha) +(v-\beta)d_q(A_{\alpha-1})\\
& = &  2z(\alpha v+\beta)-qvz(z+1)+\alpha v+\beta-1\\
& = & 2z(zu+e)-zu(z+1)+zu+e-1\\
& = &  z^2u+e(2z+1)-1.
\end{eqnarray*}

Assume now that $(\epsilon_v, \epsilon_\alpha)=(0,1)$ and $\beta=0$. Since $\beta=0$, recall that $e=0$ if and only if $\delta=0$.  We then have $$C_G(y)^0= A_\alpha A_{\alpha-1}^{v-2} A_{\alpha-2}T_{v-1}.$$ 
By Proposition \ref{p:prelimdu}(iii),  (\ref{e:dqa}), (\ref{e:dqam1}) and (\ref{e:dqam2})
\begin{eqnarray*}
d_q(C_G(y)^0) & = & (v-1)+d_q(A_\alpha A_{\alpha-1}^{v-2} A_{\alpha-2})\\
& = & (v-1)+d_q(A_\alpha)+(v-2) d_q(A_{\alpha-1}) +d_q(A_{\alpha-2})\\
& =  & \left\{ \begin{array}{ll} 2z\alpha v-qvz(z+1)+\alpha v-1 & \textrm{if} \ \delta>0\\
2z\alpha v-qvz(z+1)+\alpha v+1& \textrm{if} \ \delta=0
 \end{array}\right.\\
& = &\left\{ \begin{array}{ll} 2z(zu+e)-zu(z+1)+zu+e-1 &  \textrm{if} \ \delta>0\\ 
2z(zu+e)-zu(z+1)+zu+e+1 &  \textrm{if} \ \delta=0 \end{array}\right.\\
& = & \left\{ \begin{array}{ll} z^2u+e(2z+1)-1 &  \textrm{if} \  e>0\\ 
z^2u+e(2z+1)+1&  \textrm{if} \ e=0. \end{array}\right.\\
\end{eqnarray*}
\end{proof}

\subsection{$G$ is of type $C_\ell$}

Let $G=(C_\ell)_{{\rm s.c}}$ be a simply connected group of type ${C}_{\ell}$ with $\ell \geq 2$ defined over an algebraically closed field $K$ of  characteristic $p$. Here $|\Phi|=2\ell^2$ and $h=2\ell$. 

\begin{lem}\label{l:csc}
Let  $G=(C_{\ell})_{{\rm s.c}}$ be defined over an algebraically closed field of characteristic $p$. 
Let $u=qv$ be a positive integer where $q$ and $v$ are positive integers such that $q$ is a power of $p$. Write $h=zu+e=\alpha v +\beta$ and $\alpha=zq+\delta$ where $z$, $e$, $\alpha$, $\beta$, $\delta$ are nonnegative integers such that $e<u$, $\beta<v$, and $\delta<q$. 
Let $y$ be the element of  $G$ of order $v$ defined in Table \ref{t:csc} (see \S\ref{s:tables}). Then 
$C_G(y)^0$ and $d_q(C_G(y)^0)$ are given in Table \ref{t:csc} and $d_q(C_G(y)^0)$ is an upper bound for $d_u(G)$.
\end{lem} 

\begin{proof}
Since $h=2\ell$, note that if $\epsilon_v=1$ then $\epsilon_\alpha=\epsilon_\beta$ and $\epsilon_u=\epsilon_q$, whereas if $\epsilon_v=0$ then $\epsilon_u=0$, $\epsilon_q=1$ and $\epsilon_\beta=0$. We first determine $C_G(y)^0$. 
Suppose that $\epsilon_v=1$.  As any eigenvalue of $y$ can be paired with its inverse,  $y$ is an element of $G$ of order $v$. Also $y$ has $\alpha+\epsilon_\beta$ eigenvalues equal to $1$, $0$ eigenvalues equal to $-1$, $\alpha+1$ eigenvalues equal to $\omega^i$ where $i$ is any integer with $1 \leq |i| \leq \lfloor \beta/2\rfloor$, and every other eigenvalue of $y$ occurs with multiplicity $\alpha$. Therefore $C_G(y)^0=A_\alpha^{\lfloor\frac\beta2\rfloor} A_{\alpha-1}^{\frac{v-1}2-\lfloor\frac\beta2\rfloor}C_{\lceil \frac \alpha2\rceil}T_{\frac{v-1}2}$. \\
Suppose that $(\epsilon_{v},\epsilon_{\alpha})=(0,0)$. As any eigenvalue of $y$ can be paired with its inverse,  $y$ is an element of $G$ of order $v$. Also $y$ has $\alpha$ eigenvalues equal to $1$, $\alpha$ eigenvalues equal to $-1$, $\alpha+1$ eigenvalues equal to $\omega^i$ where $i$ is any integer with $1  \leq |i| \leq  \beta/2$, and every other eigenvalue of $y$ occurs with multiplicity $\alpha$. Therefore $C_G(y)^0=A_\alpha^{\frac\beta2} A_{\alpha-1}^{\frac{v}2-1-\frac\beta2}C_{\frac \alpha2}^2T_{\frac{v}2-1}$. \\
Suppose that $(\epsilon_{v},\epsilon_{\alpha})=(0,1)$ and $\beta \geq 2$. As any eigenvalue of $y$ can be paired with its inverse,  $y$ is an element of $G$ of order $v$. Also $y$ has $\alpha+1$ eigenvalues equal to $1$, $\alpha+1$ eigenvalues equal to $-1$, $\alpha+1$ eigenvalues equal to $\omega^i$ where $i$ is any integer with $1 \leq |i| \leq  (\beta-2)/2$, and every other eigenvalue of $y$ occurs with multiplicity $\alpha$. Therefore $C_G(y)^0=A_\alpha^{\frac\beta2-1} A_{\alpha-1}^{\frac{v}2-\frac\beta2}C_{\frac{\alpha+1}2}^2T_{\frac{v}2-1}$. \\
Suppose finally that $(\epsilon_v,\epsilon_\alpha)=(0,1)$ and $\beta=0$. As any eigenvalue of $y$ can be paired with its inverse,  $y$ is an element of $G$ of order $v$. Also $y$ has $\alpha-1$ eigenvalues equal to $1$, $\alpha+1$ eigenvalues equal to $-1$,  and every other eigenvalue of $y$ occurs with multiplicity $\alpha$. Therefore $C_G(y)^0=A_{\alpha-1}^{\frac{v}2-1}C_{\frac{\alpha+1}2}C_{\frac{\alpha-1}2}T_{\frac{v}2-1}$. \\

By Proposition \ref{p:prelimdu}(i), $d_u(G)\leq d_q(C_G(y)^0)$ and so $d_q(C_G(y)^0)$ is an upper bound for $d_u(G)$. It remains to calculate $d_q(C_G(y)^0)$.\\

Note that by Lemma \ref{l:lawther}, $$d_q(A_{\alpha}) = z^2q +(2z+1)(\alpha-zq+1)-1,$$
$$d_q(A_{\alpha-1})=z^2q+(2z+1)(\alpha-zq)-1,$$
$$d_q\left(C_{\lceil\frac{\alpha}2\rceil}\right)=\frac{z^2q}{2}+(2z+1)\left(\left\lceil\frac{\alpha}{2}\right\rceil-\frac{zq}2\right)+\left\lceil \frac z2\right\rceil\epsilon_q$$
and if $\alpha$ is odd then 
$$d_q\left(C_{\frac{\alpha-1}2}\right)=\frac{z^2q}{2}+(2z+1)\left(\frac{\alpha-1}{2}-\frac{zq}{2}\right)+\left\lceil \frac{z}{2}\right\rceil\epsilon_q.$$

Assume first that  $\epsilon_v=1$. Recall that  $\epsilon_u=\epsilon_q$ and $\epsilon_\alpha=\epsilon_\beta$. We  have $$C_G(y)^0=A_\alpha^{\lfloor\frac\beta2\rfloor} A_{\alpha-1}^{\frac{v-1}2-\lfloor\frac\beta2\rfloor}C_{\lceil \frac \alpha2\rceil}T_{\frac{v-1}2}.$$
By Proposition \ref{p:prelimdu}(iii),
\begin{eqnarray*}
d_q(C_G(y)^0) & = & \frac{v-1}2+d_q\left(A_\alpha^{\lfloor\frac\beta2\rfloor} A_{\alpha-1}^{\frac{v-1}2-\lfloor\frac\beta2\rfloor}C_{\lceil \frac \alpha2\rceil}\right)\\
& = & \frac{v-1}2+\left\lfloor \frac \beta2 \right \rfloor d_q(A_\alpha)+\left(\frac{v-1}2-\left \lfloor \frac \beta2 \right \rfloor\right)d_q(A_{\alpha-1})+d_q\left(C_{\lceil \frac \alpha2\rceil}\right)\\
& =  &z(\alpha v+ \beta)-\frac{qvz(z+1)}2+\frac{\alpha v +\beta}2+\left\lceil\frac{z}2\right\rceil\epsilon_q\\
& = & z(zu+e)-\frac{zu(z+1)}2+\frac{zu+e}2+\left\lceil\frac{z}2\right\rceil\epsilon_q\\
& = &  \frac12(z^2u+e(2z+1))+\left\lceil\frac{z}2\right\rceil\epsilon_u.
\end{eqnarray*}

Assume that  $(\epsilon_v,\epsilon_\alpha)=(0,0)$. Recall that $\epsilon_q=1$ and $\epsilon_\beta=0$. We  have $$C_G(y)^0=A_\alpha^{\frac\beta2} A_{\alpha-1}^{\frac{v}2-1-\frac\beta2}C_{\frac \alpha2}^2T_{\frac{v}2-1}.$$
By Proposition \ref{p:prelimdu}(iii),
\begin{eqnarray*}
d_q(C_G(y)^0) & = & \frac{v}2-1+d_q\left(A_\alpha^{\frac\beta2} A_{\alpha-1}^{\frac{v}2-1-\frac\beta2}C_{\frac \alpha2}^2\right) \\
& = & \frac{v}2-1+ \frac \beta2 d_q(A_\alpha)+\left(\frac{v}2-1- \frac \beta2 \right)d_q(A_{\alpha-1})+2d_q\left(C_{ \frac \alpha2}\right)\\
& =  &z(\alpha v+ \beta)-\frac{qvz(z+1)}2+\frac{\alpha v +\beta}2+2\left\lceil\frac{z}2\right\rceil\epsilon_q\\
& = & z(zu+e)-\frac{zu(z+1)}2+\frac{zu+e}2+2\left\lceil\frac{z}2\right\rceil\epsilon_q\\
& = &  \frac12(z^2u+e(2z+1))+2\left\lceil\frac{z}2\right\rceil.
\end{eqnarray*}

Assume that  $(\epsilon_v,\epsilon_\alpha)=(0,1)$ and $\beta\geq 2$. Recall that $\epsilon_q=1$ and $\epsilon_\beta=0$. We  have $$C_G(y)^0=A_\alpha^{\frac\beta2-1} A_{\alpha-1}^{\frac{v}2-\frac\beta2}C_{\frac{\alpha+1}2}^2T_{\frac{v}2-1}.$$
By Proposition \ref{p:prelimdu}(iii),
\begin{eqnarray*}
d_q(C_G(y)^0) & = & \frac{v}2-1+d_q\left(A_\alpha^{\frac\beta2-1} A_{\alpha-1}^{\frac{v}2-\frac\beta2}C_{\frac{\alpha+1}2}^2\right) \\
& = & \frac{v}2-1+\left( \frac \beta2-1\right) d_q(A_\alpha)+\left(\frac{v}2- \frac \beta2 \right)d_q(A_{\alpha-1})+2d_q\left(C_{ \frac {\alpha+1}2}\right)\\
& =  &z(\alpha v+ \beta)-\frac{qvz(z+1)}2+\frac{\alpha v +\beta}2+2\left\lceil\frac{z}2\right\rceil\epsilon_q\\
& = & z(zu+e)-\frac{zu(z+1)}2+\frac{zu+e}2+2\left\lceil\frac{z}2\right\rceil\epsilon_q\\
& = &  \frac12(z^2u+e(2z+1))+2\left\lceil\frac{z}2\right\rceil.
\end{eqnarray*}

Assume finally that  $(\epsilon_v,\epsilon_\alpha)=(0,1)$ and $\beta=0$. Recall that $\epsilon_q=1$. We  have $$C_G(y)^0=A_{\alpha-1}^{\frac{v}2-1}C_{\frac{\alpha+1}2}C_{\frac{\alpha-1}2}T_{\frac{v}2-1}.$$
By Proposition \ref{p:prelimdu}(iii),
\begin{eqnarray*}
d_q(C_G(y)^0) & = & \frac{v}2-1+d_q\left(A_{\alpha-1}^{\frac v2-1} C_{\frac{\alpha+1}2}C_{\frac{\alpha-1}2}\right) \\
& = & \frac{v}2-1+ \left(\frac v2-1\right) d_q(A_{\alpha-1})+d_q\left(C_{ \frac {\alpha+1}2}\right)+d_q\left(C_{ \frac {\alpha-1}2}\right)\\
& =  &z(\alpha v)-\frac{qvz(z+1)}2+\frac{\alpha v}2+2\left\lceil\frac{z}2\right\rceil\epsilon_q\\
& = & z(zu+e)-\frac{zu(z+1)}2+\frac{zu+e}2+2\left\lceil\frac{z}2\right\rceil\epsilon_q\\
& = &  \frac12(z^2u+e(2z+1))+2\left\lceil\frac{z}2\right\rceil.
\end{eqnarray*}
\end{proof}

\subsection{$G$ is of type $B_\ell$}

Let $G=(B_\ell)_{{\rm s.c}}$ be a simply connected group of type ${B}_{\ell}$ with $\ell \geq 3$ defined over an algebraically closed field $K$ of  characteristic $p$. Here $|\Phi|=2\ell^2$, $h=2\ell$ and $G={\rm Spin}_{h+1}(K)$ if $p\neq 2$, otherwise $G={\rm SO}_{h+1}(K)$. Given an element $g$ in $G$, we let $\overline{g}$ be the image of $g$ under the canonical surjective map $G\rightarrow {\rm SO}_{h+1}(K)$.  \\

\begin{lem}\label{l:bsc}
Let  $G=(B_{\ell})_{{\rm s.c}}$ where $\ell \geq 3$ be defined over an algebraically closed field $K$ of  characteristic $p$. Let $u=qv$ be a positive integer where $q$ and $v$ are positive integers such that $q$ is a power of $p$ and $p$ does not divide $v$. Write $h=zu+e=\alpha v +\beta$ and $\alpha=zq+\delta$ where $z$, $e$, $\alpha$, $\beta$, $\delta$ are nonnegative integers such that $e<u$, $\beta<v$, and $\delta<q$. Consider the following conditions:\\
\begin{enumerate}[(a)]
\item $\epsilon_v=1$. 
\item $v=2$. 
\item $(\epsilon_v,\epsilon_\alpha)=(0,1)$, $v>2$ and $(v \mod 4,\beta \mod 4,\ell \mod 4) \in \{(0,0,0), (0,2,3), (2,0,3), (2,2,2)\}.$
\item $(\epsilon_v,\epsilon_{\alpha})=(0,1)$ and $(v \mod 4,\beta \mod 4,\ell \mod 4) \in \{(0,2,1), (2,2,0)\}.$
\item $(\epsilon_v,\epsilon_{\alpha})=(0,1)$, $\beta >0$ and $(v \mod 4,\beta \mod 4,\ell \mod 4) \in \{(0,0,2), (2,0,1)\}.$
\item $(\epsilon_v,\epsilon_{\alpha})=(0,1)$, $v>2$, $\beta=0$ and $(v \mod 4,\ell \mod 4) \in \{(0,2), (2,1)\}.$
\item $(\epsilon_v,\epsilon_\alpha)=(0,0)$, $v>2$ and $(v \mod 4,\beta \mod 4,\ell \mod 4) \in \{(0,0,0), (0,2,3), (2,0,0), (2,2,3)\}.$
\item  $(\epsilon_v,\epsilon_{\alpha})=(0,0)$ and  $(v \mod 4,\beta \mod 4,\ell \mod 4) =(2,2,1)$.
\item $(\epsilon_v,\epsilon_{\alpha})=(0,0)$, $\beta>0$ and $(v \mod 4,\beta \mod 4,\ell \mod 4) \in \{(0,2,1),(2,0,2)\}.$
\item $(\epsilon_v,\epsilon_{\alpha})=(0,0)$, $\beta>0$ and  $(v \mod 4,\beta \mod 4,\ell \mod 4) \in (0,0,2)$.
\item $(\epsilon_v,\epsilon_{\alpha})=(0,0)$,  $v>2$, $\beta=0$ and $(v \mod 4,\ell \mod 4)= (2,2)$. 
\end{enumerate}
Let $\overline{y}$ be the element of  ${\rm SO}_{h+1}(K)$ of order $v$ defined in Table \ref{t:bsc} (see \S\ref{s:tables}). Then $\overline{y}$ has a preimage $y$ in $G$ of order $v$, and 
$C_{{\rm SO}_{h+1}}(\overline{y})^0$ and $d_q(C_G(y)^0)=d_q(C_{{\rm SO}_{h+1}(K)}(\overline{y})^0)$ are given in Table \ref{t:bsc}. Furthermore $d_q(C_G(y)^0)$ is an upper bound for $d_u(G)$.
\end{lem}

\begin{proof}
Note that by Lemma \ref{l:dufactg}(iii), $d_q(C_G(y)^0)=d_q(C_{{\rm SO}_{h+1}(K)}(\overline{y})^0)$. 
Since $h=2\ell$, note that if $\epsilon_v=1$ then $\epsilon_\alpha=\epsilon_\beta$ and $\epsilon_u=\epsilon_q$, whereas if $\epsilon_v=0$ then $\epsilon_u=0$, $\epsilon_q=1$ and $\epsilon_\beta=0$. Note also that if $\ell \equiv 2 \mod 4$, $v \equiv 0\mod 4$ and $\beta=0$ then  $\alpha$ is odd, as otherwise on one hand we would have $2\ell \equiv 4 \mod 8$ and on the other $2\ell = \alpha v \equiv 0 \mod 8$, a contradiction.  We first determine $C_{{\rm SO}_{h+1}(K)}(\overline{y})^0$. 
Suppose that case (a) holds so that $v$ is odd.  As any eigenvalue of $\overline{y}$ different from 1 can be paired with its inverse,  $\overline{y}$ is an element of ${\rm SO}_{h+1}(K)$ of order $v$.  Furthermore, since $v$ is odd, by Lemma \ref{l:soliftspin} it follows that $\overline{y}$ has a preimage $y$ in $G$ of order $v$. Also $\overline{y}$ has $\alpha+1+\epsilon_\beta$ eigenvalues equal to $1$, $0$ eigenvalues equal to $-1$, $\alpha+1$ eigenvalues equal to $\omega^i$ where $i$ is any integer with $1 \leq |i| \leq \lfloor \beta/2\rfloor$, and every other eigenvalue of $\overline{y}$ occurs with multiplicity $\alpha$. Therefore $C_{{\rm SO}_{h+1}(K)}(\overline{y})^0=A_\alpha^{\lfloor\frac\beta2\rfloor} A_{\alpha-1}^{\frac{v-1}2-\lfloor\frac\beta2\rfloor}B_{\lceil \frac {\alpha}2\rceil}T_{\frac{v-1}2}$. \\
Suppose that case (b) holds so that $v=2$.  As any eigenvalue of $\overline{y}$ different from 1 can be paired with its inverse,  $\overline{y}$ is an element of ${\rm SO}_{h+1}(K)$ of order $v$. Now the number $N$ of eigenvalues $\omega^i=(-1)^i$ of $\overline{y}$ with $i$ odd  is equal  to $\ell$, $\ell-1$, $\ell+2$ or $\ell +1$ according respectively as $\ell \mod 4$ is  0, 1, 2 or  3. Since $N$ is divisible by 4, it follows from Lemma $\ref{l:soliftspin}$ that $\overline{y}$ has a preimage $y$ in $G$ of order $v$.  Also 
$$ C_{{\rm SO}_{h+1}(K)}(\overline{y})^0=\left\{\begin{array}{ll}B_{\frac{\ell}2}D_{\frac{\ell}2} & \textrm{if} \ \ell \equiv 0 \ (4)\\ 
 B_{\frac{\ell+1}2}D_{\frac{\ell-1}2} & \textrm{if} \ \ell \equiv 1 \ (4)\\ 
B_{\frac{\ell-2}2}D_{\frac{\ell+2}2} & \textrm{if} \ \ell \equiv 2 \ (4)\\ 
B_{\frac{\ell-1}2}D_{\frac{\ell+1}2} & \textrm{if} \ \ell \equiv 3 \ (4).
 \end{array}\right.$$
 Suppose that case (c) holds.  As any eigenvalue of $\overline{y}$ different from 1 can be paired with its inverse,  $\overline{y}$ is an element of ${\rm SO}_{h+1}(K)$ of order $v$. Let $N$ be the  number of eigenvalues $\omega^i$ of $\overline{y}$ with $i$ odd. Then 
 \begin{eqnarray*}
 N& = &\left\{ \begin{array}{ll} \alpha  \cdot \frac{v}{2}+\frac{\beta}2& \textrm{if} \ (v \mod 4, \beta \mod 4)=(0,0)\\
 \alpha  \cdot \frac{v}{2}+\left(\frac{\beta}2+1\right)  & \textrm{if} \ (v \mod 4, \beta \mod 4)=(0,2)\\
  \alpha  \cdot\left( \frac{v}{2}-1\right)+(\alpha+1)+\frac{\beta}2 & \textrm{if} \ (v \mod 4, \beta \mod 4)=(2,0)\\
  \alpha  \cdot \left( \frac{v}{2}-1\right)+(\alpha+1)+\left(\frac{\beta}2+1\right)  & \textrm{if} \ (v \mod 4, \beta \mod 4)=(2,2)\\
  \end{array}\right. \\ 
&=&  \left\{ \begin{array}{ll} \ell & \textrm{if} \ (v \mod 4, \beta \mod 4) =(0,0)\\
  \ell+1 & \textrm{if} \ (v \mod 4, \beta \mod 4)\in\{(0,2),(2,0)\}\\
    \ell+2 & \textrm{if} \ (v \mod 4, \beta \mod 4) = (2,2).\\
     \end{array}\right. 
 \end{eqnarray*}
 From the assumptions of case (c), it follows that $N$ is divisible by 4. Hence 
  by Lemma $\ref{l:soliftspin}$, $\overline{y}$ has a preimage $y$ in $G$ of order $v$.
 Also $\overline{y}$ has $\alpha$ eigenvalues equal to $1$, $\alpha+1$ eigenvalues equal to $-1$, $\alpha+1$ eigenvalues equal to $\omega^i$ where $i$ is any integer with $1  \leq |i| \leq  \beta/2$, and every other eigenvalue of $\overline{y}$ occurs with multiplicity $\alpha$. Therefore $C_{{\rm SO}_{h+1}(K)}(\overline{y})^0=A_\alpha^{\frac\beta2} A_{\alpha-1}^{\frac{v}2-1-\frac\beta2}B_{\frac {\alpha-1}2}D_{\frac{\alpha+1}2}T_{\frac{v}2-1}$. \\
  Suppose that case (d) or (e) holds.  As any eigenvalue of $\overline{y}$ different from 1 can be paired with its inverse,  $\overline{y}$ is an element of ${\rm SO}_{h+1}(K)$ of order $v$. Let $N$ be  the number of eigenvalues $\omega^i$ of $\overline{y}$ with $i$ odd. Then 
 \begin{eqnarray*}
 N& = &\left\{ \begin{array}{ll} \alpha  \cdot \frac{v}{2}+\left(\frac{\beta}2-2\right)& \textrm{if} \ (v \mod 4, \beta \mod 4)=(0,0)\\
 \alpha  \cdot \frac{v}{2}+\left(\frac{\beta}2-1\right)  & \textrm{if} \ (v \mod 4, \beta \mod 4)=(0,2)\\
  \alpha  \cdot\left( \frac{v}{2}-1\right)+(\alpha+1)+\left(\frac{\beta}2-2\right) & \textrm{if} \ (v \mod 4, \beta \mod 4)=(2,0)\\
  \alpha  \cdot \left( \frac{v}{2}-1\right)+(\alpha+1)+\left(\frac{\beta}2-1\right)  & \textrm{if} \ (v \mod 4, \beta \mod 4)=(2,2)\\
  \end{array}\right. \\ 
&=&  \left\{ \begin{array}{ll} \ell-2 & \textrm{if} \ (v \mod 4, \beta \mod 4)=(0,0)\\
  \ell-1 & \textrm{if} \ (v \mod 4, \beta \mod 4)\in \{(0,2),(2,0)\}\\
  \ell  & \textrm{if} \ (v \mod 4, \beta \mod 4)=(2,2).\\
  \end{array}\right. \\ 
 \end{eqnarray*}
 From the assumptions of case (d)  or (e), it follows that $N$ is divisible by 4. Hence 
  by Lemma $\ref{l:soliftspin}$, $\overline{y}$ has a preimage $y$ in $G$ of order $v$.
 Also $\overline{y}$ has $\alpha+2$ eigenvalues equal to $1$, $\alpha+1$ eigenvalues equal to $-1$, $\alpha+1$ eigenvalues equal to $\omega^i$ where if case (d) holds then $i$ is any integer with $1 \leq |i| \leq  (\beta/2-1)$ and if case $(e)$ holds then $i$ any  integer with $2  \leq |i| \leq  \beta/2$, and every other eigenvalue of $\overline{y}$ occurs with multiplicity $\alpha$. Therefore $C_{{\rm SO}_{h+1}(K)}(\overline{y})^0=A_\alpha^{\frac\beta2-1} A_{\alpha-1}^{\frac{v}2-\frac\beta2}B_{\frac {\alpha+1}2}D_{\frac{\alpha+1}2}T_{\frac{v}2-1}$. \\
 Suppose that case (f) holds.  As any eigenvalue of $\overline{y}$ different from 1 can be paired with its inverse,  $\overline{y}$ is an element of ${\rm SO}_{h+1}(K)$ of order $v$. Let $N$ be  the number of eigenvalues $\omega^i$ of $\overline{y}$ with $i$ odd. Then 
 \begin{eqnarray*}
 N& = &\left\{ \begin{array}{ll} (\alpha-1)  \cdot \frac{v}{2}+\left(\frac{v}2-2\right)& \textrm{if} \ v \equiv 0 \mod 4\\
 (\alpha-1)  \cdot \left(\frac{v}{2}-1\right)+((\alpha-1)+2)+\left(\left(\frac{v}2-1\right)-2\right)  & \textrm{if} \ v  \equiv 2\mod 4\\
  \end{array}\right. \\ 
&=&  \left\{ \begin{array}{ll} \ell-2 & \textrm{if} \ v \equiv 0 \mod 4\\
  \ell-1& \textrm{if} \ v \equiv 2\mod 4. \end{array}\right.
 \end{eqnarray*}
 From the assumptions of case (f), it follows that $N$ is divisible by 4. Hence 
  by Lemma $\ref{l:soliftspin}$, $\overline{y}$ has a preimage $y$ in $G$ of order $v$.
 Also $\overline{y}$ has $\alpha+2$ eigenvalues equal to $1$, $\alpha+1$ eigenvalues equal to $-1$, $\alpha$ eigenvalues equal to $\omega^i$ where $i$ is any integer with $2  \leq |i| \leq  v/2-1$, and every other eigenvalue of $\overline{y}$ occurs with multiplicity $\alpha-1$. Therefore $C_{{\rm SO}_{h+1}(K)}(\overline{y})^0=A_{\alpha-1}^{\frac v2-2} A_{\alpha-2}B_{\frac {\alpha+1}2}D_{\frac{\alpha+1}2}T_{\frac{v}2-1}$. \\ 
Suppose that case (g) holds.  As any eigenvalue of $\overline{y}$ different from 1 can be paired with its inverse,  $\overline{y}$ is an element of ${\rm SO}_{h+1}(K)$ of order $v$. Let $N$ be the  number of eigenvalues $\omega^i$ of $\overline{y}$ with $i$ odd. Then 
 \begin{eqnarray*}
 N& = &\left\{ \begin{array}{ll} \alpha  \cdot \frac{v}{2}+\frac{\beta}2& \textrm{if} \ (v \mod 4, \beta \mod 4)=(0,0)\\
 \alpha  \cdot \frac{v}{2}+\left(\frac{\beta}2+1\right)  & \textrm{if} \ (v \mod 4, \beta \mod 4)=(0,2)\\
  \alpha  \cdot\left( \frac{v}{2}-1\right)+\alpha+\frac{\beta}2 & \textrm{if} \ (v \mod 4, \beta \mod 4)=(2,0)\\
  \alpha  \cdot \left( \frac{v}{2}-1\right)+\alpha+\left(\frac{\beta}2+1\right)  & \textrm{if} \ (v \mod 4, \beta \mod 4)=(2,2)\\
  \end{array}\right. \\ 
&=&  \left\{ \begin{array}{ll} \ell & \textrm{if} \ (v \mod 4, \beta \mod 4)\in\{(0,0),(2,0)\}\\
  \ell+1 & \textrm{if} \ (v \mod 4, \beta \mod 4)\in\{(0,2),(2,2)\}.\\
  \end{array}\right. \\ 
 \end{eqnarray*}
 From the assumptions of case (g), it follows that $N$ is divisible by 4. Hence 
  by Lemma $\ref{l:soliftspin}$, $\overline{y}$ has a preimage $y$ in $G$ of order $v$.
 Also $\overline{y}$ has $\alpha+1$ eigenvalues equal to $1$, $\alpha$ eigenvalues equal to $-1$, $\alpha+1$ eigenvalues equal to $\omega^i$ where $i$ is any integer with $1  \leq |i| \leq  \beta/2$, and every other eigenvalue of $\overline{y}$ occurs with multiplicity $\alpha$. Therefore $C_{{\rm SO}_{h+1}(K)}(\overline{y})^0=A_\alpha^{\frac\beta2} A_{\alpha-1}^{\frac{v}2-1-\frac\beta2}B_{\frac {\alpha}2}D_{\frac{\alpha}2}T_{\frac{v}2-1}$. \\
 Suppose that case (h) holds. As any eigenvalue of $\overline{y}$ different from 1 can be paired with its inverse,  $\overline{y}$ is an element of ${\rm SO}_{h+1}(K)$ of order $v$. Let $N$ be  the number of eigenvalues $\omega^i$ of $\overline{y}$ with $i$ odd. Then 
 $$ N= \alpha\cdot\left( \frac v2-1\right) + \left( \frac \beta2 +1-2\right)+\alpha= \frac{2(\ell-1)}2.$$
 From the assumptions of case (h), it follows that $N$ is divisible by 4. Hence 
  by Lemma $\ref{l:soliftspin}$, $\overline{y}$ has a preimage $y$ in $G$ of order $v$.
 Also $\overline{y}$ has $\alpha+1$ eigenvalues equal to $1$, $\alpha$ eigenvalues equal to $-1$, $\alpha+1$ eigenvalues equal to $\omega^i$ where $i$ is any integer with $2  \leq |i| \leq  \beta/2$,  $(\alpha+1)$ eigenvalues equal to $\omega^{\frac{v}2-1}$,  $(\alpha+1)$ eigenvalues equal to $\omega^{-\left(\frac{v}2-1\right)}$ and every other eigenvalue of $\overline{y}$ occurs with multiplicity $\alpha$. Therefore $C_{{\rm SO}_{h+1}(K)}(\overline{y})^0=A_{\alpha}^{\frac{\beta}2}A_{\alpha-1}^{\frac v2-\frac\beta2-1}B_{\frac {\alpha}2}D_{\frac{\alpha}2}T_{\frac{v}2-1}$. \\  
  Suppose that case (i) or (j) holds.  As any eigenvalue of $\overline{y}$ different from 1 can be paired with its inverse,  $\overline{y}$ is an element of ${\rm SO}_{h+1}(K)$ of order $v$. Let $N$ be the  number of eigenvalues $\omega^i$ of $\overline{y}$ with $i$ odd. Then 
 \begin{eqnarray*}
 N& = &\left\{ \begin{array}{ll} \alpha  \cdot \frac{v}{2}+\left(\frac{\beta}2-2\right)& \textrm{if} \ (v \mod 4, \beta \mod 4)=(0,0)\\
 \alpha  \cdot \frac{v}{2}+\left(\frac{\beta}2-1\right)  & \textrm{if} \ (v \mod 4, \beta \mod 4)=(0,2)\\
  \alpha  \cdot\left( \frac{v}{2}-1\right)+(\alpha+2)+\frac{\beta}2 & \textrm{if} \ (v \mod 4, \beta \mod 4)=(2,0)\\
  \end{array}\right. \\ 
&=&  \left\{ \begin{array}{ll} \ell-2 & \textrm{if} \ (v \mod 4, \beta \mod 4)=(0,0)\\
  \ell-1 & \textrm{if} \ (v \mod 4, \beta \mod 4)=(0,2)\\
  \ell+2& \textrm{if} \ (v \mod 4, \beta \mod 4)=(2,0).\\
  \end{array}\right. \\ 
 \end{eqnarray*}
 From the assumptions of case (i)  or (j), it follows that $N$ is divisible by 4. Hence 
  by Lemma $\ref{l:soliftspin}$, $\overline{y}$ has a preimage $y$ in $G$ of order $v$.
 Also $\overline{y}$ has $\alpha+1$ eigenvalues equal to $1$, $\alpha+2$ eigenvalues equal to $-1$, $\alpha+1$ eigenvalues equal to $\omega^i$ where if case (i) holds then $i$ is any integer with $1 \leq |i| \leq  (\beta/2-1)$ and if case $(j)$ holds then $i$ any  integer with $2  \leq |i| \leq  \beta/2$, and every other eigenvalue of $\overline{y}$ occurs with multiplicity $\alpha$. Therefore $C_{{\rm SO}_{h+1}(K)}(\overline{y})^0=A_\alpha^{\frac\beta2-1} A_{\alpha-1}^{\frac{v}2-\frac\beta2}B_{\frac {\alpha}2}D_{\frac{\alpha+2}2}T_{\frac{v}2-1}$. \\
 Suppose that case (k) holds.  As any eigenvalue of $\overline{y}$ different from 1 can be paired with its inverse,  $\overline{y}$ is an element of ${\rm SO}_{h+1}(K)$ of order $v$. Let $N$ be the  number of eigenvalues $\omega^i$ of $\overline{y}$ with $i$ odd. Then 
 $$ N= (\alpha-1)\cdot \left(\frac v2-1\right) + \left( \frac v2 -1\right)+(\alpha+2)= \frac{2(\ell +2)}2.$$
 From the assumption of case (k), it follows that $N$ is divisible by 4. Hence 
  by Lemma $\ref{l:soliftspin}$, $\overline{y}$ has a preimage $y$ in $G$ of order $v$.
 Also $\overline{y}$ has $\alpha+1$ eigenvalues equal to $1$, $\alpha+2$ eigenvalues equal to $-1$, $\alpha$ eigenvalues equal to $\omega^i$ where $i$ is any integer with $1  \leq |i| \leq  v/2-2$, and every other eigenvalue of $\overline{y}$ occurs with multiplicity $\alpha-1$. Therefore $C_{{\rm SO}_{h+1}(K)}(\overline{y})^0=A_{\alpha-1}^{\frac v2-2} A_{\alpha-2}B_{\frac {\alpha}2}D_{\frac{\alpha+2}2}T_{\frac{v}2-1}$. \\  
  
 By Proposition \ref{p:prelimdu}(i), $d_u(G)\leq d_q(C_G(y)^0)$ and so $d_q(C_G(y)^0)$ is an upper bound for $d_u(G)$. It remains to calculate $d_q(C_{{\rm SO}_{h+1}}(\overline{y})^0)=d_q(C_G(y)^0)$.\\

It follows from Lemma \ref{l:lawther} that 
$$d_q(A_{\alpha+1})=\left\{\begin{array}{ll}
z^2q+(2z+1)(\alpha+2-zq)-1 & \textrm{if} \ \delta < q-1\\
(z+1)^2q+(2z+3)(\alpha+2-(z+1)q)-1 & \textrm{if} \ \delta=q-1,
 \end{array}\right.$$
 $$d_q(A_{\alpha}) = z^2q +(2z+1)(\alpha-zq+1)-1,$$
$$d_q(A_{\alpha-1})=z^2q+(2z+1)(\alpha-zq)-1,$$
$$d_q(A_{\alpha-2})=\left\{\begin{array}{ll}
z^2q+(2z+1)(\alpha-1-zq)-1 & \textrm{if} \ \delta >0\\
(z-1)^2q+(2z-1)(\alpha-1-(z-1)q)-1 & \textrm{if} \ \delta=0,
 \end{array}\right.$$
$$d_q\left(B_{\lceil\frac{\alpha}2\rceil}\right)=\frac{z^2q}{2}+(2z+1)\left(\left\lceil\frac{\alpha}{2}\right\rceil-\frac{zq}2\right)+\left\lceil \frac z2\right\rceil\epsilon_q,$$
$$d_q\left(D_{\lceil\frac{\alpha+1}2\rceil}\right)=
\frac{z^2q}{2}+(2z+1)\left(\left\lceil\frac{\alpha+1}{2}\right\rceil-\frac{zq}2\right)+\left\lceil \frac z2\right\rceil\epsilon_q-z-\epsilon_z,$$
if $\alpha$ is odd then 
$$d_q\left(B_{\frac{\alpha-1}2}\right)=\frac{z^2q}{2}+(2z+1)\left(\frac{\alpha-1}{2}-\frac{zq}{2}\right)+\left\lceil \frac{z}{2}\right\rceil\epsilon_q,$$
$$d_q\left(D_{\frac{\alpha-1}2}\right)=
\left\{\begin{array}{ll}
\frac{z^2q}{2}+(2z+1)\left(\frac{\alpha-1}{2}-\frac{zq}2\right)+\left\lceil \frac z2\right\rceil\epsilon_q - z - \epsilon_z& \textrm{if} \  \delta>0\\
\frac{(z-1)^2q}{2}+(2z-1)\left(\frac{\alpha-1}{2}-\frac{(z-1)q}2\right)+\frac {z-1}2\epsilon_q-(z-1)& \textrm{if} \ \delta=0,
\end{array}\right.$$
and if $\alpha$ is even then
$$d_q\left(B_{\frac{\alpha-2}2}\right)=
\left\{\begin{array}{ll}
\frac{z^2q}{2}+(2z+1)\left(\frac{\alpha-2}{2}-\frac{zq}2\right)+\left\lceil \frac z2\right\rceil\epsilon_q & \textrm{if} \  \delta>0\\
\frac{(z-1)^2q}{2}+(2z-1)\left(\frac{\alpha-2}{2}-\frac{(z-1)q}2\right)+\left\lceil\frac{z-1}2\right\rceil \epsilon_q& \textrm{if} \ \delta=0,
\end{array}\right.$$
and
$$d_q\left(D_{\frac{\alpha}2}\right)=\left\{\begin{array}{ll}
\frac{z^2q}{2}+(2z+1)\left(\frac{\alpha}{2}-\frac{zq}2\right)+\left\lceil \frac z2\right\rceil\epsilon_q-z-\epsilon_z & \textrm{if} \ \delta>0 \ \textrm{or} \ z \  \textrm{is even}\\
\frac{z^2q}{2}+(2z+1)\left(\frac{\alpha}{2}-\frac{zq}2\right)+\left\lceil \frac z2\right\rceil\epsilon_q-z-\epsilon_z+2 & \textrm{otherwise}.  
\end{array}\right.
$$

Assume first that case (a) holds so that  $\epsilon_v=1$. Recall that  $\epsilon_u=\epsilon_q$ and $\epsilon_\alpha=\epsilon_\beta$. We  have $$C_{{\rm SO}_{h+1}(K)}(\overline{y})^0=A_\alpha^{\lfloor\frac\beta2\rfloor} A_{\alpha-1}^{\frac{v-1}2-\lfloor\frac\beta2\rfloor}B_{\lceil \frac \alpha2\rceil}T_{\frac{v-1}2}.$$
By Proposition \ref{p:prelimdu}(iii),
\begin{eqnarray*}
d_q(C_{{\rm SO}_{h+1}(K)}(\overline{y})^0) & = & \frac{v-1}2+d_q\left(A_\alpha^{\lfloor\frac\beta2\rfloor} A_{\alpha-1}^{\frac{v-1}2-\lfloor\frac\beta2\rfloor}B_{\lceil \frac \alpha2\rceil}\right)\\
& = & \frac{v-1}2+\left\lfloor \frac \beta2 \right \rfloor d_q(A_\alpha)+\left(\frac{v-1}2-\left \lfloor \frac \beta2 \right \rfloor\right)d_q(A_{\alpha-1})+d_q\left(B_{\lceil \frac \alpha2\rceil}\right)\\
& =  &z(\alpha v+ \beta)-\frac{qvz(z+1)}2+\frac{\alpha v +\beta}2+\left\lceil\frac{z}2\right\rceil\epsilon_q\\
& = & z(zu+e)-\frac{zu(z+1)}2+\frac{zu+e}2+\left\lceil\frac{z}2\right\rceil\epsilon_q\\
& = &  \frac12(z^2u+e(2z+1))+\left\lceil\frac{z}2\right\rceil\epsilon_u.
\end{eqnarray*}

Assume that   case (b) holds so that $v=2$, $\epsilon_u=\epsilon_v=0$ and $\epsilon_q=1$. Recall that $\alpha=\ell$ and $\beta=0$.  Note also that $e=0$ if and only if $\delta=0$. We have 
 $$C_{{\rm SO}_{h+1}(K)}(\overline{y})^0= \left\{ \begin{array}{ll}
 B_{\frac{\alpha}2}D_{\frac{\alpha}2} & \textrm{if} \ \alpha \equiv 0 \ (4)\\ 
 B_{\frac{\alpha+1}2}D_{\frac{\alpha-1}2} & \textrm{if} \ \alpha \equiv 1 \ (4)\\ 
B_{\frac{\alpha-2}2}D_{\frac{\alpha+2}2} & \textrm{if} \ \alpha \equiv 2 \ (4)\\ 
B_{\frac{\alpha-1}2}D_{\frac{\alpha+1}2} & \textrm{if} \ \alpha \equiv 3 \ (4).
\end{array}\right.
$$
If $\alpha \equiv 0 \mod 4$ then $z$ is even if $\delta=0$, and by Proposition \ref{p:prelimdu}(i)
\begin{eqnarray*}
d_q(C_{{\rm SO}_{h+1}(K)}(\overline{y})^0) & = &  d_q(B_{\frac{\alpha}2}D_{\frac{\alpha}2})\\
& = & d_q(B_{\frac{\alpha}2})+d_q(D_{\frac{\alpha}2})\\
& = &   -qz(z+1)+2\alpha z+ \alpha\\
& = & \frac{-qvz(z+1)}2+(zu+e)z+\frac{(zu+e)}2\\ 
& = &  \frac{-zu(z+1)}2+(zu+e)z+\frac{(zu+e)}2\\  
& = & \frac{z^2u+e(2z+1)}2. 
\end{eqnarray*}
If $\alpha \equiv 1 \mod 4$ then  by Proposition \ref{p:prelimdu}(iii)
\begin{eqnarray*}
d_q(C_{{\rm SO}_{h+1}(K)}(\overline{y})^0) & = &  d_q(B_{\frac{\alpha+1}2}D_{\frac{\alpha-1}2})\\
& = & d_q(B_{\frac{\alpha+1}2})+d_q(D_{\frac{\alpha-1}2})\\
& = &  \left\{\begin{array}{ll} -qz(z+1)+2\alpha z+ \alpha & \textrm{if} \ \delta>0 \\
-qz^2+2\alpha z +2& \textrm{if} \ \delta=0 \end{array}\right. \\
& = &  \left\{\begin{array}{ll}  \frac{z^2u+e(2z+1)}2 & \textrm{if} \ \delta>0 \\
z^2q +2& \textrm{if} \ \delta=0 
\end{array}\right.\\
& = &  \left\{\begin{array}{ll}  \frac{z^2u+e(2z+1)}2 & \textrm{if} \ \delta>0 \\
\frac{z^2u}2 +2& \textrm{if} \ \delta=0 
\end{array}\right.\\
& = &  \left\{\begin{array}{ll}  \frac{z^2u+e(2z+1)}2 & \textrm{if} \ e>0 \\
 \frac{z^2u+e(2z+1)}2 +2 & \textrm{if} \ e=0. 
\end{array}\right.
\end{eqnarray*}
If $\alpha \equiv 2 \mod 4$ then $z$ is even if $\delta=0$, and   by Proposition \ref{p:prelimdu}(iii)
\begin{eqnarray*}
d_q(C_{{\rm SO}_{h+1}(K)}(\overline{y})^0) & = &  d_q(B_{\frac{\alpha-2}2}D_{\frac{\alpha+2}2})\\
& = & d_q(B_{\frac{\alpha-2}2})+d_q(D_{\frac{\alpha+2}2})\\
& = &  \left\{\begin{array}{ll} -qz(z+1)+2\alpha z+ \alpha & \textrm{if} \ \delta>0 \\
-qz^2+2\alpha z +2& \textrm{if} \ \delta=0 \end{array}\right. \\
& = &  \left\{\begin{array}{ll}  \frac{z^2u+e(2z+1)}2 & \textrm{if} \ e>0 \\
 \frac{z^2u+e(2z+1)}2 +2 & \textrm{if} \ e=0. 
\end{array}\right.
\end{eqnarray*}
If $\alpha\equiv 3 \mod 4$ then  by Proposition \ref{p:prelimdu}(iii)
\begin{eqnarray*}
d_q(C_{{\rm SO}_{h+1}(K)}(\overline{y})^0) & = &  d_q(B_{\frac{\alpha-1}2}D_{\frac{\alpha+1}2})\\
& = & d_q(B_{\frac{\alpha-1}2})+d_q(D_{\frac{\alpha+1}2})\\
&=&  \frac{-qvz(z+1)}2+(zu+e)z+\frac{(zu+e)}2\\
& = & \frac{z^2u+e(2z+1)}2. 
\end{eqnarray*}

Assume that case (c) holds.  Recall that  $\epsilon_q=\epsilon_\alpha=1$ and $\epsilon_u=\epsilon_v=\epsilon_{\beta}=0$. We  have $$C_{{\rm SO}_{h+1}(K)}(\overline{y})^0=A_\alpha^{\frac\beta2} A_{\alpha-1}^{\frac{v}2-1-\frac\beta2}B_{\frac{\alpha-1}2}D_{\frac{\alpha+1}2}T_{\frac{v}2-1}.$$
By Proposition \ref{p:prelimdu}(iii),
\begin{eqnarray*}
d_q(C_{{\rm SO}_{h+1}(K)}(\overline{y})^0) & = & \frac{v}2-1+d_q\left(A_\alpha^{\frac\beta2} A_{\alpha-1}^{\frac{v}2-1-\frac\beta2}B_{\frac{\alpha-1}2}D_{\frac{\alpha+1}2}\right)\\
& = & \frac{v}2-1+\frac\beta2  d_q(A_\alpha)+\left(\frac{v}2-1- \frac \beta2 \right)d_q(A_{\alpha-1})+d_q\left(B_{\frac{\alpha-1}2}\right)+d_q\left(D_{\frac{\alpha+1}2}\right)\\
& =  &z(\alpha v+ \beta)-\frac{qvz(z+1)}2+\frac{\alpha v +\beta}2\\
& = & z(zu+e)-\frac{zu(z+1)}2+\frac{zu+e}2\\
& = &  \frac12(z^2u+e(2z+1)).
\end{eqnarray*}

Assume that case (d) or case (e) holds.  Recall that  $\epsilon_q=\epsilon_\alpha=1$ and $\epsilon_u=\epsilon_v=\epsilon_{\beta}=0$. We  have $$C_{{\rm SO}_{h+1}(K)}(\overline{y})^0=A_\alpha^{\frac\beta2-1} A_{\alpha-1}^{\frac{v}2-\frac\beta2}B_{\frac{\alpha+1}2}D_{\frac{\alpha+1}2}T_{\frac{v}2-1}.$$
By Proposition \ref{p:prelimdu}(iii),
\begin{eqnarray*}
d_q(C_{{\rm SO}_{h+1}(K)}(\overline{y})^0) & = & \frac{v}2-1+d_q\left(A_\alpha^{\frac\beta2-1} A_{\alpha-1}^{\frac{v}2-\frac\beta2}B_{\frac{\alpha+1}2}D_{\frac{\alpha+1}2}\right)\\
& = & \frac{v}2-1+\left(\frac\beta2-1\right)  d_q(A_\alpha)+\left(\frac{v}2- \frac \beta2 \right)d_q(A_{\alpha-1})+d_q\left(B_{\frac{\alpha+1}2}\right)+d_q\left(D_{\frac{\alpha+1}2}\right)\\
& =  &z(\alpha v+ \beta)-\frac{qvz(z+1)}2+\frac{\alpha v +\beta}2\\
& = & z(zu+e)-\frac{zu(z+1)}2+\frac{zu+e}2\\
& = &  \frac12(z^2u+e(2z+1)).
\end{eqnarray*}

Assume that case (f) holds.  Recall that  $\epsilon_q=\epsilon_\alpha=1$, $\epsilon_u=\epsilon_v=\epsilon_{\beta}=0$ and $\beta=0$. Also $\delta=0$ if and only if $e=0$.  We  have $$C_{{\rm SO}_{h+1}(K)}(\overline{y})^0=A_{\alpha-1}^{\frac v2-2} A_{\alpha-2}B_{\frac{\alpha+1}2}D_{\frac{\alpha+1}2}T_{\frac{v}2-1}.$$
By Proposition \ref{p:prelimdu}(iii),
\begin{eqnarray*}
d_q(C_{{\rm SO}_{h+1}(K)}(\overline{y})^0) & = & \frac{v}2-1+d_q\left(A_{\alpha-1}^{\frac v2-2} A_{\alpha-2}B_{\frac{\alpha+1}2}D_{\frac{\alpha+1}2}\right)\\
& = & \frac{v}2-1+\left(\frac v2-2\right)  d_q(A_{\alpha-1})+d_q(A_{\alpha-2})+d_q\left(B_{\frac{\alpha+1}2}\right)+d_q\left(D_{\frac{\alpha+1}2}\right)\\
& =  & \left\{ \begin{array}{ll} 
z\alpha v-\frac{qvz(z+1)}2+\frac{\alpha v }2 & \textrm{if} \ \delta>0\\
z\alpha v-\frac{qvz(z+1)}2+\frac{\alpha v }2+2& \textrm{if} \ \delta=0\\
\end{array}\right.\\
& = & \left\{ \begin{array}{ll} 
z(zu+e)-\frac{zu(z+1)}2+\frac{zu+e}2 & \textrm{if} \ \delta>0\\
z(zu+e)-\frac{zu(z+1)}2+\frac{zu+e}2 +2& \textrm{if} \ \delta=0\\
 \end{array}\right.\\
 & = & \left\{ \begin{array}{ll} 
  \frac12(z^2u+e(2z+1)) & \textrm{if} \  e>0\\
  \frac12(z^2u+e(2z+1))+2 & \textrm{if} \  e=0.
  \end{array}\right.
\end{eqnarray*}

Assume that case (g) or case (h) holds.  Recall that  $\epsilon_q=1$ and $\epsilon_u=\epsilon_v=\epsilon_\alpha=\epsilon_{\beta}=0$. Also if $\delta=0$ then $z$ must be even.
We  have $$C_{{\rm SO}_{h+1}(K)}(\overline{y})^0=A_\alpha^{\frac\beta2} A_{\alpha-1}^{\frac{v}2-1-\frac\beta2}B_{\frac{\alpha}2}D_{\frac{\alpha}2}T_{\frac{v}2-1}.$$
By Proposition \ref{p:prelimdu}(iii),
\begin{eqnarray*}
d_q(C_{{\rm SO}_{h+1}(K)}(\overline{y})^0) & = & \frac{v}2-1+d_q\left(A_\alpha^{\frac\beta2} A_{\alpha-1}^{\frac{v}2-1-\frac\beta2}B_{\frac{\alpha}2}D_{\frac{\alpha}2}\right)\\
& = & \frac{v}2-1+\frac\beta2  d_q(A_\alpha)+\left(\frac{v}2-1- \frac \beta2 \right)d_q(A_{\alpha-1})+d_q\left(B_{\frac{\alpha}2}\right)+d_q\left(D_{\frac{\alpha}2}\right)\\
& =  &z(\alpha v+ \beta)-\frac{qvz(z+1)}2+\frac{\alpha v +\beta}2\\
& = & z(zu+e)-\frac{zu(z+1)}2+\frac{zu+e}2\\
& = &  \frac12(z^2u+e(2z+1)).
\end{eqnarray*}

Assume that case (i) or case (j) holds.  Recall that  $\epsilon_q=1$ and  $\epsilon_u=\epsilon_v=\epsilon_\alpha=\epsilon_{\beta}=0$. We  have $$C_{{\rm SO}_{h+1}(K)}(\overline{y})^0=A_\alpha^{\frac\beta2-1} A_{\alpha-1}^{\frac{v}2-\frac\beta2}B_{\frac{\alpha}2}D_{\frac{\alpha+2}2}T_{\frac{v}2-1}.$$
By Proposition \ref{p:prelimdu}(iii),
\begin{eqnarray*}
d_q(C_{{\rm SO}_{h+1}(K)}(\overline{y})^0) & = & \frac{v}2-1+d_q\left(A_\alpha^{\frac\beta2-1} A_{\alpha-1}^{\frac{v}2-\frac\beta2}B_{\frac{\alpha}2}D_{\frac{\alpha+2}2}\right)\\
& = & \frac{v}2-1+\left(\frac\beta2-1\right)  d_q(A_\alpha)+\left(\frac{v}2- \frac \beta2 \right)d_q(A_{\alpha-1})+d_q\left(B_{\frac{\alpha}2}\right)+d_q\left(D_{\frac{\alpha+2}2}\right)\\
& =  &z(\alpha v+ \beta)-\frac{qvz(z+1)}2+\frac{\alpha v +\beta}2\\
& = & z(zu+e)-\frac{zu(z+1)}2+\frac{zu+e}2\\
& = &  \frac12(z^2u+e(2z+1)).
\end{eqnarray*}

Assume that case (k) holds.  Recall that  $\epsilon_q=1$,  $\epsilon_u=\epsilon_v=\epsilon_\alpha=\epsilon_{\beta}=0$ and $\beta=0$. Also $\delta=0$ if and only if $e=0$.  We  have $$C_{{\rm SO}_{h+1}(K)}(\overline{y})^0=A_{\alpha-1}^{\frac v2-2} A_{\alpha-2}B_{\frac{\alpha}2}D_{\frac{\alpha+2}2}T_{\frac{v}2-1}.$$
By Proposition \ref{p:prelimdu}(iii),
\begin{eqnarray*}
d_q(C_{{\rm SO}_{h+1}(K)}(\overline{y})^0) & = & \frac{v}2-1+d_q\left(A_{\alpha-1}^{\frac v2-2} A_{\alpha-2}B_{\frac{\alpha}2}D_{\frac{\alpha+2}2}\right)\\
& = & \frac{v}2-1+\left(\frac v2-2\right)  d_q(A_{\alpha-1})+d_q(A_{\alpha-2})+d_q\left(B_{\frac{\alpha}2}\right)+d_q\left(D_{\frac{\alpha+2}2}\right)\\
& =  & \left\{ 
\begin{array}{ll}
z\alpha v-\frac{qvz(z+1)}2+\frac{\alpha v}2 & \textrm{if} \ \delta>0\\
z\alpha v-\frac{qvz(z+1)}2+\frac{\alpha v}2 +2& \textrm{if} \ \delta=0\\
\end{array}
\right.\\
& = & \left\{
\begin{array}{ll}
z(zu+e)-\frac{zu(z+1)}2+\frac{zu+e}2  & \textrm{if} \ \delta>0\\
z(zu+e)-\frac{zu(z+1)}2+\frac{zu+e}2+2 & \textrm{if} \ \delta=0\\
\end{array}
\right.\\
& = & \left\{
\begin{array}{ll}
  \frac12(z^2u+e(2z+1)) & \textrm{if} \ e>0\\
  \frac12(z^2u+e(2z+1))+2 & \textrm{if} \ e=0.
  \end{array}
  \right.  
\end{eqnarray*}
\end{proof}

\subsection{$G$ is  of type $D_\ell$}

Let $G=(D_\ell)_{{\rm s.c}}$ be a simply connected group of type ${D}_{\ell}$ with $\ell \geq 4$ defined over an algebraically closed field $K$ of  characteristic $p$. Here $|\Phi|=2\ell(\ell-1)$, $h=2\ell-2$ and $G={\rm Spin}_{h+2}(K)$  if $p\neq 2$, otherwise $G={\rm SO}_{h+2}(K)$. Given an element $g$ in $G$, we let $\overline{g}$ be the image of $g$ under the canonical surjective map $G\rightarrow {\rm SO}_{h+2}(K)$.  \\

\begin{lem}\label{l:dsc}
Let  $G=(D_{\ell})_{{\rm s.c}}$ where $\ell \geq 4$ be defined over an algebraically closed field $K$ of  characteristic $p$. Let $u=qv$ be a positive integer where $q$ and $v$ are positive integers such that $q$ is a power of $p$ and $p$ does not divide $v$. Write $h=zu+e=\alpha v +\beta$ and $\alpha=zq+\delta$ where $z$, $e$, $\alpha$, $\beta$, $\delta$ are nonnegative integers such that $e<u$, $\beta<v$, and $\delta<q$. Consider the following conditions:\\
\begin{enumerate}[(a)]
\item $\epsilon_v=1$. 
\item $v=2$. 
\item $(\epsilon_v,\epsilon_\alpha)=(0,1)$, $v>2$ and $(v \mod 4,\beta \mod 4,\ell \mod 4) \in \{(0,0,1), (0,2,0), (2,0,0), (2,2,3)\}.$
\item $(\epsilon_v,\epsilon_{\alpha})=(0,1)$ and $(v \mod 4,\beta \mod 4,\ell \mod 4) = (0,2,2).$
\item $(\epsilon_v,\epsilon_{\alpha})=(0,1)$, $v>2$ and $(v \mod 4,\beta \mod 4,\ell \mod 4) = (2,2,1).$
\item $(\epsilon_v,\epsilon_{\alpha})=(0,1)$, $\beta>0$ and $(v \mod 4,\beta \mod 4,\ell \mod 4) = (0,0,3).$
\item $(\epsilon_v,\epsilon_{\alpha})=(0,1)$, $v>2$, $\beta>0$, $\beta \neq v-2$ and $(v \mod 4,\beta \mod 4,\ell \mod 4) = (2,0,2).$
\item $(\epsilon_v,\epsilon_{\alpha})=(0,1)$, $v>2$, $\beta=v-2$ and $(v \mod 4,\beta \mod 4,\ell \mod 4) = (2,0,2).$
\item $(\epsilon_v,\epsilon_{\alpha})=(0,1)$, $\beta=0$ and $(v \mod 4,\beta \mod 4,\ell \mod 4) = (0,0,3).$
\item $(\epsilon_v,\epsilon_{\alpha})=(0,1)$, $v>2$, $\beta=0$ and $(v \mod 4,\beta \mod 4,\ell \mod 4) = (2,0,2).$
\item $(\epsilon_v,\epsilon_\alpha)=(0,0)$, $v>2$ and $(v \mod 4,\beta \mod 4,\ell \mod 4) \in \{(0,0,1), (0,2,0), (2,0,1), (2,2,0)\}.$
\item  $(\epsilon_v,\epsilon_{\alpha})=(0,0)$ and  $(v \mod 4,\beta \mod 4,\ell \mod 4) =(2,2,2)$.
\item  $(\epsilon_v,\epsilon_{\alpha})=(0,0)$ and  $(v \mod 4,\beta \mod 4,\ell \mod 4) =(0,0,3)$.
\item  $(\epsilon_v,\epsilon_{\alpha})=(0,0)$ and  $(v \mod 4,\beta \mod 4,\ell \mod 4) =(0,2,2)$.
\item  $(\epsilon_v,\epsilon_{\alpha})=(0,0)$, $v>2$, $\beta>0$ and  $(v \mod 4,\beta \mod 4,\ell \mod 4) =(2,0,3)$.
\item  $(\epsilon_v,\epsilon_{\alpha})=(0,0)$, $v>2$, $\beta=0$ and  $(v \mod 4,\ell \mod 4) =(2,3)$.
\end{enumerate}
Let $\overline{y}$ be the element of  ${\rm SO}_{h+2}(K)$ of order $v$ defined in Table \ref{t:dsc} (see \S\ref{s:tables}). Then $\overline{y}$ has a preimage $y$ in $G$ of order $v$, and 
$C_{{\rm SO}_{h+2}}(\overline{y})^0$ and $d_q(C_G(y)^0)=d_q(C_{{\rm SO}_{h+2}(K)}(\overline{y})^0)$ are given in Table \ref{t:dsc}. Furthermore $d_q(C_G(y)^0)$ is an upper bound for $d_u(G)$.
\end{lem}

\begin{proof} Note that by Lemma \ref{l:dufactg}(iii), $d_q(C_G(y)^0)=d_q(C_{{\rm SO}_{h+2}(K)}(\overline{y})^0)$. 
Since $h=2\ell-2$, note that if $\epsilon_v=1$ then $\epsilon_\alpha=\epsilon_\beta$ and $\epsilon_u=\epsilon_q$, whereas if $\epsilon_v=0$ then $\epsilon_u=0$, $\epsilon_q=1$ and $\epsilon_\beta=0$. We first determine $C_{{\rm SO}_{h+2}(K)}(\overline{y})^0$. 
Suppose that case (a) holds so that $v$ is odd.  As any eigenvalue of $\overline{y}$ can be paired with its inverse,  $\overline{y}$ is an element of ${\rm SO}_{h+2}(K)$ of order $v$.  Furthermore, since $v$ is odd, by Lemma \ref{l:soliftspin} it follows that $\overline{y}$ has a preimage $y$ in $G$ of order $v$. Also $\overline{y}$ has $\alpha+2-\epsilon_\beta$ eigenvalues equal to $1$, $0$ eigenvalues equal to $-1$, $\alpha+1$ eigenvalues equal to $\omega^i$ where $i$ is any integer with $1 \leq |i| \leq \lceil \beta/2\rceil$, and every other eigenvalue of $\overline{y}$ occurs with multiplicity $\alpha$. Therefore $C_{{\rm SO}_{h+1}(K)}(\overline{y})^0=A_\alpha^{\lceil\frac\beta2\rceil} A_{\alpha-1}^{\frac{v-1}2-\lceil\frac\beta2\rceil}D_{\lceil \frac {\alpha+1}2\rceil}T_{\frac{v-1}2}$. \\
Suppose that case (b) holds so that $v=2$.  As any eigenvalue of $\overline{y}$ can be paired with its inverse,  $\overline{y}$ is an element of ${\rm SO}_{h+2}(K)$ of order $v$. Now the number $N$ of eigenvalues $\omega^i=(-1)^i$ of $\overline{y}$ with $i$ odd  is equal  to $\ell$, $\ell-1$, $\ell+2$ or $\ell +1$ according respectively as $\ell \mod 4$ is  0, 1, 2 or  3. Since $N$ is divisible by 4, it follows from Lemma $\ref{l:soliftspin}$ that $\overline{y}$ has a preimage $y$ in $G$ of order $v$.  Also 
$$ C_{{\rm SO}_{h+2}(K)}(\overline{y})^0=\left\{\begin{array}{ll}D_{\frac{\ell}2}D_{\frac{\ell}2} & \textrm{if} \ \ell \equiv 0 \ (4)\\ 
 D_{\frac{\ell+1}2}D_{\frac{\ell-1}2} & \textrm{if} \ \ell \equiv 1 \ (4)\\ 
D_{\frac{\ell-2}2}D_{\frac{\ell+2}2} & \textrm{if} \ \ell \equiv 2 \ (4)\\ 
D_{\frac{\ell-1}2}D_{\frac{\ell+1}2} & \textrm{if} \ \ell \equiv 3 \ (4).
 \end{array}\right.$$
 Suppose that case (c) holds.  As any eigenvalue of $\overline{y}$  can be paired with its inverse,  $\overline{y}$ is an element of ${\rm SO}_{h+2}(K)$ of order $v$. Let $N$ be  the number of eigenvalues $\omega^i$ of $\overline{y}$ with $i$ odd. Then 
 \begin{eqnarray*}
 N& = &\left\{ \begin{array}{ll} \alpha  \cdot \frac{v}{2}+\frac{\beta}2& \textrm{if} \ (v \mod 4, \beta \mod 4)=(0,0)\\
 \alpha  \cdot \frac{v}{2}+\left(\frac{\beta}2+1\right)  & \textrm{if} \ (v \mod 4, \beta \mod 4)=(0,2)\\
  \alpha  \cdot\left( \frac{v}{2}-1\right)+(\alpha+1)+\frac{\beta}2 & \textrm{if} \ (v \mod 4, \beta \mod 4)=(2,0)\\
  \alpha  \cdot \left( \frac{v}{2}-1\right)+(\alpha+1)+\left(\frac{\beta}2+1\right)  & \textrm{if} \ (v \mod 4, \beta \mod 4)=(2,2)\\
  \end{array}\right. \\ 
&=&  \left\{ \begin{array}{ll} \ell-1 & \textrm{if} \ (v \mod 4, \beta \mod 4) =(0,0)\\
  \ell & \textrm{if} \ (v \mod 4, \beta \mod 4)\in\{(0,2),(2,0)\}\\
    \ell+1 & \textrm{if} \ (v \mod 4, \beta \mod 4) = (2,2).\\
     \end{array}\right. 
 \end{eqnarray*}
 From the assumptions of case (c), it follows that $N$ is divisible by 4. Hence 
  by Lemma $\ref{l:soliftspin}$, $\overline{y}$ has a preimage $y$ in $G$ of order $v$.
 Also $\overline{y}$ has $\alpha+1$ eigenvalues equal to $1$, $\alpha+1$ eigenvalues equal to $-1$, $\alpha+1$ eigenvalues equal to $\omega^i$ where $i$ is any integer with $1  \leq |i| \leq  \beta/2$, and every other eigenvalue of $\overline{y}$ occurs with multiplicity $\alpha$. Therefore $C_{{\rm SO}_{h+2}(K)}(\overline{y})^0=A_\alpha^{\frac\beta2} A_{\alpha-1}^{\frac{v}2-1-\frac\beta2}D_{\frac {\alpha+1}2}D_{\frac{\alpha+1}2}T_{\frac{v}2-1}$. \\
  Suppose that case (d) or (e) holds.   Note that if case (d) holds then $h\equiv 2 \mod 8$ and so $\beta<v-2$ as otherwise $h=(\alpha +1)v-2$ which, under the assumptions of case (d), is equal to 6 modulo 8, a contradiction.  As any eigenvalue of $\overline{y}$  can be paired with its inverse,  $\overline{y}$ is an element of ${\rm SO}_{h+2}(K)$ of order $v$. Let $N$ be the  number of eigenvalues $\omega^i$ of $\overline{y}$ with $i$ odd. Then 
 \begin{eqnarray*}
 N& = &\left\{ \begin{array}{ll} \alpha  \cdot \frac{v}{2}+\left(\frac{\beta}2-1\right)  & \textrm{if} \ (v \mod 4, \beta \mod 4)=(0,2)\\
  \alpha  \cdot \left( \frac{v}{2}-1\right)+(\alpha+1)+\left(\frac{\beta}2-1\right)  & \textrm{if} \ (v \mod 4, \beta \mod 4)=(2,2)\\
  \end{array}\right. \\ 
&=&  \left\{ \begin{array}{ll} 
  \ell-2 & \textrm{if} \ (v \mod 4, \beta \mod 4)= (0,2)\\
  \ell-1  & \textrm{if} \ (v \mod 4, \beta \mod 4)=(2,2).\\
  \end{array}\right. \\ 
 \end{eqnarray*}
 From the assumptions of case (d)  or (e), it follows that $N$ is divisible by 4. Hence 
  by Lemma $\ref{l:soliftspin}$, $\overline{y}$ has a preimage $y$ in $G$ of order $v$.
 Also $\overline{y}$ has $\alpha+1$ eigenvalues equal to $1$, $\alpha+1$ eigenvalues equal to $-1$, $\alpha+1$ eigenvalues equal to $\omega^i$ where if case (d) holds then $i$ is any integer with $1 \leq |i| \leq  \beta/2-1$ or $|i|=v/2-2$ and if case $(e)$ holds then $i$ any  integer with $1  \leq |i| \leq  \beta/2-1$ or $|i|=v/2-1$, and every other eigenvalue of $\overline{y}$ occurs with multiplicity $\alpha$. Therefore $C_{{\rm SO}_{h+2}(K)}(\overline{y})^0=A_\alpha^{\frac\beta2} A_{\alpha-1}^{\frac{v}2-1-\frac\beta2}D_{\frac {\alpha+1}2}D_{\frac{\alpha+1}2}T_{\frac{v}2-1}$. \\
 Suppose that case (f) holds.  As any eigenvalue of $\overline{y}$  can be paired with its inverse,  $\overline{y}$ is an element of ${\rm SO}_{h+2}(K)$ of order $v$. Let $N$ be  the number of eigenvalues $\omega^i$ of $\overline{y}$ with $i$ odd. Then 
 \begin{eqnarray*}
 N& = &\left\{ \begin{array}{ll} \alpha  \cdot \frac{v}{2}+\left(\frac{\beta}2-2\right)& \textrm{if} \ \beta \neq v-4 \\
 \alpha  \cdot \frac{v}{2}+\frac{v}2  & \textrm{if} \ \beta=v-4\\
  \end{array}\right. \\ 
&=&  \left\{ \begin{array}{ll} \ell-3 & \textrm{if} \ \beta\neq v-4\\
  \ell+1& \textrm{if} \ \beta=v-4. \end{array}\right.
 \end{eqnarray*}
 From the assumptions of case (f), it follows that $N$ is divisible by 4. Hence 
  by Lemma $\ref{l:soliftspin}$, $\overline{y}$ has a preimage $y$ in $G$ of order $v$.
 Also $\overline{y}$ has $\alpha+1$ eigenvalues equal to $1$, $\alpha+1$ eigenvalues equal to $-1$, $\alpha+1$ eigenvalues equal to $\omega^i$ where if $\beta \neq v-4$ then  $i$ is any integer with $2  \leq |i| \leq  \beta/2$ or $|i|=v/2-2$ and if $\beta = v-4$ then  $i$ is any integer with $1  \leq |i| \leq  v/2-3$ or $|i|=v/2-1$, and every other eigenvalue of $\overline{y}$ occurs with multiplicity $\alpha$. Therefore $C_{{\rm SO}_{h+2}(K)}(\overline{y})^0=A_{\alpha}^{\frac \beta2} A_{\alpha-1}^{\frac v2-1-\frac \beta2}D_{\frac {\alpha+1}2}D_{\frac{\alpha+1}2}T_{\frac{v}2-1}$. \\ 
Suppose that case (g) holds.  As any eigenvalue of $\overline{y}$  can be paired with its inverse,  $\overline{y}$ is an element of ${\rm SO}_{h+2}(K)$ of order $v$. Let $N$ be  the number of eigenvalues $\omega^i$ of $\overline{y}$ with $i$ odd. Then 
 \begin{eqnarray*}
 N& = &\alpha \cdot \left(\frac v2-1\right)+\left(\frac{\beta}2-2\right)+(\alpha+1)\\
&=& \ell -2. \\ 
 \end{eqnarray*}
 From the assumptions of case (g), it follows that $N$ is divisible by 4. Hence 
  by Lemma $\ref{l:soliftspin}$, $\overline{y}$ has a preimage $y$ in $G$ of order $v$.
 Also $\overline{y}$ has $\alpha+1$ eigenvalues equal to $1$, $\alpha+1$ eigenvalues equal to $-1$, $\alpha+1$ eigenvalues equal to $\omega^i$ where $i$ is any integer with $2  \leq |i| \leq  \beta/2$ or $|i|=v/2-1$, and every other eigenvalue of $\overline{y}$ occurs with multiplicity $\alpha$. Therefore $C_{{\rm SO}_{h+2}(K)}(\overline{y})^0=A_{\alpha}^{\frac \beta2} A_{\alpha-1}^{\frac v2-1-\frac \beta2}D_{\frac {\alpha+1}2}D_{\frac{\alpha+1}2}T_{\frac{v}2-1}$. \\
 Suppose that case (h) holds. As any eigenvalue of $\overline{y}$  can be paired with its inverse,  $\overline{y}$ is an element of ${\rm SO}_{h+2}(K)$ of order $v$. Let $N$ be  the number of eigenvalues $\omega^i$ of $\overline{y}$ with $i$ odd. Then 
 $$ N= \alpha\cdot\left( \frac v2-1\right) + \frac \beta2+(\alpha+3)= \ell+2.$$
 From the assumptions of case (h), it follows that $N$ is divisible by 4. Hence 
  by Lemma $\ref{l:soliftspin}$, $\overline{y}$ has a preimage $y$ in $G$ of order $v$.
 Also $\overline{y}$ has $\alpha+1$ eigenvalues equal to $1$, $\alpha+3$ eigenvalues equal to $-1$, $\alpha+1$ eigenvalues equal to $\omega^i$ where $i$ is any integer with $1  \leq |i| \leq  v/2-2$,  and every other eigenvalue of $\overline{y}$ occurs with multiplicity $\alpha$. Therefore $C_{{\rm SO}_{h+2}(K)}(\overline{y})^0=A_{\alpha}^{\frac{v}2-2}A_{\alpha-1}D_{\frac {\alpha+1}2}D_{\frac{\alpha+3}2}T_{\frac{v}2-1}$. \\  
 Suppose that case (i) holds. As any eigenvalue of $\overline{y}$  can be paired with its inverse,  $\overline{y}$ is an element of ${\rm SO}_{h+2}(K)$ of order $v$. Let $N$ be  the number of eigenvalues $\omega^i$ of $\overline{y}$ with $i$ odd. Then 
 $$ N= (\alpha-1)\cdot\frac v2 + \left(\frac v2-2\right)+4= \ell+1.$$
 From the assumptions of case (i), it follows that $N$ is divisible by 4. Hence 
  by Lemma $\ref{l:soliftspin}$, $\overline{y}$ has a preimage $y$ in $G$ of order $v$.
 Also $\overline{y}$ has $\alpha+1$ eigenvalues equal to $1$, $\alpha-1$ eigenvalues equal to $-1$, $\alpha$ eigenvalues equal to $\omega^i$ where $i$ is any integer with $2  \leq |i| \leq  v/2-1$,  and every other eigenvalue of $\overline{y}$ occurs with multiplicity $\alpha+1$. Therefore $C_{{\rm SO}_{h+2}(K)}(\overline{y})^0=A_{\alpha}A_{\alpha-1}^{\frac{v}2-2}D_{\frac {\alpha+1}2}D_{\frac{\alpha-1}2}T_{\frac{v}2-1}$. \\  
  Suppose that case (j) holds.  As any eigenvalue of $\overline{y}$ can be paired with its inverse,  $\overline{y}$ is an element of ${\rm SO}_{h+2}(K)$ of order $v$. Let $N$ be  the number of eigenvalues $\omega^i$ of $\overline{y}$ with $i$ odd. Then 
 $$N= (\alpha-1)\cdot\left(\frac v2-1 \right)+ \left(\frac v2-1\right)+2+(\alpha+1)= \ell+2.$$
 From the assumptions of case (j), it follows that $N$ is divisible by 4. Hence 
  by Lemma $\ref{l:soliftspin}$, $\overline{y}$ has a preimage $y$ in $G$ of order $v$.
 Also $\overline{y}$ has $\alpha+1$ eigenvalues equal to $1$, $\alpha+1$ eigenvalues equal to $-1$, $\alpha+1$ eigenvalues equal to $\omega$, $\alpha+1$ eigenvalues equal to $\omega^{-1}$, $\alpha$ eigenvalues equal to $\omega^{i}$  where  $i$ is any integer with $2 \leq |i| \leq  v/2-2$, $\alpha-1$ eigenvalues equal to $\omega^{\frac v2-1}$ and $\alpha-1$ eigenvalues equal to $\omega^{-(\frac{v}2-1)}$. Therefore $C_{{\rm SO}_{h+2}(K)}(\overline{y})^0=A_\alpha A_{\alpha-1}^{\frac v2-3} A_{\alpha-2}D_{\frac {\alpha+1}2}D_{\frac{\alpha+1}2}T_{\frac{v}2-1}$. \\
 Suppose that case (k) holds.  As any eigenvalue of $\overline{y}$  can be paired with its inverse,  $\overline{y}$ is an element of ${\rm SO}_{h+2}(K)$ of order $v$. Let $N$ be  the number of eigenvalues $\omega^i$ of $\overline{y}$ with $i$ odd. Then 
 \begin{eqnarray*}
 N& = &\left\{ \begin{array}{ll} \alpha  \cdot \frac{v}{2}+\frac{\beta}2& \textrm{if} \ (v \mod 4, \beta \mod 4)=(0,0)\\
 \alpha  \cdot \frac{v}{2}+\left(\frac{\beta}2+1\right)  & \textrm{if} \ (v \mod 4, \beta \mod 4)=(0,2)\\
  \alpha  \cdot\left( \frac{v}{2}-1\right)+\alpha+\frac{\beta}2 & \textrm{if} \ (v \mod 4, \beta \mod 4)=(2,0)\\
  \alpha  \cdot \left( \frac{v}{2}-1\right)+\alpha+\left(\frac{\beta}2+1\right)  & \textrm{if} \ (v \mod 4, \beta \mod 4)=(2,2)\\
  \end{array}\right. \\ 
&=&  \left\{ \begin{array}{ll} \ell-1 & \textrm{if} \ (v \mod 4, \beta \mod 4) \in \{(0,0), (2,0)\}\\
  \ell & \textrm{if} \ (v \mod 4, \beta \mod 4)\in\{(0,2),(2,2)\}.\\
     \end{array}\right. 
 \end{eqnarray*}

 From the assumptions of case (k), it follows that $N$ is divisible by 4. Hence 
  by Lemma $\ref{l:soliftspin}$, $\overline{y}$ has a preimage $y$ in $G$ of order $v$.
 Also $\overline{y}$ has $\alpha+2$ eigenvalues equal to $1$, $\alpha$ eigenvalues equal to $-1$, $\alpha+1$ eigenvalues equal to $\omega^i$ where $i$ is any integer with $1  \leq |i| \leq  \beta/2$, and every other eigenvalue of $\overline{y}$ occurs with multiplicity $\alpha$. Therefore $C_{{\rm SO}_{h+2}(K)}(\overline{y})^0=A_{\alpha}^{\frac \beta2}A_{\alpha-1}^{\frac v2-1-\frac \beta2} D_{\frac {\alpha}2}D_{\frac{\alpha+2}2}T_{\frac{v}2-1}$. \\  
  Suppose that case (l) holds.  As any eigenvalue of $\overline{y}$ can be paired with its inverse,  $\overline{y}$ is an element of ${\rm SO}_{h+2}(K)$ of order $v$. Let $N$ be  the number of eigenvalues $\omega^i$ of $\overline{y}$ with $i$ odd. Then 
 $$N= \alpha\cdot\left(\frac v2-1 \right)+\alpha+ \left(\frac \beta2-1\right)= \ell-2.$$
 From the assumptions of case (l), it follows that $N$ is divisible by 4. Hence 
  by Lemma $\ref{l:soliftspin}$, $\overline{y}$ has a preimage $y$ in $G$ of order $v$.
 Also $\overline{y}$ has $\alpha+2$ eigenvalues equal to $1$, $\alpha$ eigenvalues equal to $-1$, $\alpha+1$ eigenvalues equal to $\omega^{i}$  where  $i$ is any integer with $1 \leq |i| \leq  \beta/2-1$ or $|i|=v/2-1$, and  every other eigenvalue of $\overline{y}$ occurs with multiplicity $\alpha$. Therefore $C_{{\rm SO}_{h+2}(K)}(\overline{y})^0=A_{\alpha}^{\frac \beta2}A_{\alpha-1}^{\frac v2-1-\frac \beta2} D_{\frac {\alpha}2}D_{\frac{\alpha+2}2}T_{\frac{v}2-1}$. \\
 Suppose that case (m) holds. As any eigenvalue of $\overline{y}$ can be paired with its inverse,  $\overline{y}$ is an element of ${\rm SO}_{h+2}(K)$ of order $v$. Let $N$ be the number of eigenvalues $\omega^i$ of $\overline{y}$ with $i$ odd. Then 
  \begin{eqnarray*}
 N& = &\left\{ \begin{array}{ll} \alpha  \cdot \frac{v}{2}+\left(\frac{\beta}2-2\right)& \textrm{if} \ \beta \neq v-4\\
 \alpha  \cdot \frac{v}{2}+\left(\frac{\beta}2+2\right)  & \textrm{if} \ \beta=v-4\\
  \end{array}\right. \\ 
&=&  \left\{ \begin{array}{ll} \ell-3 & \textrm{if} \ \beta \neq v-4\\
  \ell+1 & \textrm{if} \ \beta=v-4.\\
     \end{array}\right. 
 \end{eqnarray*}
 From the assumptions of case (m), it follows that $N$ is divisible by 4. Hence 
  by Lemma $\ref{l:soliftspin}$, $\overline{y}$ has a preimage $y$ in $G$ of order $v$.
 Also if $\beta \neq v-4$ then $\overline{y}$ has $\alpha+2$ eigenvalues equal to $1$, $\alpha$ eigenvalues equal to $-1$, $\alpha+1$ eigenvalues equal to $\omega^i$ where $i$ is any integer with $2  \leq |i| \leq  \beta/2$ or $|i|=v/2-2$, and every other eigenvalue of $\overline{y}$ occurs with multiplicity $\alpha$. On the other hand if $\beta=v-4$ then $\overline{y}$ has $\alpha+2$ eigenvalues equal to $1$, $\alpha$ eigenvalues equal to $-1$, $\alpha+1$ eigenvalues equal to $\omega^i$ where $i$ is any integer with $1  \leq |i| \leq  \beta/2-1$ or $|i|=\beta/2+1$, and every other eigenvalue of $\overline{y}$ occurs with multiplicity $\alpha$.  Therefore in both cases $C_{{\rm SO}_{h+2}(K)}(\overline{y})^0=A_\alpha^{\frac\beta2} A_{\alpha-1}^{\frac{v}2-1-\frac\beta2}D_{\frac {\alpha}2}D_{\frac{\alpha+2}2}T_{\frac{v}2-1}$. \\
Suppose that case (n) holds. As any eigenvalue of $\overline{y}$ can be paired with its inverse,  $\overline{y}$ is an element of ${\rm SO}_{h+2}(K)$ of order $v$. Let $N$ be the number of eigenvalues $\omega^i$ of $\overline{y}$ with $i$ odd. Then 
 $$N =  \alpha  \cdot \frac{v}{2}+\left(\frac{\beta}2-1\right)=\ell-2.$$ 
 From the assumptions of case (n), it follows that $N$ is divisible by 4. Hence 
  by Lemma $\ref{l:soliftspin}$, $\overline{y}$ has a preimage $y$ in $G$ of order $v$.
 Also if $\beta \neq v-2$ then $\overline{y}$ has $\alpha+2$ eigenvalues equal to $1$, $\alpha$ eigenvalues equal to $-1$, $\alpha+1$ eigenvalues equal to $\omega^i$ where $i$ is any integer with $1  \leq |i| \leq  \beta/2-1$ or $|i|=v/2-2$, and every other eigenvalue of $\overline{y}$ occurs with multiplicity $\alpha$. On the other hand if $\beta=v-2$ then $\overline{y}$ has $\alpha+2$ eigenvalues equal to $1$, $\alpha+2$ eigenvalues equal to $-1$, $\alpha+1$ eigenvalues equal to $\omega^i$ where $i$ is any integer with $1  \leq |i| \leq  v/2-2$, and every other eigenvalue of $\overline{y}$ occurs with multiplicity $\alpha$.  Therefore 
 $$C_{{\rm SO}_{h+2}(K)}(\overline{y})^0=
 \left\{ \begin{array}{ll}
 A_\alpha^{\frac\beta2} A_{\alpha-1}^{\frac{v}2-1-\frac\beta2}D_{\frac {\alpha}2}D_{\frac{\alpha+2}2}T_{\frac{v}2-1} & \textrm{if} \ \beta \neq v-2\\
 A_\alpha^{\frac v2-2} A_{\alpha-1}D_{\frac {\alpha+2}2}D_{\frac{\alpha+2}2}T_{\frac{v}2-1} & \textrm{if} \ \beta=v-2.
 \end{array} \right. 
 $$
 Suppose that case (o) holds. As any eigenvalue of $\overline{y}$ can be paired with its inverse,  $\overline{y}$ is an element of ${\rm SO}_{h+2}(K)$ of order $v$. Let $N$ be  the number of eigenvalues $\omega^i$ of $\overline{y}$ with $i$ odd. Then 
 \begin{eqnarray*}
 N & = &\left\{\begin{array}{ll}   \alpha  \cdot \left(\frac{v}{2}-1\right)+\left(\frac{\beta}2-2\right)+\alpha & \textrm{if} \ \beta\neq v-2\\
   \alpha  \cdot \left(\frac{v}{2}-1\right)+\frac{\beta}2+(\alpha+2) & \textrm{if} \ \beta = v-2\\
   \end{array}\right. \\
   & = & \left\{\begin{array}{ll}
  \ell-3 & \textrm{if} \ \beta\neq v-2\\
  \ell+1 & \textrm{if} \ \beta = v-2.\\
   \end{array}\right. \\
 \end{eqnarray*} 
 From the assumptions of case (o), it follows that $N$ is divisible by 4. Hence 
  by Lemma $\ref{l:soliftspin}$, $\overline{y}$ has a preimage $y$ in $G$ of order $v$.
 Also if $\beta \neq v-2$ then $\overline{y}$ has $\alpha+2$ eigenvalues equal to $1$, $\alpha$ eigenvalues equal to $-1$, $\alpha+1$ eigenvalues equal to $\omega^i$ where $i$ is any integer with $2  \leq |i| \leq  \beta/2$ or $|i|=v/2-1$, and every other eigenvalue of $\overline{y}$ occurs with multiplicity $\alpha$. On the other hand if $\beta=v-2$ then $\overline{y}$ has $\alpha+2$ eigenvalues equal to $1$, $\alpha+2$ eigenvalues equal to $-1$, $\alpha+1$ eigenvalues equal to $\omega^i$ where $i$ is any integer with $1  \leq |i| \leq  v/2-2$, and every other eigenvalue of $\overline{y}$ occurs with multiplicity $\alpha$.  Therefore 
 $$C_{{\rm SO}_{h+2}(K)}(\overline{y})^0=
 \left\{ \begin{array}{ll}
 A_\alpha^{\frac\beta2} A_{\alpha-1}^{\frac{v}2-1-\frac\beta2}D_{\frac {\alpha}2}D_{\frac{\alpha+2}2}T_{\frac{v}2-1} & \textrm{if} \ \beta \neq v-2\\
 A_\alpha^{\frac v2-2} A_{\alpha-1}D_{\frac {\alpha+2}2}D_{\frac{\alpha+2}2}T_{\frac{v}2-1} & \textrm{if} \ \beta=v-2.
 \end{array} \right. 
 $$
Suppose that case (p) holds. As any eigenvalue of $\overline{y}$ can be paired with its inverse,  $\overline{y}$ is an element of ${\rm SO}_{h+2}(K)$ of order $v$. Let $N$ be  the number of eigenvalues $\omega^i$ of $\overline{y}$ with $i$ odd. Then 
$$ N= \alpha\cdot\left(\frac v2-1\right)+(\alpha+2)=\ell+1.$$
From the assumptions of case (p), it follows that $N$ is divisible by 4. Hence 
  by Lemma $\ref{l:soliftspin}$, $\overline{y}$ has a preimage $y$ in $G$ of order $v$.
 Also  $\overline{y}$ has $\alpha$ eigenvalues equal to $1$, $\alpha+2$ eigenvalues equal to $-1$, and every other eigenvalue of $\overline{y}$ occurs with multiplicity $\alpha$. Therefore
   $$C_{{\rm SO}_{h+2}(K)}(\overline{y})^0=A_{\alpha-1}^{\frac v2-1} D_{\frac {\alpha}2}D_{\frac{\alpha+2}2}T_{\frac{v}2-1} .$$
 By Proposition \ref{p:prelimdu}(i), $d_u(G)\leq d_q(C_G(y)^0)$ and so $d_q(C_G(y)^0)$ is an upper bound for $d_u(G)$. It remains to calculate $d_q(C_{{\rm SO}_{h+2}}(\overline{y})^0)=d_q(C_G(y)^0)$.\\

It follows from Lemma \ref{l:lawther} that 
 $$d_q(A_{\alpha}) = z^2q +(2z+1)(\alpha-zq+1)-1,$$
$$d_q(A_{\alpha-1})=z^2q+(2z+1)(\alpha-zq)-1,$$
$$d_q(A_{\alpha-2})=\left\{\begin{array}{ll}
z^2q+(2z+1)(\alpha-1-zq)-1 & \textrm{if} \ \delta >0\\
(z-1)^2q+(2z-1)(\alpha-1-(z-1)q)-1 & \textrm{if} \ \delta=0,
 \end{array}\right.$$
$$d_q\left(D_{\lceil\frac{\alpha+1}2\rceil}\right)=
\frac{z^2q}{2}+(2z+1)\left(\left\lceil\frac{\alpha+1}{2}\right\rceil-\frac{zq}2\right)+\left\lceil \frac z2\right\rceil\epsilon_q-z-\epsilon_z,$$
if $\alpha$ is odd then 
$$d_q\left(D_{\frac{\alpha-1}2}\right)=
\left\{\begin{array}{ll}
\frac{z^2q}{2}+(2z+1)\left(\frac{\alpha-1}{2}-\frac{zq}2\right)+\left\lceil \frac z2\right\rceil\epsilon_q - z - \epsilon_z& \textrm{if} \  \delta>0\\
\frac{(z-1)^2q}{2}+(2z-1)\left(\frac{\alpha-1}{2}-\frac{(z-1)q}2\right)+\frac {z-1}2\epsilon_q-(z-1)& \textrm{if} \ \delta=0,
\end{array}\right.$$
$$d_q\left(D_{\frac{\alpha+3}2}\right)=
\left\{\begin{array}{ll}
\frac{z^2q}{2}+(2z+1)\left(\frac{\alpha+3}{2}-\frac{zq}2\right)+\left\lceil \frac z2\right\rceil\epsilon_q - z - \epsilon_z& \textrm{if} \  \delta<q-1\\
\frac{(z+1)^2q}{2}+(2z+3)\left(\frac{\alpha+3}{2}-\frac{(z+1)q}2\right)+\frac {z+1}2\epsilon_q-(z+1)-\epsilon_{z+1}& \textrm{if} \ \delta=q-1,
\end{array}\right.$$
and if $\alpha$ is even then
$$d_q\left(D_{\frac{\alpha}2}\right)=\left\{\begin{array}{ll}
\frac{z^2q}{2}+(2z+1)\left(\frac{\alpha}{2}-\frac{zq}2\right)+\left\lceil \frac z2\right\rceil\epsilon_q-z-\epsilon_z & \textrm{if} \ \delta>0 \ \textrm{or} \ z \  \textrm{is even}\\
\frac{z^2q}{2}+(2z+1)\left(\frac{\alpha}{2}-\frac{zq}2\right)+\left\lceil \frac z2\right\rceil\epsilon_q-z-\epsilon_z+2 & \textrm{otherwise}.  
\end{array}\right.
$$

Assume first that case (a) holds so that  $\epsilon_v=1$. Recall that  $\epsilon_u=\epsilon_q$ and $\epsilon_\alpha=\epsilon_\beta$. We  have $$C_{{\rm SO}_{h+2}(K)}(\overline{y})^0=A_\alpha^{\lceil\frac\beta2\rceil} A_{\alpha-1}^{\frac{v-1}2-\lceil\frac\beta2\rceil}D_{\lceil \frac {\alpha+1}2\rceil}T_{\frac{v-1}2}.$$
By Proposition \ref{p:prelimdu}(iii),
\begin{eqnarray*}
d_q(C_{{\rm SO}_{h+2}(K)}(\overline{y})^0) & = & \frac{v-1}2+d_q\left(A_\alpha^{\lceil\frac\beta2\rceil} A_{\alpha-1}^{\frac{v-1}2-\lceil\frac\beta2\rceil}D_{\lceil \frac {\alpha+1}2\rceil}\right)\\
& = & \frac{v-1}2+\left\lceil \frac \beta2 \right \rceil d_q(A_\alpha)+\left(\frac{v-1}2-\left \lceil \frac \beta2 \right \rceil\right)d_q(A_{\alpha-1})+d_q\left(D_{\lceil \frac{\alpha+1}2\rceil}\right)\\
& =  &z(\alpha v+ \beta)-\frac{qvz(z+1)}2+\frac{\alpha v +\beta}2+\left\lceil\frac{z}2\right\rceil\epsilon_q+z+1-\epsilon_z\\
& = & z(zu+e)-\frac{zu(z+1)}2+\frac{zu+e}2+\left\lceil\frac{z}2\right\rceil\epsilon_q+z+1-\epsilon_z\\
& = &  \frac12(z^2u+e(2z+1))+\left\lceil\frac{z}2\right\rceil\epsilon_u+z+1-\epsilon_z.
\end{eqnarray*}

Assume that   case (b) holds so that $v=2$, $\epsilon_u=\epsilon_v=0$ and $\epsilon_q=1$. Recall that $\alpha=\ell-1$ and $\beta=0$.  Note also that $e=0$ if and only if $\delta=0$. We have 
 $$C_{{\rm SO}_{h+2}(K)}(\overline{y})^0= \left\{ \begin{array}{ll}
 D_{\frac{\alpha}2}D_{\frac{\alpha+2}2} & \textrm{if} \ \alpha \equiv 0 \ (2)\\ 
 D_{\frac{\alpha-1}2}D_{\frac{\alpha+3}2} & \textrm{if} \ \alpha \equiv 1 \ (4)\\ 
D_{\frac{\alpha+1}2}D_{\frac{\alpha+1}2} & \textrm{if} \ \alpha \equiv 3 \ (4).\\ 
\end{array}\right.
$$
If $\alpha \equiv 0 \mod 2$ then $z$ is even if $\delta=0$, and by Proposition \ref{p:prelimdu}(iii)
\begin{eqnarray*}
d_q(C_{{\rm SO}_{h+2}(K)}(\overline{y})^0) & = &  d_q(D_{\frac{\alpha}2}D_{\frac{\alpha+2}2})\\
& = & d_q(D_{\frac{\alpha}2})+d_q(D_{\frac{\alpha+2}2})\\
& = &   -qz(z+1)+2\alpha z+ \alpha+z+1-\epsilon_z\\
& = & \frac{-qvz(z+1)}2+(zu+e)z+\frac{(zu+e)}2+z+1-\epsilon_z\\ 
& = &  \frac{-zu(z+1)}2+(zu+e)z+\frac{(zu+e)}2+z+1-\epsilon_z\\  
& = & \frac{z^2u+e(2z+1)}2+z+1-\epsilon_z. 
\end{eqnarray*}
If $\alpha \equiv 1 \mod 4$ then  by Proposition \ref{p:prelimdu}(iii)
\begin{eqnarray*}
d_q(C_{{\rm SO}_{h+2}(K)}(\overline{y})^0) & = &  d_q(D_{\frac{\alpha-1}2}D_{\frac{\alpha+3}2})\\
& = & d_q(D_{\frac{\alpha-1}2})+d_q(D_{\frac{\alpha+3}2})\\
& = &  \left\{\begin{array}{ll} -z^2+2\alpha z+\alpha+4& \textrm{if} \ \delta=q-1=0 \\
-qz^2+(2\alpha+1) z +2& \textrm{if} \ \delta=0\neq q-1\\
-qz^2+(2\alpha-2q+1)z+2\alpha-q+3& \textrm{if} \ \delta=q-1\neq 0 \\
-qz^2+(2\alpha-q+1) z +\alpha+1-\epsilon_z& \textrm{if} \ \delta\not \in \{0, q-1\} 
 \end{array}\right. \\
 & = &  \left\{\begin{array}{ll} \frac{-zu(z+1)}2+(zu+e) z+\frac{zu+e}2+z+1-\epsilon_z+4& \textrm{if} \ u=2 \\
\frac{-qvz(z+1)}2+\alpha v z+ \alpha+z+1-\epsilon_z+2& \textrm{if} \ e=0\neq u-v\\
\frac{-qvz(z+1)}2+\alpha v z+\alpha+z+1-\epsilon_z+2& \textrm{if} \ e=u-v\neq 0 \\
\frac{-qvz(z+1)}2+\alpha vz +\alpha+z+1-\epsilon_z&   \textrm{otherwise}
 \end{array}\right. \\
 & = &  \left\{\begin{array}{ll} \frac{-zu(z+1)}2+(zu+e) z+\frac{zu+e}2+z+1-\epsilon_z+4& \textrm{if} \ u=2 \\
\frac{-zu(z+1)}2+(zu+e)z+\frac{zu+e}2+z+1-\epsilon_z+2& \textrm{if} \ e=0\neq u-v\\
\frac{-zu(z+1)}2+(zu+e)z+\frac{zu+e}2+z+1-\epsilon_z+2& \textrm{if} \ e=u-v\neq 0 \\
\frac{-zu(z+1)}2+(zu+e)z+\frac{zu+e}2+z+1-\epsilon_z&   \textrm{otherwise}
 \end{array}\right. \\
 & = &  \left\{\begin{array}{ll} \frac{z^2u+e(2z+1)}2+z+1-\epsilon_z+4& \textrm{if} \ u=2 \\
 \frac{z^2u+e(2z+1)}2+z+1-\epsilon_z+2& \textrm{if} \ e=0\neq u-v\\
 \frac{z^2u+e(2z+1)}2+z+1-\epsilon_z+2& \textrm{if} \ e=u-v\neq 0 \\
 \frac{z^2u+e(2z+1)}2+z+1-\epsilon_z&   \textrm{otherwise}.
 \end{array}\right. 
\end{eqnarray*}
If $\alpha \equiv 3 \mod 4$ then by Proposition \ref{p:prelimdu}(iii)
\begin{eqnarray*}
d_q(C_{{\rm SO}_{h+2}(K)}(\overline{y})^0) & = &  d_q(D_{\frac{\alpha+1}2}D_{\frac{\alpha+1}2})\\
& = & 2d_q(D_{\frac{\alpha+1}2})\\
& = & -qz(z+1)+2\alpha z+\alpha+z+1-\epsilon_z\\
& = & \frac{-zu(z+1)}2+(zu+e)z+\frac{zu+e}2+z+1-\epsilon_z \\
& = &  \frac{z^2u+e(2z+1)}2+z+1-\epsilon_z. 
\end{eqnarray*}

Assume that case (c), (d), (e), (f) or (g) holds.  Recall that  $\epsilon_q=\epsilon_\alpha=1$ and $\epsilon_u=\epsilon_v=\epsilon_{\beta}=0$. We  have $$C_{{\rm SO}_{h+2}(K)}(\overline{y})^0=A_\alpha^{\frac\beta2} A_{\alpha-1}^{\frac{v}2-1-\frac\beta2}D_{\frac{\alpha+1}2}D_{\frac{\alpha+1}2}T_{\frac{v}2-1}.$$
By Proposition \ref{p:prelimdu}(iii),
\begin{eqnarray*}
d_q(C_{{\rm SO}_{h+2}(K)}(\overline{y})^0) & = & \frac{v}2-1+d_q\left(A_\alpha^{\frac\beta2} A_{\alpha-1}^{\frac{v}2-1-\frac\beta2}D_{\frac{\alpha+1}2}D_{\frac{\alpha+1}2}\right)\\
& = & \frac{v}2-1+\frac\beta2  d_q(A_\alpha)+\left(\frac{v}2-1- \frac \beta2 \right)d_q(A_{\alpha-1})+2d_q\left(D_{\frac{\alpha+1}2}\right)\\
& =  &z(\alpha v+ \beta)-\frac{qvz(z+1)}2+\frac{\alpha v +\beta}2+z+1-\epsilon_z\\
& = & z(zu+e)-\frac{zu(z+1)}2+\frac{zu+e}2+z+1-\epsilon_z\\
& = &   \frac{z^2u+e(2z+1)}2+z+1-\epsilon_z. 
\end{eqnarray*}

Assume that case (h)  holds.  Recall that  $\epsilon_q=\epsilon_\alpha=1$,  $\epsilon_u=\epsilon_v=\epsilon_{\beta}=0$ and $\beta=v-2$. We  have $$C_{{\rm SO}_{h+2}(K)}(\overline{y})^0=A_\alpha^{\frac v2-2} A_{\alpha-1}D_{\frac{\alpha+1}2}D_{\frac{\alpha+3}2}T_{\frac{v}2-1}.$$
By Proposition \ref{p:prelimdu}(iii),
\begin{eqnarray*}
d_q(C_{{\rm SO}_{h+2}(K)}(\overline{y})^0) & = & \frac{v}2-1+d_q\left(A_\alpha^{\frac v2-2} A_{\alpha-1}D_{\frac{\alpha+1}2}D_{\frac{\alpha+3}2}\right)\\
& = & \frac{v}2-1+\left(\frac v2-2\right)  d_q(A_\alpha)+d_q(A_{\alpha-1})+d_q\left(D_{\frac{\alpha+1}2}\right)+d_q\left(D_{\frac{\alpha+3}2}\right)\\
& =  &\left\{ \begin{array}{ll} z(\alpha v+ \beta)-\frac{qvz(z+1)}2+\frac{\alpha v +\beta}2+z+1-\epsilon_z& \textrm{if} \ 
\delta<q-1\\
z(\alpha v+ \beta)-\frac{qvz(z+1)}2+\frac{\alpha v +\beta}2+z+1-\epsilon_z+2&  \textrm{if} \ \delta=q-1\end{array}\right.
\\
& = & \left\{ \begin{array}{ll}  \frac{z^2u+e(2z+1)}2+z+1-\epsilon_z& \textrm{if} \ 
e\neq u-v+\beta\\
 \frac{z^2u+e(2z+1)}2+z+1-\epsilon_z+2& \textrm{if} \ e=u-v+\beta
\end{array}\right.
\end{eqnarray*}

Assume that case (i)  holds.  Recall that  $\epsilon_q=\epsilon_\alpha=1$,  $\epsilon_u=\epsilon_v=\epsilon_{\beta}=0$ and $\beta=0$. We  have $$C_{{\rm SO}_{h+2}(K)}(\overline{y})^0=A_{\alpha}A_{\alpha-1}^{\frac v2-2} D_{\frac{\alpha-1}2}D_{\frac{\alpha+1}2}T_{\frac{v}2-1}.$$
By Proposition \ref{p:prelimdu}(iii),
\begin{eqnarray*}
d_q(C_{{\rm SO}_{h+2}(K)}(\overline{y})^0) & = & \frac{v}2-1+d_q\left(A_{\alpha}A_{\alpha-1}^{\frac v2-2} D_{\frac{\alpha-1}2}D_{\frac{\alpha+1}2}\right)\\
& = & \frac{v}2-1+d_q(A_\alpha)+\left(\frac v2-2\right)  d_q(A_{\alpha-1})+d_q\left(D_{\frac{\alpha-1}2}\right)+d_q\left(D_{\frac{\alpha+1}2}\right)\\
& =  &\left\{ \begin{array}{ll} z(\alpha v+ \beta)-\frac{qvz(z+1)}2+\frac{\alpha v +\beta}2+z+1-\epsilon_z+2& \textrm{if} \ 
\delta=0\\
z(\alpha v+ \beta)-\frac{qvz(z+1)}2+\frac{\alpha v +\beta}2+z+1-\epsilon_z&  \textrm{if} \ \delta\neq0
\end{array}\right.
\\
& = & \left\{ \begin{array}{ll}  \frac{z^2u+e(2z+1)}2+z+1-\epsilon_z +2& \textrm{if} \ 
e=0\\
\frac{z^2u+e(2z+1)}2+z+1-\epsilon_z &  \textrm{otherwise}.
\end{array}\right.
\end{eqnarray*}

Assume that case (j)  holds.  Recall that  $\epsilon_q=\epsilon_\alpha=1$,  $\epsilon_u=\epsilon_v=\epsilon_{\beta}=0$ and $\beta=0$. We  have $$C_{{\rm SO}_{h+2}(K)}(\overline{y})^0=A_{\alpha}A_{\alpha-1}^{\frac v2-3}A_{\alpha-2}D_{\frac{\alpha+1}2}D_{\frac{\alpha+1}2}T_{\frac{v}2-1}.$$
By Proposition \ref{p:prelimdu}(iii),
\begin{eqnarray*}
d_q(C_{{\rm SO}_{h+2}(K)}(\overline{y})^0) & = & \frac{v}2-1+d_q\left(A_{\alpha}A_{\alpha-1}^{\frac v2-3}A_{\alpha-2}D_{\frac{\alpha+1}2}D_{\frac{\alpha+1}2}\right)\\
& = & \frac{v}2-1+d_q(A_{\alpha})+\left(\frac v2-3\right)  d_q(A_{\alpha-1})+d_q(A_{\alpha-2})+2d_q\left(D_{\frac{\alpha+1}2}\right)\\
& =  &\left\{ \begin{array}{ll} z(\alpha v+ \beta)-\frac{qvz(z+1)}2+\frac{\alpha v +\beta}2+z+1-\epsilon_z& \textrm{if} \ 
\delta>0\\
z(\alpha v+ \beta)-\frac{qvz(z+1)}2+\frac{\alpha v +\beta}2+z+1-\epsilon_z+2&  \textrm{if} \ \delta=0\end{array}\right.
\\
& = & \left\{ \begin{array}{ll}  \frac{z^2u+e(2z+1)}2+z+1-\epsilon_z& \textrm{if} \ 
e>0\\
 \frac{z^2u+e(2z+1)}2+z+1-\epsilon_z+2& \textrm{if} \ e=0.
\end{array}\right.
\end{eqnarray*}

Assume that case (k), (l), (m), (n) or (o)  holds, and that $\beta\neq v-2$ if in cases (n) or (o).   Recall that  $\epsilon_q=1$ and $\epsilon_u=\epsilon_v=\epsilon_\alpha=\epsilon_{\beta}=0$. We  have $$C_{{\rm SO}_{h+2}(K)}(\overline{y})^0=A_{\alpha}^\frac\beta2A_{\alpha-1}^{\frac v2-1 -\frac \beta2}D_{\frac{\alpha}2}D_{\frac{\alpha+2}2}T_{\frac{v}2-1}.$$
By Proposition \ref{p:prelimdu}(iii),
\begin{eqnarray*}
d_q(C_{{\rm SO}_{h+2}(K)}(\overline{y})^0) & = & \frac{v}2-1+d_q\left(A_{\alpha}^\frac\beta2A_{\alpha-1}^{\frac v2-1 -\frac \beta2}D_{\frac{\alpha}2}D_{\frac{\alpha+2}2}\right)\\
& = & \frac{v}2-1+\frac\beta2d_q(A_{\alpha})+\left(\frac v2-1-\frac{\beta}2\right)  d_q(A_{\alpha-1})+d_q(D_{\frac{\alpha}2})+d_q\left(D_{\frac{\alpha+2}2}\right)\\
& =  & z(\alpha v+ \beta)-\frac{qvz(z+1)}2+\frac{\alpha v +\beta}2+z+1-\epsilon_z\\
& = & \frac{z^2u+e(2z+1)}2+z+1-\epsilon_z.
\end{eqnarray*}

Assume that case (n) or (o) holds, and $\beta= v-2$. Recall that  $\epsilon_q=1$ and $\epsilon_u=\epsilon_v=\epsilon_\alpha=\epsilon_{\beta}=0$. We  have $$C_{{\rm SO}_{h+2}(K)}(\overline{y})^0=A_{\alpha}^{\frac v2-2}A_{\alpha-1}D_{\frac{\alpha+2}2}D_{\frac{\alpha+2}2}T_{\frac{v}2-1}.$$
By Proposition \ref{p:prelimdu}(iii),
\begin{eqnarray*}
d_q(C_{{\rm SO}_{h+2}(K)}(\overline{y})^0) & = & \frac{v}2-1+d_q\left(A_{\alpha}^{\frac v2-2}A_{\alpha-1}D_{\frac{\alpha+2}2}D_{\frac{\alpha+2}2}\right)\\
& = & \frac{v}2-1+\left(\frac v2-2\right)d_q(A_{\alpha})+  d_q(A_{\alpha-1})+2d_q\left(D_{\frac{\alpha+2}2}\right)\\
& =  & z(\alpha v+ \beta)-\frac{qvz(z+1)}2+\frac{\alpha v +\beta}2+z+1-\epsilon_z\\
& = & \frac{z^2u+e(2z+1)}2+z+1-\epsilon_z.
\end{eqnarray*}

Assume finally that case (p) holds. Recall that  $\epsilon_q=1$, $\epsilon_u=\epsilon_v=\epsilon_\alpha=\epsilon_{\beta}=0$ and $\beta=0$.
We  have $$C_{{\rm SO}_{h+2}(K)}(\overline{y})^0=A_{\alpha-1}^{\frac v2-1}D_{\frac{\alpha}2}D_{\frac{\alpha+2}2}T_{\frac{v}2-1}.$$
By Proposition \ref{p:prelimdu}(iii),
\begin{eqnarray*}
d_q(C_{{\rm SO}_{h+2}(K)}(\overline{y})^0) & = & \frac{v}2-1+d_q\left(A_{\alpha-1}^{\frac v2-1}D_{\frac{\alpha}2}D_{\frac{\alpha+2}2}\right)\\
& = & \frac{v}2-1+\left(\frac v2-1\right)  d_q(A_{\alpha-1})+d_q\left(D_{\frac{\alpha}2}\right)+d_q\left(D_{\frac{\alpha+2}2}\right)\\
& =  & z(\alpha v+ \beta)-\frac{qvz(z+1)}2+\frac{\alpha v +\beta}2+z+1-\epsilon_z\\
& = &  \frac{z^2u+e(2z+1)}2+z+1-\epsilon_z.
\end{eqnarray*}

\end{proof}

\section{Proofs of Theorems \ref{t:blawther}  and   \ref{t:duotsca}}\label{s:pc}

Recall that in \S\ref{s:exceptionalgroups} we proved Theorem \ref{t:blawther} for $G$ of exceptional type, see Propositions \ref{p:e6sc} and \ref{p:e7sc}.   
In this section  given a  positive integer $u$ and  a simple algebraic group $G$ of classical type over an algebraically closed field $K$ of characteristic $p$, we determine $d_u(G)$. 
We begin with the case where $G$ is of simply connected type and complete the proof of Theorem \ref{t:blawther}. We will then consider $G$ of neither  simply connected type nor adjoint type and prove Theorem \ref{t:duotsca}.

\noindent \textit{Proof of Theorem \ref{t:blawther}.}
As noticed at the beginning of this section, following Propositions \ref{p:e6sc} and \ref{p:e7sc} we can assume that $G$ is simply connected of classical type.  By \cite[Theorem 1]{Lawther}, where Lawther determines $d_u(G_{a.})$  and shows that $d_u(G)\geq d_u(G_{a.})$, and  Lemmas \ref{l:asc}, \ref{l:csc}, \ref{l:bsc} and \ref{l:dsc}, where the upper bounds for $d_u(G)$ are determined accordingly respectively as $G$ is of type $A_\ell$, $C_\ell$, $B_\ell$, or $D_\ell$, we get that $d_u(G)=d_u(G_{a.})$ except possibly if one of the  cases in Table \ref{ta:casestoconsider} (see \S\ref{s:tables}) holds.

It remains to show that in the cases appearing in Table \ref{ta:casestoconsider} the upper bounds given for $d_u(G)-d_u(G_{a.})$ are in fact the precise values for  $d_u(G)-d_u(G_{a.})$. As usual, we let $h$ be the Coxeter number of $G$ and  write:  $h=zu+e=\alpha v +\beta$ and $\alpha=zq+\delta$ where $z$, $e$, $\alpha$, $\beta$, $\delta$ are nonnegative integers such that $e<u$, $\beta<v$, and $\delta<q$. Noting that in the cases appearing in Table \ref{ta:casestoconsider}  $u$ is even, we set $s=v/2$. \\
Suppose first that $G$ is of type $A_\ell$.  By Table \ref{ta:casestoconsider}, we have to consider the case where $p\neq 2$, $u$ is even, $h=zu$ and $z$ is odd. Since $e=0$, we have $\beta=\delta=0$.  Also since $p\neq 2$ and $z$ is odd, $q$ and $\alpha$ are also odd. \\
Since  $0\leq d_u(G)-d_u(G_{a.})\leq 2$, by Lemma \ref{l:centmod2} it suffices to show that $d_u(G)-d_u(G_{a.})>0$.
Under the assumptions on $h$, $u$ and $z$, adapting the proof of of \cite[Proposition 3.3]{Lawther}, we obtain that $d_u(G_{a.})$ is attained for an element of $G_{a.}$ whose semisimple part $y$ of order   $v$ has centralizer satisfying
 $C_{G_{a.}}(y)^0=A_{\alpha-1}^v T_{v-1}$. Moreover $d_u(G_{a.})=d_q(A_{\alpha-1}^vT_{v-1})$.  \\
By the assumptions on $h$, $u$ and $z$ it follows  that $y$ is not the image of an element of ${\rm SL}_{h}(K)$ of order dividing $v$  under the canonical map ${\rm SL}_{h}(K)\rightarrow {\rm PSL}_{h}(K)$.  Hence  $d_u(G)-d_u(G_{a.})>0$   and $d_u(G)=d_u(G_{a.})+2$ as required. \\

Suppose now that $G$ is of type $B_\ell$.  By Table \ref{ta:casestoconsider}, we have to consider the case where $p\neq 2$, $u$ is even, $h=zu$ and $z$ is odd unless $u\equiv 2 \mod 4$ and $z\equiv 2 \mod 4$; moreover if $z$ is odd then either $u\equiv 2 \mod 4$ and $z\equiv u/2 \mod 4$, or $u \equiv 4 \mod 8$.  Since $e=0$, we have $\beta=\delta=0$ and $\alpha>0$.  Also since $p\neq 2$, $q$ is also odd. \\
Write $s=v/2$ and let $a$ and $b$ be nonnegative integers such that $\ell-1=as+b$ with $0\leq b\leq s-1$. Under the assumptions on $h$, $u$ and $z$ we have $a=\alpha-1$ and $b=s-1$. \\
Since  $0\leq d_u(G)-d_u(G_{a.})\leq 2$, by Lemma \ref{l:centmod2} it suffices to show that $d_u(G)-d_u(G_{a.})>0$. Under the assumptions on $h$, $u$ and $z$, adapting the proof  of \cite[Proposition 3.6]{Lawther}, we obtain that $d_u(G_{a.})$ is attained for an element of $G_{a.}$ whose semisimple part $y$ of order   $v$ has centralizer satisfying $$C_{G_{a.}}(y)^0=A_{\alpha-1}^{\frac{v}2-1}B_{\lceil \frac {\alpha-1}2 \rceil} D_{\lceil \frac{\alpha}2\rceil} T_{\frac{v}2-1}={\rm GL}_{\alpha}(K)^{\frac{v}2-1}{\rm SO}_{\alpha}(K){\rm SO}_{\alpha+1}(K).$$ Moreover $$d_u(G_{a.})=d_q(A_{\alpha-1}^{\frac{v}2-1}B_{\lceil \frac {\alpha-1}2 \rceil} D_{\lceil \frac{\alpha}2\rceil} T_{\frac{v}2-1}).$$  \\
By the assumptions on $h$, $u$ and $z$ it follows from Lemma \ref{l:soliftspin} that $y$ is not the image of an element of ${\rm Spin}_{h+1}(K)$ of order dividing $v$  under the canonical map ${\rm Spin}_{h+1}(K)\rightarrow {\rm SO}_{h+1}(K)$.  Hence  $d_u(G)-d_u(G_{a.})>0$   and $d_u(G)=d_u(G_{a.})+2$ as required. \\

Suppose now that $G$ is of type $C_\ell$. By Table \ref{ta:casestoconsider}, we have to consider the case where $p\neq 2$ and $u$ is even. 
Let  $y \in G_{[v]}$ be of order $v$ and such that $Z=C_{G}(y)^0$ is of minimal dimension.  We first show that $d_q(Z)-d_u(G_{a.})= 2\lceil z/2\rceil$. \\
 There are essentially three possibilities for the structure of $Z=C_{G}(y)^0$.  Indeed either $y$ has no eigenvalues equal to $1$ or $-1$ in which case $Z$ has no $C$ factors, or $y$ has some eigenvalues equal to $1$ or $-1$ but not both in which case $Z$ has one $C$ factor, or $y$ has some eigenvalues equal to $1$ and $-1$ in which case $Z$ has two $C$ factors. 
Using  Corollary \ref{c:lawther} and arguing in a similar way as in the proof of \cite[Proposition 3.4]{Lawther}  we can assume that either $Z= A_a^bA_{a-1}^{s-1-b}T_{s-1}$ where $a,b$ are nonnegative integers such that $\ell=a(s-1)+b$ and $0\leq b <s-1$, or $Z=A_a^{\lfloor \frac b2\rfloor}A_{a-1}^{s-1-\lfloor \frac b2\rfloor}C_{\lceil \frac a2\rceil}T_{s-1}$ where $a,b$ are nonnegative integers such that $2\ell=a(2s-1)+b$ and $0\leq b < 2s-1$, or $Z=A_a^{b-\epsilon_a(1-\sigma_b)}A_{a-1}^{s-1-b+\epsilon_a(1-\sigma_b)}C_{\lceil \frac a2\rceil}C_{\lceil \frac{a}2\rceil-\epsilon_a\sigma_b}T_{s-1}$ where  $a,b$ are nonnegative integers such that $\ell=as+b$ and $0\leq b <s$. \\
Suppose first that  $Z= A_a^bA_{a-1}^{s-1-b}T_{s-1}$ where $a,b$ are nonnegative integers such that $\ell=a(s-1)+b$ and $0\leq b <s-1$. Note that $v>2$ and so $s>1$.  Write $a=cq+d$ where $c,d$ are nonnegative integers such that $0\leq d<q$. It follows from Lemma \ref{l:lawther} that $$d_q(Z)=(2c+1)\ell-\frac{c(c+1)u}2+c(c+1)q.$$ An easy check yields  $h=zu+e= c(u-2q)+f$ where $f=d(v-2)+2b$ satisfies $0\leq f \leq u-2q-2$. In particular $c\geq z$.  Write $c=z+j$ where $j\geq 0$. 
 \cite[Theorem 1]{Lawther} which gives the value of $d_u(G_{a.})$ now yields
$$ d_q(Z)-d_u(G_{a.})=\frac{uj}2(j-1)+(z+j)q(z-j+1)+fj.$$  We claim that $z\geq j-1$. Suppose not. Then $c\geq 2z+2$ and so we get
\begin{eqnarray*}
0& = & 2\ell -2\ell \\
& = & (c(u-2q)+f)-(zu+e)\\
 &  \geq & (2z+2)(u-2q)+f-(zu+e)\\
 &  =& z(u-4q)+2(u-2q)+f-e\\
 & \geq & z(u-4q)+2(u-2q)-u+1\\
 & = & z(u-4q)+u-4q+1
\end{eqnarray*}
Since $v>2$, we have $u\geq 4q$ and so  we get $0\geq 1$, a contradiction. Therefore $z \geq j-1$ as claimed. It follows that $$ d_q(Z)-d_u(G_{a.})\geq 2\lceil z/2\rceil.$$  \\
Suppose now that   $Z=A_a^{\lfloor \frac b2\rfloor}A_{a-1}^{s-1-\lfloor \frac b2\rfloor}C_{\lceil \frac a2\rceil}T_{s-1}$ where $a,b$ are nonnegative integers such that $2\ell=a(2s-1)+b$ and $0\leq b < 2s-1$.  Write $a=cq+d$ where $c,d$ are nonnegative integers such that $0\leq d<q$. It follows from Lemma \ref{l:lawther} that $$d_q(Z)=(2c+1)\ell-\frac{c(c+1)u}2+\frac{c(c+1)q}2+\lceil c/2\rceil.$$ An easy check yields  $h=zu+e= c(u-q)+f$ where $f=d(v-1)+b$ satisfies $0\leq f \leq u-q-1$. In particular $c\geq z$.  Write $c=z+j$ where $j\geq 0$. 
 \cite[Theorem 1]{Lawther} which gives the value of $d_u(G_{a.})$ now yields
$$ d_q(Z)-d_u(G_{a.})=\frac{uj}2(j-1)+\frac{(z+j)q(z-j+1)}2+\frac{z+j+\epsilon_{z+j}}2+fj.$$  We claim that $z\geq j-1$. Suppose not. Then $c\geq 2z+2$ and so we get
\begin{eqnarray*}
0& = & 2\ell -2\ell \\
& = & (c(u-q)+f)-(zu+e)\\
 &  \geq & (2z+2)(u-q)+f-(zu+e)\\
 &  =& z(u-2q)+2(u-q)+f-e\\
 & \geq & z(u-2q)+2(u-q)-u+1\\
 & = & z(u-2q)+u-2q+1
\end{eqnarray*}
We have $u\geq 2q$ and so  we get $0\geq 1$, a contradiction. Therefore $z \geq j-1$ as claimed. It follows that $$ d_q(Z)-d_u(G_{a.})\geq 2\lceil z/2\rceil.$$  \\
Suppose finally that $Z=A_a^{b-\epsilon_a(1-\sigma_b)}A_{a-1}^{s-1-b+\epsilon_a(1-\sigma_b)}C_{\lceil \frac a2\rceil}C_{\lceil \frac{a}2\rceil-\epsilon_a\sigma_b}T_{s-1}$ where  $a,b$ are nonnegative integers such that $\ell=as+b$ and $0\leq b <s$. Writing $a=cq+d$ where $c,d$ are nonnegative integers such that $0\leq d<q$ one can show that $c=z$ and Lemma \ref{l:lawther} yields $d_q(Z)-d_u(G_{a.})=2\lceil z/2\rceil$ (see the proof of \cite[Proposition 3.4]{Lawther}) for more details. \\
We have in fact showed that if $G$ is  simply connected of type $C_\ell$, $p\neq 2$, $v$ is even and $y$ is a semisimple element of $G$ of order $v$ having  centralizer $Z$  of minimal dimension then $d_q(Z)=d_{qv}(G_{a.})+2\lceil z/2 \rceil$ where $z=\lfloor h/(qv)\rfloor$.  Since $$d_u(G)=\min_{y\in {G}_{[v]}} \ d_q(C_{G}(y)^0)$$ and   for any divisor $r$ of $v$ we have $\lfloor rh/(qv) \rfloor \geq \lfloor  h/qv\rfloor$  and, by \cite[Lemma 1.3]{Lawther}, $d_{qv/r}(G_{a.}) \geq d_{qv}(G_{a.})$, we deduce that $d_u(G)=d_q(Z)=d_u(G_{a.})+2\lceil z/2\rceil$. \\

Suppose finally  that $G$ is of type $D_\ell$.  By Table \ref{ta:casestoconsider}, we have to consider the case where $p\neq 2$, $z$ is odd, and   $u\equiv 2 \mod 4$  or $u \equiv 4 \mod 8$. Moreover if $u\equiv 4 \mod 8$ then $h=zu$. Finally if $u \equiv 2 \mod 4$ then $h=zu$ and $z\equiv u/2 \mod 4$, or $e=u-2\neq 0$ and $z\equiv 1 \mod 4$.  Since $v$ is even, $\beta$ is also even. Also since $p\neq 2$, $q$ is odd.  Finally under the assumptions on $h$, $u$ and $z$, $\alpha$ is odd.\\
Write $s=v/2$ and let $a$ and $b$ be nonnegative integers such that $\ell-1=as+b$ with $0\leq b\leq s-1$. We have $a=\alpha$ is odd and $b=\beta/2$. \\
Suppose first that $u=2$ so that $q=1$ and $v=2$.  Corollary \ref{c:lawther} and Lemma \ref{l:soliftspin}  yield that an element $y$ of $G$ of order $u$ with centralizer of minimal dimension  has  centralizer with connected component $D_{(z+3)/2}D_{(z-1)/2}$ and $d_u(G)=d_u(G_{a.})+4$.  \\
Suppose now that $u>2$. Then $0\leq d_u(G)-d_u(G_{a.})\leq 2$, and so by Lemma \ref{l:centmod2} it suffices to show that $d_u(G_{s.c.})-d_u(G_{a.})>0$. 
Under the assumptions on $h$, $u$ and $z$, adapting the proof of  \cite[Proposition 3.5]{Lawther}, we obtain that $d_u(G_{a.})$ is attained for an element of $G_{a.}$ whose semisimple part $y$ of order   $v$ has centralizer satisfying
 $$C_{G_{a.}}(s)^0=A_{\alpha}^\frac \beta 2A_{\alpha-1}^{\frac{v}2-1-\frac{\beta}2}D_{ \frac {\alpha+1}2 } D_{ \frac{\alpha+1}2} T_{\frac{v}2-1}={\rm GL}_{\alpha+1}^{\frac \beta 2}{\rm GL}_{\alpha}(K)^{\frac{v}2-1-\frac \beta 2}{\rm SO}_{\alpha+1}(K){\rm SO}_{\alpha+1}(K).$$ Moreover $$d_u(G_{a.})=d_q(A_{\alpha}^\frac \beta 2A_{\alpha-1}^{\frac{v}2-1-\frac{\beta}2}D_{ \frac {\alpha+1}2 } D_{ \frac{\alpha+1}2} T_{\frac{v}2-1}).$$  \\
By the assumptions on $h$, $u$ and $z$, it follows from Lemma \ref{l:soliftspin} that $y$ is not the image of an element of ${\rm Spin}_{h+2}(K)$ of order dividing $v$  under the canonical map ${\rm Spin}_{h+2}(K)\rightarrow {\rm SO}_{h+2}(K)$.  Hence  $d_u(G)-d_u(G_{a.})>0$   and $d_u(G_{s.c.})=d_u(G_{a.})+2$ as required. 
$\square$

\bigskip

Given a positive integer $u$, we can now  determine $d_u(G)$ for $G$ a simple algebraic group of neither simply connected nor adjoint type and  prove  Theorem \ref{t:duotsca}.\\

\noindent \textit{Proof of Theorem \ref{t:duotsca}.} Note that since $G$ is neither  simply connected nor adjoint, $G$ is of classical type. We first consider part (i).  If $d_u(G_{s.c.})=d_u(G_{a.})$ then Proposition \ref{p:ingsimple} yields $d_u(G)=d_u(G_{a.})$.  \\
We now consider part (ii)  and assume that $d_u(G_{s.c.})\neq d_u(G_{a.})$. In particular, by Theorem \ref{t:blawther}, $p\neq 2$ and $u$ is even (and so  $v$ is also even). By Proposition \ref{p:ingsimple} we have $d_u(G_{a.})\leq d_u(G)\leq d_u(G_{s.c.})$.  \\
(a) Suppose first that $G={\rm SL}_{\ell+1}(K)/C$ where  $C\leq Z({\rm SL}_{\ell+1}(K))$ is a central subgroup of $G_{s.c.}$. Note that $C$ is  finite and cyclic and write $C=\langle c \rangle$.   Since  $d_u(G_{s.c.})\neq d_u(G_{a.})$ Theorem \ref{t:blawther} yields $h=zu=zqv$ where $z$ is odd. \\
 Assume $(u,|C|)=1$.   Let $gC\in G$ be an element of order dividing $u$. Then $g^u \in C$ and say the eigenvalue of $g^u$ is $c^l$ for some $0\leq l \leq |C|-1$.  
 Since $(u,|C|)=1$, there is a positive integer $j$ such that $ju\equiv 1 \mod |C|$. Let $k$ be a positive integer such that $k\equiv -jl \mod |C|$ and set $g'=g\cdot {\rm diag}(c^{k}, \dots , c^{k})$. Then  $g'$ is an element of $G_{s.c.}$ of order dividing $u$, $g'C=gC$  and by Lemma \ref{l:dufactg} $\dim C_{G_{s.c.}}(g')=\dim C_G(g'C)$.  It follows that $d_u(G)=d_u(G_{s.c.})$.     \\
  Assume now that $(u,|C|)>1$. We claim that $d_u(G)=d_u(G_{s.c.})$ if $|C|$ is odd, otherwise $d_u(G)=d_u(G_{a.})$.  By Lemma \ref{l:centmod2}, Proposition \ref{p:ingsimple} and Theorem \ref{t:blawther}, $d_u(G) \equiv \ell \mod 2$,  $d_u(G_{a.})\leq d_u(G) \leq d_u(G_{s.c.})$ and $d_u(G_{s.c})-d_u(G_{a.})=2$. Hence either $d_u(G)=d_u(G_{a.})$ or $d_u(G)=d_u(G_{s.c.})$.     \\
  Suppose $d_u(G)=d_u(G_{a.})$. Let $gC$ be an element of $G$ of order dividing $u$ such that $\dim C_G(gC)=d_u(G_{a.})$.  Write $g=xy$ where $x$ and $y$ are respectively the unipotent and semisimple parts of $g$. As $g^u\in C$ and $(p,|C|)=1$,  the order of $x$ divides $q$ and the order of $y$ divides $v|C|$.  By Lemma \ref{l:dufactg}  $d_u(G_{a.})=\dim C_{G_{s.c.}}(g)$ and so  $y^v \neq 1$, as otherwise $g$ has order dividing $u$ and $d_u(G_{a})=d_u(G_{s.c.})$, a contradiction.\\
  Furthermore, $y$ must have $v$ distinct eigenvalues, each repeated $zq$ times. These $v$ distinct eigenvalues must be $\omega^{1+jk}, \omega^{-(1+jk)}$ where $0\leq j\leq (v-2)/2$, $k>1$ is a divisor   of $|C|$  and $\omega \in K$ is a primitive $kv$-th root of unity. Considering $1\neq y^v\in C$ we deduce that $\omega^v=\omega^{-v}$ and so $\omega^v$ has order $2$. Hence $2$ divides $|C|$. It follows that if $|C|$ is odd then $d_u(G)=d_u(G_{s.c.})$.  \\
  Suppose that $|C|$ is even. Consider $g$ a semisimple element of $G_{s.c.}$ having $v$ distinct  eigenvalues, each repeated $zq$ times, equal to   $\omega^{1+2j}, \omega^{-(1+2j)}$ where $0\leq j\leq (v-2)/2$, $\omega\in K$ is a primitive $2v$-th root of unity. Then $gC$ is an element of $G$ of order $v$ and $C_G(gC)^0= A_{\alpha-1}^vT_{v-1}$. It follows that $d_q(C_G(gC)^0)=d_u(G_{a.})$ and so $d_u(G)=d_u(G_{a.})$.\\
  (b) Suppose $G={\rm SO}_{2\ell}(K)$.   Using the notation of Lemma  \ref{l:dsc} and Theorem \ref{t:blawther}, we let $\omega \in K$ be  a primitive $v$-th root of $1$ and $y \in G$ be a semisimple element of order $v$ such that
  $$ y=\left\{\begin{array}{ll} (-1)^\ell \oplus (1)^\ell & \textrm{if} \ v=2 \\
  M_1^{\alpha}\oplus M_3\oplus(1)\oplus(-1) & \textrm{if} \ v>2 \ \textrm{and} \ \alpha \ \textrm{is odd}\\
  M_1^{\alpha}\oplus M_3\oplus(1)^2 & \textrm{if} \ v>2 \ \textrm{and} \ \alpha \ \textrm{is even}.\\
   \end{array} \right.$$
   Then  $$ C_G(y)^0=\left\{\begin{array}{ll} D_{\frac \ell2}D_{\frac \ell2} & \textrm{if} \ v=2\\
A_{\alpha}^{\frac \beta2}A_{\alpha-1}^{\frac v2 -1-\frac \beta2}D_{\frac{\alpha+1}2} D_{\frac{\alpha+1}2} & \textrm{if} \ v>2 \ \textrm{and} \ \alpha \ \textrm{is odd}\\
A_{\alpha}^{\frac \beta2}A_{\alpha-1}^{\frac v2 -1-\frac \beta2}D_{\frac{\alpha}2} D_{\frac{\alpha+2}2} & \textrm{if} \ v>2 \ \textrm{and} \ \alpha \ \textrm{is even}.
   \end{array} \right.$$
   An easy check yields that $d_q(C_G(y))=d_u(G_{a.})$ and so $d_u(G)=d_u(G_{a.})$. \\
 (c)  Finally suppose that $G={\rm HSpin}_{2\ell}(K)$ where $p\neq 2$ and $\ell$ is even.  
 Write $H=G_{s.c.}$ and let  $C=Z(H)=\langle c_1,c_2\rangle$, $C_1=\langle c_1\rangle$, $C_2=\langle c_2\rangle $  where $c_1,c_2$ are of order 2 and $H/C=G_{a.}$, $H/C_1=G$ and $H/C_2={\rm SO}_{2\ell}(K)$.\\
 Suppose first that $u=2$.  Let $hC_1$ be an element of $H/C_1$ of order 2 satisfying $$d_2(H/C_1)=\dim C_{H/C_1}(hC_1).$$  
 By Lemma \ref{l:dufactg} we have 
  \begin{equation}\label{e:spinsopso}
  \dim C_H(h)=\dim C_{H/C_1}(hC_1)=\dim C_{H/C_2}(hC_2)=\dim C_{H/C}(hC).
  \end{equation}
 Since $hC_1$ has order 2, either $h^2=1$ or $h^2=c_1$. Also note that $hC$ has order dividing 2. If $hC$ is trivial then $h\in C$ and  so by (\ref{e:spinsopso}) $\dim C_{H/C_1}(hC_1)= \dim H/C_1$, a contradiction. Hence $hC$ has order 2.  
 We claim that $h^2 \neq c_1$. Suppose not. Then $hC$ is an element of $H/C$ of order 2 and $hC_2$ is an element of $H/C_2$ of order 4. It follows that $hC_2$ has $\ell$ eigenvalues equal to $\omega$ and $\ell$ eigenvalues equal to $\omega^{-1}$ where $\omega \in K$ is a primitive fourth root of unity.  Hence $C_{H/C_2}(hC_2)=A_{\ell}$.  Using (\ref{e:spinsopso}), it follows that $d_2(G) > d_2(G_{s.c.})$, contradicting $d_2(G)\leq d_2(G_{s.c.})$. 
 Hence $h^2=1$ and so by (\ref{e:spinsopso}) $d_2(G)\geq d_2(G_{s.c.})$. Therefore  $d_2(G)= d_2(G_{s.c.})$.\\
 Suppose now that $u>2.$ Since   by Theorem \ref{t:blawther} $d_u(G_{s.c.})-d_u(G_{a.})=2$, in order to show that $d_u(G)=d_u(G_{s.c.})$ it is enough by Lemma \ref{l:centmod2} to prove that $d_u(G)\neq d_u(G_{a.})$.  Suppose for a contradiction that $d_u(G)=d_u(G_{a.})$.  Let  $hC_1$  be an element of $G=H/C_1$ of order $u$  such that  $$\dim C_{H/C_1}(hC_1)=d_u(G_{a.}).$$  
 By Lemma \ref{l:dufactg} we have 
  \begin{equation}\label{e:spinsopsobis}
  \dim C_H(h)=\dim C_{H/C_1}(hC_1)=\dim C_{H/C_2}(hC_2)=\dim C_{H/C}(hC).
  \end{equation}
Since $hC_1$ has order $u$, $h^u\in C_1$ and so $h^u=1$ or $h^u=c_1$.  Also $hC$ has order dividing $u$ whereas $hC_2$ has order dividing $u$ or  $2u$. 
Suppose $hC_2$ has order dividing $u$. Then $h^u\in C_1\cap C_2=1$ and so $h$ is an element of $G_{s.c.}$ of oder dividing $u$. Hence by  (\ref{e:spinsopsobis}), $d_u(G_{s.c.})=d_u(G_{a.})$, a contradiction.\\ 
 Suppose $hC_2$ has order dividing $2u$ but $(hC_2)^u\neq C_2$. Then $h^u=c_1$ and $(hC_2)^u=c_1C_2$. Let $h=xy$ be the Jordan decomposition of $h$ where $x$ is unipotent and $y$ is semisimple. Note that $(yC_2)^v=c_1C_2$. Without loss of generality, $yC_2$ is diagonal and $(yC_2)_{i,i}(yC_2)_{i+1,i+1}=1$ for every odd $i$ with $1\leq i \leq 2\ell-1$. Note that no  eigenvalue of $yC_2$ is equal to $1$ or $-1$. Let $d$ be a semisimple element of $H$ such  that $dC_2$ is diagonal, $(dC_2)_{i,i}(dC_2)_{i+1,i+1}=1$ for every odd $i$ with $1\leq i \leq 2\ell-1$ and $dC_2$ has  $\ell$ eigenvalues equal to $\omega$ and $\ell$ eigenvalues equal to $\omega^{-1}$ where $\omega\in K$ is a primitive $2v$-th root of unity. Then  $(dC_2)^v=c_1C_2$ and $(dyC_2)^v=C_2$. 
 Hence $dyC_2$ is an element of $H/C_2$ of order dividing $v$ and $C_{H/C_2}(dyC_2)\cong C_{H/C_2}(yC_2)$.  
  Since $\ell$ is even, it follows from Lemma \ref{l:soliftspin} that $d$ is a semisimple element of $H$  of order dividing $2v$ and so  $dy$ is a semisimple element of $H$ of order dividing $v$.   
  Now by Lemma \ref{l:dufactg} $$d_q(C_H(dy)^0)=d_q(C_{H/C_2}(dyC_2)^0)$$ and so we obtain 
  \begin{eqnarray*}
  d_q(C_H(dy)^0)& = & d_q(C_{H/C_2}(yC_2)^0)\\
  & = & d_u(G_{a.}).
  \end{eqnarray*}
  Hence $d_u({G_{s.c.}})=d_u(G_{a.})$, a contradiction.
$\square$

\section{Proof of Proposition \ref{p:dudecreasing}}\label{s:decreasing}

Given a simple algebraic group $G$ defined over an algebraically closed field $K$ of characteristic $p$, we now show that $d_u(G): \mathbb{N}\rightarrow \mathbb{N}$ is a decreasing function of $u$.
We need to show that for every positive integer $u$, we have $d_u(G)-d_{u+1}(G)\geq 0$. Now by Proposition \ref{p:ingsimple}, we have $d_u(G_{a.})\leq  d_u(G) \leq d_u(G_{s.c.})$ for every positive integer $u$. Hence the conclusion of the proposition will follow at once after we show that for every positive integer $u$, we have 
\begin{equation}\label{e:dasc}
d_u(G_{a.})-d_{u+1}(G_{s.c.})\geq 0.  \end{equation}
If $G$ is of exceptional type then the result follows at once from \cite[Theorem 1]{Lawther} and Theorem \ref{t:blawther}. We therefore assume that $G$ is of classical type. \\
Note that if $u=1$ then $d_u(G_{a.})=d_1(G_{a.})=\dim G \geq d_{u+1}(G_{s.c.})$ and so (\ref{e:dasc}) holds. \\
Suppose now that $u\geq h$. Since $u\geq h$ we have $d_u(G_{a.})=\ell$. Also as $u+1> h$, Theorem \ref{t:blawther} yields $d_{u+1}(G_{s.c.})=d_{u+1}(G_{a})=\ell$. Therefore   $d_u(G_{a.})=d_{u+1}(G_{s.c.})=\ell$ and so (\ref{e:dasc}) holds. \\
We can therefore assume that $2\leq u \leq h-1$. In particular $z\geq 1$. Write $h=z'(u+1)+e'$ where $z'$, $e'$ are nonnegative integers such that $0\leq e'\leq u.$	  An easy check gives $1\leq z'\leq z.$ Write $z'=z-j$ where $0\leq j \leq z-1$. Then $e'=e-z+j(u+1)$. Note that if $j=1$ then $z\geq 2$. Also $z-e\leq j(u+1)\leq u+z-e$.  We claim that either $z>2j-1$ or $j=1$. Suppose otherwise. Then as $z\geq 1$ we get $j\geq 2$ and $z\leq 2j-1$. Using the latter inequality and the fact that $j(u+1)\leq u+z-e$, we get $j\leq 1-e/(u-1)$ and so $j\leq 1$, a contradiction establishing the claim.\\
 
Suppose that $G$ is of type $A$. By Theorem \ref{t:blawther} we have $d_{u+1}(G_{s.c.})\leq d_{u+1}(G_{a.})+2$ and so  
\begin{equation}\label{e:daa} d_u(G_{a.})-d_{u+1}(G_{s.c.}) \geq d_u(G_{a.})-d_{u+1}(G_{a})-2.\end{equation}
Let $\mathbf{D}=  d_u(G_{a.})-d_{u+1}(G_{a})$. Note that Lemma \ref{l:centmod2} yields $d_{u}(G_{a.})\equiv d_{u+1}(G_{a.}) \mod 2$ and so $\mathbf{D}\equiv 0 \mod 2$. Using \cite[Theorem 1]{Lawther}, we have
\begin{eqnarray*}
\mathbf{D} & = & z^2u+e(2z+1)-1-((z-j)^2(u+1)+(e-z+j(u+1))(2(z-j)+1) -1)\\
& = & (u+1)j(j-1)+z(z-2j+1)+2je.
\end{eqnarray*}
Recalling that $z\geq 2$ if $j=1$, we note  that if $j=1$ then $\mathbf{D}>0$. Also if $z>2j-1$ then $\mathbf{D}>0$.    \\
Hence $\mathbf{D}>0$ in all cases.
 As $\mathbf{D}\equiv 0 \mod 2$ we deduce that $\mathbf{D}\geq 2$ and so (\ref{e:daa}) yields $d_u(G_{a.})-d_{u+1}(G_{s.c.})\geq 0$. Hence (\ref{e:dasc}) holds.\\

Suppose that $G$ is of type $C$. By Theorem \ref{t:blawther} we have $d_{u+1}(G_{s.c.})\leq d_{u+1}(G_{a.})+2\epsilon_u\lceil (z-j)/2 \rceil$ and so  
\begin{equation}\label{e:daac} 
d_u(G_{a.})-d_{u+1}(G_{s.c.}) \geq d_u(G_{a.})-d_{u+1}(G_{a})-2\epsilon_u\left \lceil \frac{z-j}2  \right\rceil.\end{equation}
Let $\mathbf{D}=  d_u(G_{a.})-d_{u+1}(G_{a})$. Note that Lemma \ref{l:centmod2} yields $d_{u}(G_{a.})\equiv d_{u+1}(G_{a.}) \mod 2$ and so $\mathbf{D}\equiv 0 \mod 2$. Using \cite[Theorem 1]{Lawther}, we have
\begin{eqnarray*}
\mathbf{D} & = & \frac{z^2u+e(2z+1)}2+\epsilon_u\left\lceil \frac z2\right\rceil-\left(\frac{ (z-j)^2(u+1)+(e-z+j(u+1))(2(z-j)+1)}2\right.\\ 
& & \left. +\epsilon_{u+1}\left\lceil \frac{z-j}2\right \rceil \right)\\
& = & \frac{ (u+1)j(j-1)+z(z-2j+1)+2je}2+\epsilon_u \left(\left\lceil \frac{z-j}2 \right\rceil +\left\lceil \frac z2 \right\rceil\right)-\left \lceil \frac{z-j}2 \right\rceil.
\end{eqnarray*}
 Recalling that $z\geq 2$ if $j=1$, we note  that if $j=1$ then $\mathbf{D}\geq 0$ and moreover $\mathbf{D}\geq (z^2+\epsilon_z)/2\geq z+\epsilon_z=2\lceil z/2\rceil $ for $u$ odd. Also if $z>2j-1$ then $\mathbf{D}\geq 0$ and moreover $\mathbf{D}\geq 2\lceil z/2\rceil$ for $u$ odd.  \\
In all cases we therefore have $\mathbf{D}\geq 0$ and moreover $\mathbf{D}\geq 2\lceil  z/2\rceil$ for $u$  odd.
 Now (\ref{e:daac}) yields $d_u(G_{a.})-d_{u+1}(G_{s.c.})\geq 0$. Hence (\ref{e:dasc}) holds.\\

Suppose that $G$ is of type $B$.  By Theorem \ref{t:blawther} we have $d_{u+1}(G_{s.c.})\leq d_{u+1}(G_{a.})+2\epsilon_u$ and so  
\begin{equation}\label{e:daab} 
d_u(G_{a.})-d_{u+1}(G_{s.c.}) \geq d_u(G_{a.})-d_{u+1}(G_{a})-2\epsilon_u.\end{equation}
Let $\mathbf{D}=  d_u(G_{a.})-d_{u+1}(G_{a})$. Note that Lemma \ref{l:centmod2} yields $d_{u}(G_{a.})\equiv d_{u+1}(G_{a.}) \mod 2$ and so $\mathbf{D}\equiv 0 \mod 2$. Using \cite[Theorem 1]{Lawther}, we have
\begin{eqnarray*}
\mathbf{D} & = & \frac{z^2u+e(2z+1)}2+\epsilon_u\left\lceil \frac z2\right\rceil-\left(\frac{ (z-j)^2(u+1)+(e-z+j(u+1))(2(z-j)+1)}2\right.\\ 
& & \left. +\epsilon_{u+1}\left\lceil \frac{z-j}2\right \rceil \right)\\
& = & \frac{ (u+1)j(j-1)+z(z-2j+1)+2je}2+\epsilon_u \left(\left\lceil \frac{z-j}2 \right\rceil +\left\lceil \frac z2 \right\rceil\right)-\left \lceil \frac{z-j}2 \right\rceil.
\end{eqnarray*}
Arguing as in the case where $G$ is of type $C$, we deduce that $\mathbf{D}\geq 0$ and moreover $\mathbf{D}\geq 2$ for $u$  odd. Now (\ref{e:daab}) yields $d_u(G_{a.})-d_{u+1}(G_{s.c.})\geq 0$.  Hence (\ref{e:dasc}) holds.\\

Suppose that $G$ is of type $D$.  Since $u\geq 2$, by Theorem \ref{t:blawther} we have $d_{u+1}(G_{s.c.})\leq d_{u+1}(G_{a.})+2\epsilon_u$ and so  
\begin{equation*} 
d_u(G_{a.})-d_{u+1}(G_{s.c.}) \geq d_u(G_{a.})-d_{u+1}(G_{a})-2\epsilon_u.\end{equation*}
Let $\mathbf{D}=  d_u(G_{a.})-d_{u+1}(G_{a})$. Note that Lemma \ref{l:centmod2} yields $d_{u}(G_{a.})\equiv d_{u+1}(G_{a.}) \mod 2$ and so $\mathbf{D}\equiv 0 \mod 2$. Using \cite[Theorem 1]{Lawther}, we have
\begin{eqnarray*}
\mathbf{D} 
& = & \frac{ (u+1)j(j-1)+z(z-2j+1)+2je}2+\epsilon_u \left(\left\lceil \frac{z-j}2 \right\rceil +\left\lceil \frac z2 \right\rceil\right)-\left \lceil \frac{z-j}2 \right\rceil+j+\epsilon_{z-j}-\epsilon_z.
\end{eqnarray*}
Arguing as in the case where $G$ is of type $B$, we deduce that $\mathbf{D}\geq 0$ and moreover  
$\mathbf{D}\geq 2$ for $u$ odd. We deduce that $d_u(G_{a.})-d_{u+1}(G_{s.c.})\geq 0$. Hence (\ref{e:dasc}) holds.
$\square$

\section{Proof of  Theorem \ref{t:classification}}\label{s:classification}

In this section we prove Theorem \ref{t:classification}. We proceed in two steps, first in the special case where $G$ is of simply connected type (see Proposition \ref{p:classificationsc}), and then we consider other types (see Proposition \ref{p:classificationa}).

\begin{prop}\label{p:classificationsc}
Let $G$ be a simple simply connected algebraic group  over an algebraic closed field  $K$ of prime characteristic $p$. The classification of the hyperbolic triples $(a,b,c)$ of integers for $G$ is as given in Theorem \ref{t:classification}. 
\end{prop}

\begin{proof}
As usual we let $h$ be the Coxeter number of $G$ and for a positive integer $u$, we let $z$ and $e$  be the nonnegative integers such that $h=zu+e$ and $0\leq e \leq u-1$. 
Also $d_u(G)$ denotes the minimal dimension of the centralizer of an element of $G$ of order dividing $u$. By Proposition \ref{p:fcu} $d_u(G)$ is the codimension of the subvariety $G_{[u]}$ of $G$ consisting of elements of order dividing $u$. Writing $j_u(G)=\dim G_{[u]}$, we have $d_u(G)=\dim G-j_u(G)$. \\
Let $(a,b,c)$ be a hyperbolic triple of integers. Let $\mathbf{S}_{(a,b,c)}=d_a(G)+d_b(G)+d_c(G)$ and $\mathbf{D}_{(a,b,c)}=\mathbf{S}_{(a,b,c)}-\dim G$. Saying that $(a,b,c)$ is reducible (respectively, rigid, nonrigid) for $G$ amounts to saying that  $\mathbf{D}_{(a,b,c)}$ is greater than (respectively equal to, less than) 0. \\
Recall the partial order we put on the set of hyperbolic triples of integers. Given two hyperbolic triples $(a,b,c)$ and $(a',b',c')$ of integers, we say  that $(a,b,c)\leq (a',b',c')$ if and only if $a\leq a'$, $b\leq b'$ and $c\leq c'$. Among all hyperbolic triples, exactly three are minimal: $(2,3,7)$, $(3,3,4)$ and $(2,4,5)$. 
 By Proposition \ref{p:dudecreasing}, if $(a,b,c)$ is nonrigid for $G$ then $(a',b',c')$ is nonrigid for $G$ for every $(a',b',c')\geq (a,b,c)$.\\
 For a nonnegative integer $p$, we let $\theta_p \in \{0,1\}$ be such that $\theta_p=1$ if $p$ is odd or $p=0$, otherwise $\theta_p=0$.\\
 Suppose first that $G$ is of exceptional type. By  Theorem \ref{t:blawther} and \cite[Theorem 1]{Lawther}, every hyperbolic triple $(a,b,c)$ of integers is nonrigid for $G$ unless $G$ is of type $G_2$ and $(a,b,c)\in\{(2,4,5),(2,5,5)\}$ in which case $(a,b,c)$ is rigid for $G$. \\
  Suppose now that $G$ is of classical type. In a first step we  show that if $\ell \geq 11$, $\ell \geq 10$, $\ell \geq 15$, or $\ell \geq 9$ accordingly respectively as $G$ is of type $A$, $B$, $C$ or $D$, then every hyperbolic triples of integers is nonrigid for $G$.\\
Suppose first  that $G$ is of type $A$. Note that $\dim G=h^2-1$.
By Theorem \ref{t:blawther} and \cite[Theorem 1]{Lawther}, we have 
\begin{eqnarray*}
d_u(G)&\leq& z^2u+e(2z+1)-1+2\theta_p(1-\epsilon_u)\\
& = & \frac{(h-e)(h+e)}u+e-1+2\theta_p(1-\epsilon_u). 
\end{eqnarray*} 
Let $g(e)= \frac{(h-e)(h+e)}u+e-1+2\theta_p(1-\epsilon_u).$ Then $g'(e)=\frac{u-2e}u$, and $g'(e)>0$ if and only if $e<u/2$. Hence $d_u(G)\leq F(u)$ where 
\begin{eqnarray*}
F(u) &  =& g(u/2)\\
& = & \frac{u^2-4u+4h^2}{4u}+2\theta_p(1-\epsilon_u).  
\end{eqnarray*}
In particular, for any hyperbolic triple $(a,b,c)$ of integers, we have $\mathbf{S}_{(a,b,c)}\leq F(a)+F(b)+F(c)$.\\
Suppose $(a,b,c)=(2,3,7)$. We have $\mathbf{S}_{(2,3,7)}\leq F(2)+F(3)+F(7)$. Now
$$F(2)=\frac{h^2-1}2+2\theta_p, \quad F(3)=\frac{4h^2-3}{12}, \quad F(7)=\frac{4h^2+21}{28}.$$
Hence $$\mathbf{S}_{(2,3,7)}\leq  \frac{41h^2}{42}+2\theta_p.$$
and $$\mathbf{D}_{(2,3,7)}\leq -\frac{h^2}{42}+2\theta_p+1.$$
Since   $-\frac{h^2}{42}+2\theta_p+1$ is negative for $h \geq 12$, it follows that $(2,3,7)$ is nonrigid for $G$ provided $h\geq 12$, that is  $\ell\geq 11$. Also every hyperbolic triple $(a',b',c')$ of integers with $(a',b',c')\geq (2,3,7)$ is nonrigid for $G$ with $\ell\geq 11$.\\ 
Suppose $(a,b,c)=(2,4,5)$. We have $\mathbf{S}_{(2,4,5)}\leq F(2)+F(4)+F(5)$. Now
$$F(2)=\frac{h^2-1}2+2\theta_p, \quad F(4)=\frac{h^2}{4}+2\theta_p, \quad F(5)=\frac{4h^2+5}{20}.$$
Hence $$\mathbf{S}_{(2,4,5)}\leq  \frac{19h^2}{20}+4\theta_p-\frac14.$$
and $$\mathbf{D}_{(2,4,5)}\leq -\frac{h^2}{20}+\frac{3}4+4\theta_p.$$
Since   $-\frac{h^2}{20}+4\theta_p+\frac 34$ is negative for $h \geq 10$, it follows that $(2,4,5)$ is nonrigid for $G$ provided $h\geq 10$, that is  $\ell\geq 9$. Also every hyperbolic triple $(a',b',c')$ of integers with $(a',b',c')\geq (2,4,5)$ is nonrigid for $G$ with $\ell\geq 9$.\\
Suppose $(a,b,c)=(3,3,4)$. We have $\mathbf{S}_{(3,3,4)}\leq F(3)+F(3)+F(4)$. Now
$$ F(3)=\frac{4h^2-3}{12}, \quad F(4)=\frac{h^2}{4}+2\theta_p.$$
Hence $$\mathbf{S}_{(3,3,4)}\leq  \frac{11h^2}{12}+2\theta_p-\frac12.$$
and $$\mathbf{D}_{(3,3,4)}\leq -\frac{h^2}{12}+\frac{1}2+2\theta_p.$$
Since   $-\frac{h^2}{12}+2\theta_p+\frac 12$ is negative for $h \geq 6$, it follows that $(3,3,4)$ is nonrigid for $G$ provided $h\geq 6$, that is  $\ell\geq 5$. Also every hyperbolic triple $(a',b',c')$ of integers with $(a',b',c')\geq (3,3,4)$ is nonrigid for $G$ with $\ell\geq 5$.\\

Suppose now  that $G$ is of type $B$. Note that $\dim G=h(h+1)/2$.
By Theorem \ref{t:blawther} and \cite[Theorem 1]{Lawther}, we have 
\begin{eqnarray*}
d_u(G)&\leq& \frac{z^2u+e(2z+1)}2+\frac{z+1}2+2\theta_p(1-\epsilon_u)\\
& = & \frac{(h-e)(h+e+1)}{2u}+\frac{e+1}2+2\theta_p(1-\epsilon_u). 
\end{eqnarray*} 
Let $g(e)= \frac{(h-e)(h+e+1)}{2u}+\frac{e+1}2+2\theta_p(1-\epsilon_u).$ Then $g'(e)=\frac{u-2e-1}{2u}$, and $g'(e)>0$ if and only if $e<(u-1)/2$. Hence $d_u(G)\leq F(u)$ where 
\begin{eqnarray*}
F(u) &  =& g((u-1)/2)\\
& = & \frac{u^2+2u+(2h+1)^2}{8u}+2\theta_p(1-\epsilon_u).  
\end{eqnarray*}
In particular, for any hyperbolic triple $(a,b,c)$ of integers, we have $\mathbf{S}_{(a,b,c)}\leq F(a)+F(b)+F(c)$.\\
Suppose $(a,b,c)=(2,3,7)$. We have $\mathbf{S}_{(2,3,7)}\leq F(2)+F(3)+F(7)$. Now
$$F(2)=\frac{4h^2+4h+9}{16}+2\theta_p, \quad F(3)=\frac{h^2+h+4}{6}, \quad F(7)=\frac{h^2+h+16}{14}.$$
Hence $$\mathbf{S}_{(2,3,7)}\leq  \frac{164h^2+164h+797}{336}+2\theta_p.$$
and $$\mathbf{D}_{(2,3,7)}\leq -\frac{4h^2+4h-797}{336}+2\theta_p.$$
Since   $-\frac{4h^2+4h-797}{336}+2\theta_p$ is negative for $h \geq 20$, it follows that $(2,3,7)$ is nonrigid for $G$ provided $h\geq 20$, that is  $\ell\geq 10$. Also every hyperbolic triple $(a',b',c')$ of integers with $(a',b',c')\geq (2,3,7)$ is nonrigid for $G$ with $\ell\geq 10$.\\ 
Suppose $(a,b,c)=(2,4,5)$. We have $\mathbf{S}_{(2,4,5)}\leq F(2)+F(4)+F(5)$. Now
$$F(2)=\frac{4h^2+4h+9}{16}+2\theta_p, \quad F(4)=\frac{4h^2+4h+25}{32}+2\theta_p, \quad F(5)=\frac{h^2+h+9}{10}.$$
Hence $$\mathbf{S}_{(2,4,5)}\leq  \frac{76h^2+76h+359}{160}+4\theta_p.$$
and $$\mathbf{D}_{(2,4,5)}\leq -\frac{4h^2+4h-359}{160}+4\theta_p.$$
Since   $ -\frac{4h^2+4h-359}{160}+4\theta_p$ is negative for $h \geq 16$, it follows that $(2,4,5)$ is nonrigid for $G$ provided $h\geq 16$, that is  $\ell\geq 8$. Also every hyperbolic triple $(a',b',c')$ of integers with $(a',b',c')\geq (2,4,5)$ is nonrigid for $G$ with $\ell\geq 8$.\\
Suppose $(a,b,c)=(3,3,4)$. We have $\mathbf{S}_{(3,3,4)}\leq F(3)+F(3)+F(4)$. Now
$$ F(3)=\frac{h^2+h+4}{6}, \quad F(4)=\frac{4h^2+4h+25}{32}+2\theta_p.$$
Hence $$\mathbf{S}_{(3,3,4)}\leq  \frac{44h^2+44h+203}{96}+2\theta_p.$$
and $$\mathbf{D}_{(3,3,4)}\leq -\frac{4h^2+4h-203}{96}+2\theta_p.$$
Since   $-\frac{4h^2+4h-203}{96}+2\theta_p$ is negative for $h \geq 10$, it follows that $(3,3,4)$ is nonrigid for $G$ provided $h\geq 10$, that is  $\ell\geq 5$. Also every hyperbolic triple $(a',b',c')$ of integers with $(a',b',c')\geq (3,3,4)$ is nonrigid for $G$ with $\ell\geq 5$.\\

Suppose now  that $G$ is of type $C$. Note that $\dim G=h(h+1)/2$.
By Theorem \ref{t:blawther} and \cite[Theorem 1]{Lawther}, we have 
\begin{eqnarray*}
d_u(G)&\leq& \left\{\begin{array}{ll} \frac{z^2u+e(2z+1)}2+z+1 & \textrm{if} \ u \ \textrm{is even}\\ 
\frac{z^2u+e(2z+1)}2+\frac{z+1}2 & \textrm{if} \ u \ \textrm{is odd}\\
 \end{array} \right.  \\
& = &  \left\{\begin{array}{ll}  \frac{(h-e)(h+e+2)}{2u}+\frac{e}2+1 & \textrm{if} \ u \ \textrm{is even}\\
 \frac{(h-e)(h+e+1)}{2u}+\frac{e+1}2& \textrm{if} \ u \ \textrm{is even}
 \end{array} \right.
\end{eqnarray*} 
Let $$g(e)= \left\{\begin{array}{ll}  \frac{(h-e)(h+e+2)}{2u}+\frac{e}2+1 & \textrm{if} \ u \ \textrm{is even}\\
 \frac{(h-e)(h+e+1)}{2u}+\frac{e+1}2& \textrm{if} \ u \ \textrm{is odd}
 \end{array} \right.
$$ Then $$g'(e)=\left\{\begin{array}{ll} \frac{u-2e-2}{2u}   & \textrm{if} \ u \ \textrm{is even}\\ 
\frac{u-2e-1}{2u}   & \textrm{if} \ u \ \textrm{is odd}
\end{array}\right.$$ and if $u$ is even then $g'(e)>0$ if and only if $e<(u-2)/2$, else $g'(e)>0$ if and only if $e<(u-1)/2$.  Hence $d_u(G)\leq F(u)$ where 
\begin{eqnarray*}
F(u) &  =&\left\{\begin{array}{ll} g((u-2)/2) & \textrm{if} \ u \ \textrm{is even} \\
g((u-1)/2) & \textrm{if} \ u \ \textrm{is odd} \end{array} \right. \\
& = & \left \{\begin{array}{ll}  \frac{u^2+4u+4(h+1)^2}{8u} & \textrm{if} \ u \ \textrm{is even}\\
\frac{u^2+2u+(2h+1)^2}{8u} & \textrm{if} \ u \ \textrm{is odd} \\
  \end{array}\right .
\end{eqnarray*}
In particular, for any hyperbolic triple $(a,b,c)$ of integers, we have $\mathbf{S}_{(a,b,c)}\leq F(a)+F(b)+F(c)$.\\
Suppose $(a,b,c)=(2,3,7)$. We have $\mathbf{S}_{(2,3,7)}\leq F(2)+F(3)+F(7)$. Now
$$F(2)=\frac{h^2+2h+4}{4}, \quad F(3)=\frac{h^2+h+4}{6}, \quad F(7)=\frac{h^2+h+16}{14}.$$
Hence $$\mathbf{S}_{(2,3,7)}\leq  \frac{41h^2+62h+236}{84}.$$
and $$\mathbf{D}_{(2,3,7)}\leq -\frac{h^2-20h-236}{84}.$$
Since   $-\frac{h^2-20h-236}{84}$ is negative for $h \geq 30$, it follows that $(2,3,7)$ is nonrigid for $G$ provided $h\geq 30$, that is  $\ell\geq 15$. Also every hyperbolic triple $(a',b',c')$ of integers with $(a',b',c')\geq (2,3,7)$ is nonrigid for $G$ with $\ell\geq 15$.\\ 
Suppose $(a,b,c)=(2,4,5)$. We have $\mathbf{S}_{(2,4,5)}\leq F(2)+F(4)+F(5)$. Now
$$F(2)=\frac{h^2+2h+4}{4}, \quad F(4)=\frac{h^2+2h+9}{8}, \quad F(5)=\frac{h^2+h+9}{10}.$$
Hence $$\mathbf{S}_{(2,4,5)}\leq  \frac{19h^2+34h+121}{40}$$
and $$\mathbf{D}_{(2,4,5)}\leq -\frac{h^2-14h-121}{40}.$$
Since   $ -\frac{h^2-14h-121}{40}$ is negative for $h \geq 22$, it follows that $(2,4,5)$ is nonrigid for $G$ provided $h\geq 22$, that is  $\ell\geq11$. Also every hyperbolic triple $(a',b',c')$ of integers with $(a',b',c')\geq (2,4,5)$ is nonrigid for $G$ with $\ell\geq 11$.\\
Suppose $(a,b,c)=(3,3,4)$. We have $\mathbf{S}_{(3,3,4)}\leq F(3)+F(3)+F(4)$. Now
$$ F(3)=\frac{h^2+h+4}{6}, \quad F(4)=\frac{h^2+2h+9}{8}.$$
Hence $$\mathbf{S}_{(3,3,4)}\leq  \frac{11h^2+14h+59}{24}.$$
and $$\mathbf{D}_{(3,3,4)}\leq -\frac{h^2-2h-59}{24}.$$
Since   $-\frac{h^2-2h-59}{24}$ is negative for $h \geq 10$, it follows that $(3,3,4)$ is nonrigid for $G$ provided $h\geq 10$, that is  $\ell\geq 5$. Also every hyperbolic triple $(a',b',c')$ of integers with $(a',b',c')\geq (3,3,4)$ is nonrigid for $G$ with $\ell\geq 5$.\\

Suppose finally that $G$ is of type $D$.   Note that $\dim G=(h+1)(h+2)/2$.
By Theorem \ref{t:blawther} and \cite[Theorem 1]{Lawther}, we have 
\begin{eqnarray*}
d_u(G)&\leq& \left\{\begin{array}{ll} \frac{z^2u+e(2z+1)}2+z+1+4\theta_p\sigma_{u-2}+2\theta_p(1-\sigma_{u-2}) & \textrm{if} \ u \ \textrm{is even}\\ 
\frac{z^2u+e(2z+1)}2+\frac{3(z+1)}2 & \textrm{if} \ u \ \textrm{is odd}\\
 \end{array} \right.  \\
& = &  \left\{\begin{array}{ll}  \frac{(h-e)(h+e+2)}{2u}+\frac{e}2+1+ 4\theta_p\sigma_{u-2}+2\theta_p(1-\sigma_{u-2}) & \textrm{if} \ u \ \textrm{is even}\\
 \frac{(h-e)(h+e+3)}{2u}+\frac{e+3}2& \textrm{if} \ u \ \textrm{is odd}
 \end{array} \right.
\end{eqnarray*} 
Let $$g(e)= \left\{\begin{array}{ll}  \frac{(h-e)(h+e+2)}{2u}+\frac{e}2+1+4\theta_p\sigma_{u-2}+2\theta_p(1-\sigma_{u-2}) & \textrm{if} \ u \ \textrm{is even}\\
 \frac{(h-e)(h+e+3)}{2u}+\frac{e+3}2& \textrm{if} \ u \ \textrm{is odd}
 \end{array} \right.
$$ Then $$g'(e)=\left\{\begin{array}{ll} \frac{u-2e-2}{2u}   & \textrm{if} \ u \ \textrm{is even}\\ 
\frac{u-2e-3}{2u}   & \textrm{if} \ u \ \textrm{is odd}
\end{array}\right.$$ and if $u$ is even then $g'(e)>0$ if and only if $e<(u-2)/2$, else $g'(e)>0$ if and only if $e<(u-3)/2$.  Hence $d_u(G)\leq F(u)$ where 
\begin{eqnarray*}
F(u) &  =&\left\{\begin{array}{ll} g((u-2)/2) & \textrm{if} \ u \ \textrm{is even} \\
g((u-3)/2) & \textrm{if} \ u \ \textrm{is odd} \end{array} \right. \\
& = & \left \{\begin{array}{ll}  \frac{u^2+4u+4(h+1)^2}{8u}+4\theta_p\sigma_{u-2}+2\theta_p(1-\sigma_{u-2}) & \textrm{if} \ u \ \textrm{is even}\\
\frac{u^2+6u+(2h+3)^2}{8u} & \textrm{if} \ u \ \textrm{is odd} \\
  \end{array}\right .
\end{eqnarray*}
In particular, for any hyperbolic triple $(a,b,c)$ of integers, we have $\mathbf{S}_{(a,b,c)}\leq F(a)+F(b)+F(c)$.\\
Suppose $(a,b,c)=(2,3,7)$. We have $\mathbf{S}_{(2,3,7)}\leq F(2)+F(3)+F(7)$. Now
$$F(2)=\frac{h^2+2h+4}{4}+4\theta_p, \quad F(3)=\frac{h^2+3h+9}{6}, \quad F(7)=\frac{h^2+3h+25}{14}.$$
Hence $$\mathbf{S}_{(2,3,7)}\leq  \frac{41h^2+102h+360}{84}+4\theta_p.$$
and $$\mathbf{D}_{(2,3,7)}\leq -\frac{h^2+24h-276}{84}+4\theta_p.$$
Since   $-\frac{h^2+24h-276}{84}+4\theta_p$ is negative for $h \geq 16$, it follows that $(2,3,7)$ is nonrigid for $G$ provided $h\geq 16$, that is  $\ell\geq 9$. Also every hyperbolic triple $(a',b',c')$ of integers with $(a',b',c')\geq (2,3,7)$ is nonrigid for $G$ with $\ell\geq 9$.\\ 
Suppose $(a,b,c)=(2,4,5)$. We have $\mathbf{S}_{(2,4,5)}\leq F(2)+F(4)+F(5)$. Now
$$F(2)=\frac{h^2+2h+4}{4}+4\theta_p, \quad F(4)=\frac{h^2+2h+9}{8}+2\theta_p, \quad F(5)=\frac{h^2+3h+16}{10}.$$
Hence $$\mathbf{S}_{(2,4,5)}\leq  \frac{19h^2+42h+149}{40}+6\theta_p$$
and $$\mathbf{D}_{(2,4,5)}\leq -\frac{h^2+18h-109}{40}+6\theta_p.$$
Since   $ -\frac{h^2+18h-109}{40}+6\theta_p$ is negative for $h \geq 12$, it follows that $(2,4,5)$ is nonrigid for $G$ provided $h\geq 12$, that is  $\ell\geq7$. Also every hyperbolic triple $(a',b',c')$ of integers with $(a',b',c')\geq (2,4,5)$ is nonrigid for $G$ with $\ell\geq 7$.\\
Suppose $(a,b,c)=(3,3,4)$. We have $\mathbf{S}_{(3,3,4)}\leq F(3)+F(3)+F(4)$. Now
$$ F(3)=\frac{h^2+3h+9}{6}, \quad F(4)=\frac{h^2+2h+9}{8}+2\theta_p.$$
Hence $$\mathbf{S}_{(3,3,4)}\leq  \frac{11h^2+30h+99}{24}+2\theta_p.$$
and $$\mathbf{D}_{(3,3,4)}\leq -\frac{h^2+6h-75}{24}+2\theta_p.$$
Since   $-\frac{h^2+6h-75}{24}+2\theta_p$ is negative for $h \geq 10$, it follows that $(3,3,4)$ is nonrigid for $G$ provided $h\geq 10$, that is  $\ell\geq 6$. Also every hyperbolic triple $(a',b',c')$ of integers with $(a',b',c')\geq (3,3,4)$ is nonrigid for $G$ with $\ell\geq 6$.\\

Finally to classify hyperbolic triples $(a,b,c)$ of integers for $G$ with $\ell \leq 10$, $\ell \leq 9$, $\ell \leq 14$, $\ell \leq 8$ accordingly respectively as $G$ is of type $A$, $B$, $C$, $D$  one can use Theorem \ref{t:blawther} and \cite[Theorem 1]{Lawther} to find $d_a(G)$, $d_b(G)$ and $d_c(G)$ and compute $\mathbf{D}_{(a,b,c)}$. Note that if $u>h$ then $d_u(G)=\ell$. 
\end{proof}

\begin{prop}\label{p:classificationa}
Let $G$ be a simple algebraic group  over an algebraic closed field $K$ of prime characteristic $p$. The classification of hyperbolic triples $(a,b,c)$ of integers for $G$ is as given in Theorem \ref{t:classification}. 
\end{prop}

\begin{proof}
 The case where $G$ is of simply connected type is treated in Proposition \ref{p:classificationsc}. We therefore suppose that $G$ is not of simply connected type. 
  We  denote by $G_{s.c.}$ (respectively, $G_{a.}$)  the simple algebraic group of simply connected (respectively, adjoint) type having the same Lie type and Lie rank as  $G$. \\
  For an integer $u$, let $d_u(G)$ denote the minimal dimension of the centralizer of an element of $G$ of order dividing $u$. By Proposition \ref{p:fcu} $d_u(G)$ is the codimension of the subvariety $G_{[u]}$ of $G$ consisting of elements of order dividing $u$. Writing $j_u(G)=\dim G_{[u]}$, we have $d_u(G)=\dim G-j_u(G)$. \\
Let $(a,b,c)$ be a hyperbolic triple of integers. Let $\mathbf{S}_{(a,b,c)}=d_a(G)+d_b(G)+d_c(G)$ and $\mathbf{D}_{(a,b,c)}=\mathbf{S}_{(a,b,c)}-\dim G$. Recall that saying that $(a,b,c)$ is reducible (respectively, rigid, nonrigid) for $G$ amounts to saying that  $\mathbf{D}_{(a,b,c)}$ is greater than (respectively equal to, less than) 0. \\
  By Proposition \ref{p:ingsimple}, given a positive integer $u$, we have $d_u(G)\leq d_u(G_{s.c.})$.
     It follows that every hyperbolic triple of integers which is nonrigid for $G_{s.c.}$ is nonrigid for $G$. In particular, by Proposition \ref{p:classificationsc}, we can now assume that $G$ is of classical type.    By the proof of Proposition \ref{p:classificationsc} if $G$ is such that $\ell \leq 11$, $\ell \leq 10$, $\ell \leq 15$, $\ell \leq 9$ accordingly respectively as $G$ is of type $A$, $B$, $C$, $D$ then every hyperbolic triple of integers is nonrigid for $G$. For $G$ of small rank, one can use  \cite[Theorem 1]{Lawther} and Theorems \ref{t:blawther} and   \ref{t:duotsca} to find $d_a(G)$, $d_b(G)$ and $d_c(G)$, compute $\mathbf{D}_{(a,b,c)}$ and classify a given hyperbolic triple $(a,b,c)$ of integers for $G$. Note that if $u> h$ then $d_u(G)=\ell$. 
    \end{proof}

\section{Proof of Proposition \ref{p:marionred}}\label{s:reducibility}

 We consider Proposition \ref{p:marionred}.  Let $G$ be a simple algebraic group over  an algebraically closed field $K$. The main ingredient in the proof of Proposition \ref{p:marionred} is the following result proved in \cite{Marionconj} combined with the classification given in Theorem \ref{t:classification} of the reducible and the rigid hyperbolic triples of integers for simple algebraic groups
  if $p$ is  a bad prime for $G$ or $G$ is of exceptional type. Recall that $p$ is said to be bad for $G$ if $G$ is of type $B_\ell$, $C_\ell$, $D_\ell$ and $p=2$, or of type $G_2$, $F_4$, $E_6$, $E_7$ and $p\in\{2,3\}$, or of type $E_8$ and $p \in \{2,3,5\}$. A prime $p$ is said to be good for $G$ if it is not bad for $G$.  Also in the statement below and from now on, by an irreducible subgroup of a classical group $G$, we mean a subgroup acting irreducibly on the natural module for $G$.

  \begin{prop}\label{p:reducibility} \cite[Proposition 2.1]{Marionconj}.
  Suppose that $G$ is of classical type and  $p$ is a good prime for $G$. If $g_1$, $g_2$, $g_3$ are elements of $G$ such that $g_1g_2g_3=1$ and $\langle g_1,g_2\rangle$ is an irreducible subgroup of $G$ then 
  $$\dim g_1^G+\dim g_2^G+\dim g_3^G\geq 2 \dim G.$$
  \end{prop}

  We can now prove Proposition \ref{p:marionred}.\\
  
 \noindent \textit{Proof of Proposition \ref{p:marionred}.}
  If $G$ is of classical type and $p$ is a good prime for $G$ then the result follows from  Proposition \ref{p:reducibility}. The remaining cases follow from Theorem \ref{t:classification} which shows that there are no reducible hyperbolic triples of integers for $G$ if $G$ is of symplectic or orthogonal type and $p= 2$, or if $G$ is of exceptional type.  
  $\hspace{50mm} \square$\\

 \section{Some tables}\label{s:tables}
 
  In this section we collect Tables \ref{t:asc}-\ref{ta:casestoconsider} of the paper. 
  Tables \ref{t:asc}-\ref{t:dsc} appear in \S \ref{s:upc} and Table \ref{ta:casestoconsider} appears in \S \ref{s:pc}.

\begin{table}[h]
\caption{A  semisimple element $y$ of $G=(A_{\ell})_{\rm{s.c}}$ of order $v$ giving an upper bound for $d_u(G)$}\label{t:asc}
\small{{
\begin{tabular}{| l | l | l | l|}
\hline
Case & Element $y$ & $C_{G_{s.c.}}(y)^0$ & $d_q(C_{G_{s.c.}}(y)^0)$\\ 
\hline
$\begin{array}{l} \epsilon_v=1,\\  \textrm{or} \  (\epsilon_v,\epsilon_\alpha)=(0,0)  \end{array}$& $ M_{1}^\alpha\oplus M_{2}\oplus (1)^{\epsilon_\beta}$& $A_\alpha^\beta A_{\alpha-1}^{v-\beta}T_{v-1}$ & $z^2u+e(2z+1)-1$\\
\cline{1-2}
$(\epsilon_v,\epsilon_{\alpha},\epsilon_{\beta})=(0,1,1)$  & $M_{1}^\alpha\oplus M_{2}\oplus (-1)$ &  & \\
\cline{1-2}
$\begin{array}{l}(\epsilon_v,\epsilon_{\alpha},\epsilon_{\beta})=(0,1,0), \\ \beta\geq 2\end{array}$ & $M_1^\alpha\oplus M_5\oplus (1)\oplus (-1)$&  &  \\
\hline 
$(\epsilon_v,\epsilon_{\alpha})=(0,1)$, $\beta=0$  &  $M_1^{\alpha-1}\oplus M_4\oplus (-1)$& $A_\alpha A_{\alpha-1}^{v-2} A_{\alpha-2}T_{v-1}$&  $\left\{\begin{array}{ll} z^2u+e(2z+1)-1&  \textrm{if} \  e>0\\ z^2u+e(2z+1)+1&  \textrm{if} \ e=0  \end{array}\right.$\\
\hline
\end{tabular}}}
\end{table}

\vspace{5mm}

\begin{table}[h]
\caption{A  semisimple element $y$ of $G=(C_{\ell})_{\rm{s.c}}$ of order $v$ giving an upper bound for $d_u(G)$}\label{t:csc}
\small{{
\begin{tabular}{| l | l | l | l|}
\hline
Case & Element $y$ & $C_G(y)^0$ & $d_q(C_G(y)^0)$\\ 
\hline
 $\epsilon_v=1$ & $M_{1}^\alpha\oplus M_{2}\oplus (1)^{\epsilon_\beta}$& $A_\alpha^{\lfloor\frac\beta2\rfloor} A_{\alpha-1}^{\frac{v-1}2-\lfloor\frac\beta2\rfloor}C_{\lceil \frac \alpha2\rceil}T_{\frac{v-1}2}$ & $\frac{(z^2u+e(2z+1))}2+\lceil \frac z2 \rceil \epsilon_u$\\
 \hline
 $(\epsilon_v,\epsilon_\alpha)=(0,0)$ & $ M_{1}^\alpha\oplus M_{2}$& $A_\alpha^{\frac\beta2} A_{\alpha-1}^{\frac{v}2-1-\frac\beta2}C_{\frac \alpha2}^2T_{\frac{v}2-1}$ & $\frac{(z^2u+e(2z+1))}2+2\lceil \frac z2 \rceil$\\
\cline{1-3}
$(\epsilon_v,\epsilon_\alpha)=(0,1)$, $\beta\geq 2$ & $ M_{1}^\alpha\oplus M_{5}\oplus(1)\oplus(-1)$&  $A_\alpha^{\frac\beta2-1} A_{\alpha-1}^{\frac{v}2-\frac\beta2}C_{\frac{\alpha+1}2}^2T_{\frac{v}2-1}$ & \\
\cline{1-3}
$(\epsilon_v,\epsilon_\alpha)=(0,1)$, $\beta=0$ & $ M_{1}^{\alpha-1}\oplus M_{4}\oplus(-1)$&  $A_{\alpha-1}^{\frac{v}2-1}C_{\frac{\alpha+1}2}C_{\frac{\alpha-1}2}T_{\frac{v}2-1}$ & \\
\hline
\end{tabular}}}
\end{table}

\begin{sidewaystable}[h]
\caption{A  semisimple element $y$ of $G=(B_{\ell})_{\rm{s.c}}$ of order $v$ giving an upper bound for $d_u(G)$}\label{t:bsc}
\footnotesize{{
\begin{tabular}{| l | l | l | l|}
\hline
Case & Element $\overline{y}$ & $C_{{\rm SO}_{h+1}(K)}(\overline{y})^0$ & $d_q(C_G(y)^0)$\\ 
\hline
 (a) & $ M_{1}^\alpha\oplus M_{2}\oplus (1)^{1+\epsilon_\beta}$& $A_\alpha^{\lfloor\frac\beta2\rfloor} A_{\alpha-1}^{\frac{v-1}2-\lfloor\frac\beta2\rfloor}B_{\lceil \frac \alpha2\rceil}T_{\frac{v-1}2}$ & $\frac12(z^2u+e(2z+1))+\lceil \frac z2 \rceil \epsilon_u$\\
 \hline
 (b) & $\left\{\begin{array}{ll}(1)^{\ell+1}\oplus(-1)^{\ell} & \textrm{if} \ \ell \equiv 0 \ (4) \\
 (1)^{\ell+2}\oplus(-1)^{\ell-1} & \textrm{if} \ \ell \equiv 1 \ (4)\\
 (1)^{\ell-1}\oplus(-1)^{\ell+2} & \textrm{if} \ \ell \equiv 2 \ (4)\\
 (1)^{\ell}\oplus(-1)^{\ell+1} & \textrm{if} \ \ell \equiv 3 \  (4)
 \end{array}\right.$  &  $\left\{\begin{array}{ll}B_{\frac{\ell}2}D_{\frac{\ell}2} & \textrm{if} \ \ell \equiv 0 \ (4)\\ 
 B_{\frac{\ell+1}2}D_{\frac{\ell-1}2} & \textrm{if} \ \ell \equiv 1 \ (4)\\ 
B_{\frac{\ell-2}2}D_{\frac{\ell+2}2} & \textrm{if} \ \ell \equiv 2 \ (4)\\ 
B_{\frac{\ell-1}2}D_{\frac{\ell+1}2} & \textrm{if} \ \ell \equiv 3 \ (4)
 \end{array}\right.$& $\left \{\begin{array}{ll}\\
 \frac{(z^2u+e(2z+1))}2+2 & \textrm{if} \ \ell \equiv 1,2 \ (4)\\
 & \textrm{and} \ e=0\\
 \frac{(z^2u+e(2z+1))}2 & \textrm{otherwise} 
  \end{array}\right.$ \\
 \hline 
(c) & $M_{1}^\alpha\oplus M_{2}\oplus(-1) $& $A_\alpha^{\frac\beta2} A_{\alpha-1}^{\frac{v}2-1-\frac\beta2}B_{\frac{\alpha-1}2}D_{\frac {\alpha+1}2}T_{\frac{v}2-1}$ & $\frac12(z^2u+e(2z+1))$\\
\cline{1-3}
(d) & $ M_{1}^\alpha\oplus M_{5}\oplus(1)^2\oplus(-1)$&  $A_\alpha^{\frac\beta2-1} A_{\alpha-1}^{\frac{v}2-\frac\beta2}B_{\frac{\alpha+1}2}D_{\frac {\alpha+1}2}T_{\frac{v}2-1}$ & \\
\cline{1-2}
(e) & $ M_{1}^{\alpha}\oplus M_{6}\oplus(1)^2\oplus(-1)$&   & \\
\hline
(f) & $  M_{1}^{\alpha-1}\oplus M_{7}\oplus(1)^3\oplus(-1)^2 $&  $A_{\alpha-1}^{\frac v2-2} A_{\alpha-2}B_{\frac{\alpha+1}2}D_{\frac {\alpha+1}2}T_{\frac{v}2-1}$ & $\left \{\begin{array}{ll}\\
 \frac{(z^2u+e(2z+1))}2+2 & \textrm{if} \ e=0\\
 \frac{(z^2u+e(2z+1))}2 & \textrm{otherwise} 
  \end{array}\right.$ \\
\hline
(g) & $M_1^\alpha\oplus M_2 \oplus (1)$ & $A_\alpha^{\frac\beta2} A_{\alpha-1}^{\frac{v}2-1-\frac\beta2}B_{\frac{\alpha}2}D_{\frac {\alpha}2}T_{\frac{v}2-1}$ & $\frac12(z^2u+e(2z+1))$\\
\cline{1-2}
(h) & $M_1^\alpha\oplus M_6\oplus (\omega^{\frac v2-1})\oplus (\omega^{-\left(\frac v2-1\right)})  \oplus (1)$ & &\\
\cline{1-3}
(i) & $ M_{1}^\alpha\oplus M_{5}\oplus(1)\oplus(-1)^2$&  $A_\alpha^{\frac\beta2-1} A_{\alpha-1}^{\frac{v}2-\frac\beta2}B_{\frac{\alpha}2}D_{\frac {\alpha+2}2}T_{\frac{v}2-1}$ & \\
\cline{1-2}
(j) & $ M_{1}^\alpha\oplus M_{6}\oplus(1)\oplus(-1)^2$& &\\
\hline
(k) & $ M_{1}^{\alpha-1}\oplus M_{9}\oplus(1)^2\oplus(-1)^3 $&  $A_{\alpha-1}^{\frac v2-2} A_{\alpha-2}B_{\frac{\alpha}2}D_{\frac {\alpha+2}2}T_{\frac{v}2-1}$ & $\left \{\begin{array}{ll}\\
 \frac{(z^2u+e(2z+1))}2+2 & \textrm{if} \ e=0\\
 \frac{(z^2u+e(2z+1))}2 & \textrm{otherwise} 
  \end{array}\right.$ \\
  \hline
\end{tabular}}}
\end{sidewaystable}

\begin{sidewaystable}[h]
\caption{A  semisimple element $y$ of $G=(D_{\ell})_{\rm{s.c}}$ of order $v$ giving an upper bound for $d_u(G)$}\label{t:dsc}
\footnotesize{{
\begin{tabular}{| l | l | l | l|}
\hline
Case & Element $\overline{y}$ & $C_{{\rm SO}_{h+1}(K)}(\overline{y})^0$ & $d_q(C_G(y)^0)$\\ 
\hline
 (a) & $ M_{1}^\alpha\oplus M_{3}\oplus (1)^{2-\epsilon_\beta}$& $A_\alpha^{\lceil\frac\beta2\rceil} A_{\alpha-1}^{\frac{v-1}2-\lceil\frac\beta2\rceil}D_{\lceil \frac {\alpha+1}2\rceil}T_{\frac{v-1}2}$ & $\frac12(z^2u+e(2z+1))+\lceil \frac z2 \rceil \epsilon_u+z+1-\epsilon_z$\\
 \hline
 (b) & $\left\{\begin{array}{ll}(1)^{\ell}\oplus(-1)^{\ell} & \textrm{if} \ \ell \equiv 0 \ (4) \\
 (1)^{\ell+1}\oplus(-1)^{\ell-1} & \textrm{if} \ \ell \equiv 1 \ (4)\\
 (1)^{\ell-2}\oplus(-1)^{\ell+2} & \textrm{if} \ \ell \equiv 2 \ (4)\\
 (1)^{\ell-1}\oplus(-1)^{\ell+1} & \textrm{if} \ \ell \equiv 3 \  (4)
 \end{array}\right.$  &  $\left\{\begin{array}{ll}D_{\frac{\ell}2}D_{\frac{\ell}2} & \textrm{if} \ \ell \equiv 0 \ (4)\\ 
 D_{\frac{\ell+1}2}D_{\frac{\ell-1}2} & \textrm{if} \ \epsilon_\ell =1\\ 
D_{\frac{\ell-2}2}D_{\frac{\ell+2}2} & \textrm{if} \ \ell \equiv 2 \ (4)\\ 
 \end{array}\right.$& $\left \{\begin{array}{ll}\\
 \frac{z^2u+e(2z+1)}2+z+1-\epsilon_z+4 & \textrm{if} \ \ell \equiv 2 \ (4)\\ & \textrm{and}\  u =2\\
 \frac{z^2u+e(2z+1)}2+z+1-\epsilon_z+2 & \textrm{if} \ \ell \equiv 2 \ (4)\  \textrm{and}\\ &    e=u-v\neq 0\\
 & \textrm{or}     \ e=0\neq u-v\\
 \frac{z^2u+e(2z+1)}2+z+1-\epsilon_z & \textrm{otherwise} 
  \end{array}\right.$ \\
 \hline 
(c) & $M_{1}^\alpha\oplus M_{3}\oplus(1)\oplus(-1) $& $A_\alpha^{\frac\beta2} A_{\alpha-1}^{\frac{v}2-1-\frac\beta2}D_{\frac{\alpha+1}2}D_{\frac {\alpha+1}2}T_{\frac{v}2-1}$ & $\frac{z^2u+e(2z+1)}2+z+1-\epsilon_z$\\
\cline{1-2}
(d) & $ M_{1}^\alpha\oplus M_{5}\oplus (\omega^{\frac v2-2})\oplus(\omega^{-(\frac v2-2)}) \oplus(1)\oplus(-1)$& &\\
\cline{1-2}
(e) & $ M_{1}^\alpha\oplus M_{5}\oplus (\omega^{\frac v2-1})\oplus(\omega^{-(\frac v2-1)}) \oplus(1)\oplus(-1)$& &\\
\cline{1-2}
(f) & $ M_{1}^\alpha\oplus M_{6}\oplus (\omega^{\frac v2-2})\oplus(\omega^{-(\frac v2-2)}) \oplus(1)\oplus(-1)$& &\\
& if $\beta\neq v-4$& & \\
& $ M_{1}^\alpha\oplus M_{8} \oplus(1)\oplus(-1)$& &\\
& if $\beta=v-4$& & \\
\cline{1-2}
(g) & $ M_{1}^\alpha\oplus M_{6}\oplus (\omega^{\frac v2-1})\oplus(\omega^{-(\frac v2-1)}) \oplus(1)\oplus(-1)$& &\\
\hline
(h) & $M_1^\alpha \oplus M_9 \oplus (1) \oplus (-1)^3 $ & $A_\alpha^{\frac v2-2} A_{\alpha-1} D_{\frac{\alpha+1}2} D_{\frac{\alpha+3}2}T_{\frac v2-1}$ & $\left\{ \begin{array}{ll} \frac{z^2u+e(2z+1)}2 +z+1-\epsilon_z +2 & \textrm{if} \ e=u-2\\
\frac{z^2u+e(2z+1)}2 +z+1-\epsilon_z  & \textrm{otherwise}
  \end{array}\right.$  \\
\hline
(i) & $M_1^{\alpha-1} \oplus M_7 \oplus (\omega)^2 \oplus (\omega^{-1})^2 \oplus (1)^2 $ &  $A_{\alpha}A_{\alpha-1}^{\frac v2-2} D_{\frac{\alpha-1}2} D_{\frac{\alpha+1}2}T_{\frac v2-1}$&  $\left\{ \begin{array}{ll} \frac{z^2u+e(2z+1)}2 +z+1-\epsilon_z +2 & \textrm{if} \ e=0\\
\frac{z^2u+e(2z+1)}2 +z+1-\epsilon_z  & \textrm{otherwise}
  \end{array}\right.$ \\
\hline
(j) & $M_1^{\alpha -1}\oplus M_9 \oplus (\omega)\oplus (\omega^{-1}) \oplus (1)^2 \oplus(-1)^2 $ & $A_\alpha A_{\alpha-1}^{\frac v2-3}A_{\alpha-2} D_{\frac{\alpha+1}2}^2 T_{\frac v2-1}$ &   $\left\{ \begin{array}{ll} \frac{z^2u+e(2z+1)}2 +z+1-\epsilon_z +2 & \textrm{if} \ e=0\\
\frac{z^2u+e(2z+1)}2 +z+1-\epsilon_z  & \textrm{otherwise}
  \end{array}\right.$  \\
\hline
(k) & $M_1^\alpha\oplus M_3 \oplus (1)^2$ & $A_\alpha^{\frac \beta2}A_{\alpha-1}^{\frac v2-1-\frac \beta2}D_{\frac \alpha2}D_{\frac{\alpha+2}2}T_{\frac v2-1}$ & $\frac{z^2u+e(2z+1)}2+z+1-\epsilon_z$ \\
\cline{1-2} 
(l)  & $ M_{1}^\alpha\oplus M_{5}\oplus (\omega^{\frac v2-1})\oplus(\omega^{-(\frac v2-1)}) \oplus(1)^2$ & & \\
\cline{1-2}
(m) & $ M_{1}^\alpha\oplus M_{6} \oplus (\omega^{\frac v2-2})\oplus (\omega^{-(\frac v2-2)})\oplus(1)^2$&  &\\
& if $\beta\neq v-4$& & \\
& $ M_{1}^\alpha\oplus M_{8} \oplus(1)^2$&  &\\
& if $\beta=v-4$& & \\
\cline{1-3}
(n) & $ M_{1}^\alpha\oplus M_{5}\oplus (\omega^{\frac v2-2})\oplus(\omega^{-(\frac v2-2)}) \oplus(1)^2$&  $A_\alpha^{\frac \beta2}A_{\alpha-1}^{\frac v2-1-\frac \beta2}D_{\frac \alpha2}D_{\frac{\alpha+2}2}T_{\frac v2-1}$&\\
& if $\beta\neq v-2$&  if $\beta\neq v-2$ & \\
& $ M_{1}^\alpha\oplus M_{9} \oplus(1)^2\oplus(-1)^2$& $A_\alpha^{\frac v2-2}A_{\alpha-1}D_{\frac {\alpha+2}2}D_{\frac{\alpha+2}2}T_{\frac v2-1}$ &\\
& if $\beta=v-2$& if $\beta=v-2$& \\
\cline{1-3}
(o) & $ M_{1}^\alpha\oplus M_{6}\oplus (\omega^{\frac v2-1})\oplus(\omega^{-(\frac v2-1)}) \oplus(1)^2$&  $A_\alpha^{\frac \beta2}A_{\alpha-1}^{\frac v2-1-\frac \beta2}D_{\frac \alpha2}D_{\frac{\alpha+2}2}T_{\frac v2-1}$&\\
& if $\beta\neq v-2$&  if $\beta\neq v-2$ & \\
& $ M_{1}^\alpha\oplus M_{9} \oplus(1)^2\oplus(-1)^2$& $A_\alpha^{\frac v2-2}A_{\alpha-1}D_{\frac {\alpha+2}2}D_{\frac{\alpha+2}2}T_{\frac v2-1}$ &\\
& if $\beta=v-2$& if $\beta=v-2$& \\
  \cline{1-3}
  (p) & $M_1^{\alpha}\oplus(-1)^2$ & $A_{{\alpha-1}}^{\frac v2-1}D_{\frac \alpha2}D_{\frac{\alpha+2}2}T_{\frac v2-1}$ & \\
  \hline
\end{tabular}}}
\end{sidewaystable}

\begin{table}
\begin{tabular}{|l|l|l|}
\hline
$G$ &   Cases & $d_u(G_{s.c})-d_u(G_{a.})$ \\
\hline
$A_\ell$ & $p\neq 2$, $u$ even, $h=zu$ and $z$ odd & $\leq 2$\\
\hline
$C_\ell$ & $p\neq 2$ and $u$ even & $\leq 2\left\lceil \frac z2\right\rceil$\\
\hline
$B_{\ell}$ & $p\neq 2$,  $u\equiv 2 \mod 4$, $h=zu$, and $z\equiv u/2 \mod 4$ or $z\equiv 2 \mod 4$ & $\leq 2$\\
\cline{2-2}
&  $p\neq 2$, $u \equiv 4 \mod 8$, $h=zu$ and $z$  odd & \\
\hline 
$D_\ell$ & $p\neq 2$, $u=2$, $h=zu$ and $z \equiv u/2 \mod 4$ & $\leq 4$\\
\cline{2-3}
&  $p\neq 2$, $u \equiv 2 \mod 4$, $u>2$, $h=zu$ and $z \equiv u/2 \mod 4$ & $\leq2$ \\
\cline{2-2}
&  $p\neq 2$, $u\equiv 2 \mod 4$, $e=u-2\neq 0$ and 
 $z \equiv 1 \mod 4$& \\
\cline{2-2}
& $p\neq 2$, $u \equiv 4 \mod 8$, $z$ is odd,  and  $h=zu$  & \\
\hline
\end{tabular}
\caption{The possible exceptions to $d_u(G_{s.c.})=d_u(G_{a.})$ for $G$ of classical type}\label{ta:casestoconsider}
\end{table}


\begin{thebibliography}{99}

\bibitem{Cohen} A. Cohen, R. Griess. On finite simple subgroups of the complex Lie group
of type $E_8$.   The Arcata Conference on Representations of Finite Groups
(Arcata, Calif., 1986), 367--405, Proc. Sympos. Pure Math. \bf{47}, Part 2.
Amer. Math. Soc. Providence, RI, 1987.



\bibitem{Conder} M. Conder. Hurwitz groups: a brief survey. Bull. Amer. Math. Soc. \textbf{23} (1990), 359--370. 


\bibitem{Humphreys} J. Humphreys.  Conjugacy classes in semisimple algebraic groups.  Mathematical Surveys and Monographs \textbf{43}. American Mathematical Society, Providence, RI, 1995.

\bibitem{Guralnick} R. Guralnick. Intersections of conjugacy classes and subgroups of algebraic groups.  Proc. Amer. Math. Soc. \textbf{135} (2007), 689--693.

\bibitem{JLM} S. Jambor, A. Litterick, C. Marion. On finite simple images of triangle groups. Submitted for publication. 

\bibitem{LLM1} M. Larsen, A. Lubotzky, C. Marion. Deformation theory and finite simple quotients of triangle groups I. J. Eur. Math. Soc. \textbf{16} (2014), 1349--1375.




\bibitem{Lawther} R. Lawther. Elements of specified order in simple algebraic groups. Trans. Amer. Math. Soc. \textbf{357} (2005), 221--245.  



\bibitem{Lusztig} G. Lusztig. On the finiteness of the number of unipotent classes. Invent. Math. \textbf{34} (1976), 201--213. 

\bibitem{MaTe} G. Malle, D. Testerman. Linear algebraic groups and finite groups of Lie type.  Cambridge Studies in Advanced Mathematics \textbf{133}. Cambridge University Press, Cambridge, 2011.





\bibitem{Marionconj} C. Marion. On triangle generation of finite groups of Lie type. J. Group Theory \textbf{13} (2010), 619--648. 







\bibitem{Strambach} K. Strambach, H. V\"{o}lklein. On linearly rigid tuples. J. Reine. Angew. Math. \textbf{510} (1999), 57--62. 

\end{thebibliography}
\end{document}